\documentclass[11pt,reqno]{amsart}
\usepackage{amssymb,dsfont,enumerate}
\usepackage[pagebackref]{hyperref}
\usepackage[portrait,a4paper,margin=4cm]{geometry}

\usepackage{mathtools}
\mathtoolsset{showonlyrefs}

\usepackage[american]{babel}
\usepackage{amsmath}
\usepackage{mathtools}
\usepackage{tikz}
\usepackage{subfigure}
\usepackage{color}
\usepackage{mathrsfs}
\mathtoolsset{showonlyrefs,showmanualtags}

\numberwithin{equation}{section}
\numberwithin{figure}{section}

\title{Nonlinear recombinations and generalized random transpositions}
\author{Pietro Caputo}
\address{Department of Mathematics and Physics, Roma Tre University, Largo San Murialdo 1, 00146 Roma, Italy.}
\email{pietro.caputo@uniroma3.it}
\author{Daniel Parisi}
\address{Department of Mathematics and Physics, Roma Tre University, Largo San Murialdo 1, 00146 Roma, Italy.}
\email{daniel.parisi@uniroma3.it}

\definecolor{darkblue}{rgb}{0.1,0.1,0.7}

\definecolor{darkred}{rgb}{0.7,0.1,0.1}

\newcommand{\tv}{{\rm TV}}	
	
\DeclareMathSymbol{\leqslant}{\mathalpha}{AMSa}{"36} % nicer `smaller or equal' 
\DeclareMathSymbol{\geqslant}{\mathalpha}{AMSa}{"3E} % nicer `larger or equal' 
\DeclareMathSymbol{\eset}{\mathalpha}{AMSb}{"3F}     % nicer `emptyset' 
\renewcommand{\leq}{\;\leqslant\;}                   % redef. of < or = 
\renewcommand{\geq}{\;\geqslant\;}                   % redef. of > or = 
\newcommand{\la}{\label} 
 
\newcommand{\be}{\begin{equation}}

\newcommand{\Om}{\Omega}

\def\1{\ifmmode {1\hskip -3pt \rm{I}} \else {\hbox {$1\hskip -3pt \rm{I}$}}\fi}

\newcommand{\ind}{\mathbf{1}}

\newcommand{\e}{\varepsilon}

\newcommand{\IND}{{\bf 1}}

\newcommand{\rhoN}{\rho^{\pi}}

\newcommand{\si}{\sigma} 
\newcommand{\ent}{{\rm Ent} } 

\newcommand{\tc}{\, |\, }

\newcommand{\scalar}[2]{\langle #1 , #2\rangle}

\newtheorem{theorem}{Theorem}[section]
\newtheorem{proposition}[theorem]{Proposition}
\newtheorem{lemma}[theorem]{Lemma}
\newtheorem{corollary}[theorem]{Corollary}
\newtheorem{remark}[theorem]{Remark}
\newtheorem{definition}[theorem]{Definition}

\newcommand{\cA}{\ensuremath{\mathcal A}} 
\newcommand{\cB}{\ensuremath{\mathcal B}}

\newcommand{\cE}{\ensuremath{\mathcal E}}

\newcommand{\cL}{\ensuremath{\mathcal L}}

\newcommand{\cP}{\ensuremath{\mathcal P}}

\newcommand{\cS}{\ensuremath{\mathcal S}} 
\newcommand{\cT}{\ensuremath{\mathcal T}} 
\newcommand{\cU}{\ensuremath{\mathcal U}} 
\newcommand{\cV}{\ensuremath{\mathcal V}}

\newcommand{\bbC}{{\ensuremath{\mathbb C}} } 
 
\newcommand{\bbE}{{\ensuremath{\mathbb E}} }

\newcommand{\bbN}{{\ensuremath{\mathbb N}} } 
 
\newcommand{\bbP}{{\ensuremath{\mathbb P}} } 
 
\newcommand{\bbR}{{\ensuremath{\mathbb R}} } 
\newcommand{\bbS}{{\ensuremath{\mathbb S}} }

\DeclareMathOperator{\var}{Var}

\let\a=\alpha \let\b=\beta   \let\d=\delta  \let\e=\varepsilon
 \let\g=\gamma       \let\l=\lambda
      \let\o=\omega      
\let\r=\rho   \let\t=\tau   
  \let\z=\zeta
   \let\G=\Gamma   
\let\O=\Omega

\def\({\left(}
\def\){\right)}
\begin{document}
\begin{abstract}
We study a nonlinear recombination model from population genetics 
as a combinatorial version of the Kac-Boltzmann equation from kinetic theory. Following Kac's approach, the nonlinear model is approximated by a  mean field linear evolution with a large number of particles.  In our setting, the latter takes the form of a generalized random transposition dynamics. 
Our main results establish a uniform in time propagation of chaos with quantitative bounds, and a tight entropy production estimate for the generalized random transpositions, which holds uniformly in the number of particles. 
As a byproduct of our analysis we obtain sharp estimates on the speed of convergence to stationarity for the nonlinear equation, both in terms of relative entropy and total variation norm.  
\end{abstract}

\maketitle
\thispagestyle{empty}

\setcounter{page}{1}

\section{Introduction}\label{sec:1}
Recombinations are one of the principal components in the analysis of stochastic genetic algorithms \cite{Lyubich,Nagylaki}. 
Nonlinear recombinations provide a simple combinatorial setup for quadratic evolutions described by a Boltzmann-like equation \cite{Sinetal2}. A particle is represented by a finite string of characters from some finite alphabet and the binary collision mechanism is 
given by a recombination, that is the transposition of a random portion of the two colliding strings. The model belongs to the family of symmetric quadratic systems introduced in \cite{Sinetal}; see also \cite{entprod}
for the more general framework of reversible quadratic systems. 

Following the strategy introduced by Mark Kac in his seminal 1956 paper \cite{KAC}, one can approximate the nonlinear evolution of one particle by a linear mean field type Markov process involving  a large number of particles. Roughly speaking, if one has a good control of this approximation, together with a good control of the linear particle system, then the difficulties due to the nonlinearity in the original  process can be overcome.

In the context of Boltzmann's equation and its closely related  kinetic models this line of research has witnessed important progress in recent years \cite{entfound,kacsprog,anewapp,revchaos}, see also \cite{cattiaux,elementary,lacker} for related results for mean field diffusions of McKean-Vlasov type.  

The combinatorial setup considered in the present paper appears to be less explored; see however \cite{review_discrete,Reza,Bertini_etal} for the analysis of  Boltzmann-like equations with discrete velocities. One advantage of the combinatorial setting is that thanks to the discrete setup one can avoid a number of technical assumptions,  such as regularity and moments constraints, on the various distributions considered. Moreover, and perhaps more importantly, in contrast with the well studied case of the  Kac-Boltzmann equation \cite{entfound,kacsprog,einav2011villani,bonetto2017entropy}, in our setup it is possible to obtain tight entropy production estimates for the particle system which hold uniformly in the number of particles.  This provides  a class of  models for which the renowned Kac program can be completed in a  strong sense. 

In the setting of nonlinear recombinations the linear particle system takes the form of a generalized random transposition dynamics. This yields a natural generalization of the mean field exchange dynamics that are commonly studied in the probabilistic literature such as the Bernoulli-Laplace or the random transposition model \cite{DiacShash}. The purpose of this paper is twofold. 
On one hand we establish {\em uniform in time propagation of chaos}. On the other hand we prove tight estimates on the {\em entropy production} of the linear system which hold {\em uniformly in the number of particles}. 
As a corollary we obtain quantitative control on the convergence to stationarity for the nonlinear model in terms of relative entropy. In particular, this extends some results previously obtained in \cite{Sinetal2,entprod} by direct analysis of the entropy in the nonlinear recombination model. 
We now proceed with a detailed description of the model and of our main results.

\subsection{The nonlinear equation}\label{setup}
Let $\Omega = \prod_{i=1}^nX_i$ be the set of $n$-vectors $\sigma = (\sigma_1,\dots,\sigma_n)$ where $\si_i\in X_i$, and the $X_i$ are given finite sets. We interpret $\si$ as a particle. %with $n$ degrees of freedom. 
Thus, a particle is a string of $n$ characters each taken from a finite alphabet. 
A basic example is obtained by taking $\Omega=\{0,1\}^n$.  Without loss of generality we will assume that each $X_i$ has the form $X_i :=\{0,1,2,\dots, q_i\}$, for some $q_i \in \bbN$. %, and write $\vec q=(q_1,\dots,q_n)$. 
Given a subset $A \subset [n],$ $[n]=\{1,\dots,n\}$, and $\sigma \in \Omega,$ $\sigma_A$ denotes the $A-$component of $\sigma,$ that is the string $(\sigma_i, i \in A).$ If $(\sigma, \eta) \in \Omega \times \Omega$ is a pair of particles, the recombination at $A$ consists in exchanging the $A$-component of $\sigma$ with the $A$-component of $\eta.$ This defines the map 
\begin{align}
(\sigma,\eta) \mapsto (\eta_A\sigma_{A^c},\sigma_A\eta_{A^c}),
\end{align}
where $\eta_A\sigma_{A^c}$ denotes the element of $\Omega$ with entries $\eta_i$ for $i \in A$ and $\sigma_i$ for $i \in A^c=[n]\setminus A.$ 
Let $\cP(\O)$ denote the set of probability measures on $\O$.
If the original pair $(\si,\eta)$ is obtained by sampling independently from $p\in\cP(\O)$, then the new particle $\eta_A\si_{A^c}$ is distributed according to $p_A\otimes p_{A^c}$, the product of the two marginals of $p$ on the $A$ and $A^c$ components respectively. By choosing the set $A\subset[n]$ according to some distribution $\nu$, one obtains the quadratic collision kernel
\begin{align}
p\mapsto Q(p) = \sum_{A \subset [n]}\nu(A)\,p_{A}\otimes p_{A^c}.
\end{align}
The nonlinear evolution is defined by the dynamical system $\dot{p}_t = Q(p_t) - p_t$, that is 
\begin{align}\label{boltzeq}
\frac{d}{dt}\,p_t = \sum_{A \subset [n]}\nu(A)\left(p_{t,A}\otimes p_{t,A^c} - p_t\right)
\end{align}
with the initial condition $p_0=p\in\cP(\O)$. Here $p_t\in\cP(\O)$ is the distribution  of the particle at time $t$, $p_{t,A}$ denotes its marginal on $A$ and $\nu$ is a given probability measure over the subsets of $[n]$.
The study of this model starts with the pioneering work of Geiringer \cite{Gei}; %and Bennett \cite{Ben}; 
see also \cite{Sinetal,Sinetal2,Baakeetal,Martinez} for more recent accounts. It is well known that the Cauchy problem associated to \eqref{boltzeq} has a unique solution for every initial distribution $p\in\cP(\O)$; see e.g.\ \cite{Baakeetal}.  Moreover, it is not difficult to see that the evolution preserves the single site marginals, that is $p_{t,i}=p_{i}$ for all $t\geq 0$ and for all $i\in[n]$.  We say that the recombination measure $\nu$ is {\em separating} if for any $i, j \in [n]$ there is a positive probability that the random set $A$ with distribution $\nu$ separates $i$ and $j$. Equivalently, if $r(\nu)<1$ where
\begin{equation}\label{def:rnu}
r(\nu):= \max_{i<j \in [n]}\,\nu\left(\{i,j\}\subset A \text{ or }\{i,j\}\subset A^c\right)
\end{equation} 
denotes the maximum over $i<j$ of the probability that $A$ does not separate $i,j$.
It is a classical fact that, under the assumption that $\nu$ is separating, the system converges to the stationary state given by the product of the marginals of 
the initial state $p$; namely, if $\pi_i=p_{i}$ denotes the marginal of $p$ at site $i$, then 
\begin{equation}\label{station}
\pi=\otimes_{i=1}^n\,\pi_i
\end{equation}
is the equilibrium distribution and one has the convergence $p_t\to \pi$, $t\to\infty$, 
which  can be interpreted as the effect of repeated fragmentations of the initial state; see e.g.\ \cite{Sinetal2}. Some of our results will hold for arbitrary separating $\nu$. 
In some other cases we consider a slightly stronger assumption on $\nu$. 
Two examples to which all our results apply  are the following distributions $\nu$, which are commonly considered in the genetic recombination literature. We refer to \cite{Sinetal2,entprod} for more examples. \bigskip
\begin{enumerate}[1)]
\item 
{\em Uniform crossover}: $\nu(A)=\frac1{2^n}$, for all $A\subset[n]$;
\item 
{\em One-point crossover}: 
$\nu(A)=\tfrac1{n+1}\sum_{i=0}^n\ind_{A=J_i}%+\ind_{A=J_i^c}]
$, where $J_0=\emptyset$, $J_i=\{1,\dots,i\}$, $i\geq 1$.
\end{enumerate}

The quantitavive analysis of the convergence to equilibrium $p_t\to \pi$, $t\to\infty$ has been initiated in \cite{Sinetal,Sinetal2}, where a ``mixing time" bound was obtained for the discrete time version of the model. The decay to equilibrium in relative entropy for the continuous time model was studied in \cite{entprod}. These results were obtained by direct analysis of the nonlinear problem. In this paper we shall follow an entirely different approach, inspired by Kac's program from kinetic theory. As a byproduct of our analysis, we shall obtain an alternative proof of the  known results mentioned above. 

\subsection{The particle system}
Suppose there are $N$ ``particles'', described by variables $\eta(j)\in\O$, $j=1,\dots,N$. That is, each particle is a single string from $\O$ and $\eta_i(j)$ denotes the content of the $j$-th particle at site $i\in[n]$. We may picture $\eta\in\O^N$ as a $N\times n$ matrix such that each row is a particle with $n$ entries, and for each $i\in[n]$, the $i$-th column $\eta_i$ represents the content of site $i$ for different particles. 

Notice that $N$ and $n$ play two very different roles here. The number $N$ of particles will eventually be taken to $+\infty$ to recover the non linear mean field limit, in accordance with the general Kac program. The number $n$ should be thought as a fixed, possibly large quantity describing the size of a single particle space.
 
The Markov process is given by the following random pair-exchange process. 
Pairs of particles $\{j,l\}$, $1\leq j<l\leq N$ are chosen independently according to a Poisson clock process with rate $1/N$. When the pair $\{j,l\}$ ``rings", %If $j=\ell$ nothing happens; if $\ell\neq j$, 
then 
a set $A\subset[n]$ is chosen with probability $\nu(A)$ and the recombination 
\begin{align}\label{eqtran}
(\eta(l),\eta(j)) \mapsto (\eta_A(j)\eta_{A^c}(l),\eta_A(l)\eta_{A^c}(j))
\end{align} 
is performed, that is the $A$-content  is exchanged between particle $j$ and particle $l$. For all $j,l \in [N], A \subset [n],$ for all $\eta \in \Omega^N$, we write $\eta^{j,l,A}$ for the new configuration $\eta' \in \Omega$ defined by 
\begin{align}\label{def:etajl}
\eta'(k)=\eta(k), \;\forall k \neq j,l;\qquad  \eta'(l)=\eta_{A}(j)\eta_{A^c}(l),\qquad \eta'(j)=\eta_A(l)\eta_{A^c}(j).
\end{align}
With this notation the pair configuration in the right hand side of \eqref{eqtran} is $(\eta^{j,l,A}(l),\eta^{j,l,A}(j))$. 
Define also $f^{j,l,A}(\eta)=f(\eta^{j,l,A})$ %the operator $Q_{j,l,A}$ by $Q_{j,l,A}f(\eta) = f(\eta^{j,l,A}),$ 
for all $f:\Omega^N \rightarrow \bbR$.  Then, the infinitesimal generator of the continuous time Markov process is given by %can be written as 
\begin{align}\label{operator}
\cL_N f = \frac{1}{N}\sum_{1\leq j<l\leq N}\sum_{A \subset [n]}\nu(A)\left(f^{j,l,A}-f\right), \qquad\; f:\Omega^N \rightarrow \bbR.
\end{align}
Any product measure $\mu_N$ on $\Omega^N$ of the form $\mu_N = \mu^{\otimes N},$ where $\mu$ is itself a product measure $\mu=\mu_1\otimes\cdots\otimes\mu_n$ on $\Omega=\prod_{i=1}^nX_i,$ defines a reversible measure for the generator $\cL_N$.  Indeed, for such a measure one has the symmetry 
$\mu_N(\eta^{j,l,A})=\mu_N(\eta)$ for all $\eta\in\O^N$ and therefore $\cL_N$ is self-adjoint in $L^2(\mu_N)$.  
The process is not irreducible in the state space $\Omega^N$ since the content at a site for one particle is always exchanged with the content at the same site for another particle, and thus the number of particles with a given element $x \in X_i$ at a given site $i \in [n]$ is constant in time. To obtain an irreducible process one must fix the densities $\rho_{i,x},$ $i \in [n],x\in X_i$ defined by
\begin{align}
\rho_{i,x} = \frac{1}{N}\sum_{j=1}^N\ind(\eta_i(j) = x).
\end{align}
We call $\rho = (\rho_{i,x})$ the corresponding vector. % Once the vector $\rho$ has been fixed, then the process becomes irreducible. More precisely, g
Given a density vector $\rho$, consider the space
\begin{align}\label{omegarhosets}
\Omega_{\rho} := \Big\{(\eta(1),\dots,\eta(N)) \in \Om^N \!\!:\, \;
\textstyle{\rho_{i,x} = \frac{1}{N}\sum_{j=1}^N \ind(\eta_i(j) = x)},\;\, \forall i \in [n], \  x \in X_i\Big\}.
\end{align}
The set $\Om_{\rho}$ is well defined and non-empty for every vector $\r=\r_N$ such that $\r_{i,x}\in[0,1]$, $\sum_{x\in X_i}\r_{i,x}=1$ for all $i\in[n]$, and such that $N\r_{i,x}$ is an integer for all $i,x$. When  this holds %for all $N\in\bbN$, 
we say that $\r_N$ is an {\em admissible sequence}.
Under suitable assumptions on the recombination measure $\nu$, see Definition \ref{def:strict}, for any given admissible $\r_N$, the Markov process with state space $\O_{\r_N}$ and generator $\cL_N$ is 
irreducible and converges to the uniform distribution on $\O_{\r_N}$. We will be interested in quantitative statements about this convergence. 

We often use the following procedure to construct admissible sequences. Fix a given $\pi=(\pi_{i,x})$ satisfying
$\pi_{i,x}\in[0,1]$ for all $i,x$ and $\sum_{x\in X_i}\pi_{i,x}=1$ for all $i\in[n]$. Then we call $\rhoN=\rhoN(\pi)$ the density %converging to any $\rho$ we choose such that $\Om_{\rhoN}$ is always well 
defined by 
\begin{align}\label{rhoseq}
\rhoN_{i,x} :=\tfrac1N\,{\left\lfloor N \pi_{i,x}\right\rfloor}\,, \qquad i\in[n], \;x\in \{1,\dots,q_i\},
\end{align}
 and we set $ \rhoN_{i,0}=1-\sum_{x=1}^{q_i}\rhoN_{i,x}$. 
We remark that $\rhoN=\rhoN_N$ is an admissible sequence satisfying $$\rhoN_{i,x}=\pi_{i,x} +O\left(\tfrac1N\right)$$ for all $i,x$.  
Fixing the probability vector $\pi=(\pi_{i,x})$ is equivalent to fixing the stationary measure \eqref{station} of the nonlinear evolution, %$\pi=\prod_{i=1}^n\pi_i$ of $p_t,$ where $p_{0,i}=\pi_i \, \, \forall i \in [n]$, 
and thus we use the same symbol for them. 
From now on it is assumed that the densities $\pi_i$, and thus the corresponding product measure $\pi$ as in \eqref{station}, are fixed.  
Without loss of generality we restrict to the case where $\pi_{i,x} \in(0,1)$ for all $i \in [n],x\in X_i$ since for each $i$ we can otherwise discard those letters $x\in X_i$ such that $\pi_{i,x}$ is zero, and thus consider a new configuration space such that $\pi$ is everywhere positive.

\subsection{Chaos and the propagation of chaos}
The chaos property is commonly defined as follows. 
A measure $\mu_N \in \cP\left(\O^{N}\right)$ is {\em symmetric} if it is invariant under any permutation of the $N$  particles. We write $ P_k\mu_N\in\cP(\O^k)$ for the corresponding $k$-particle marginal.  \begin{definition}[Kac's chaos or ``Boltzmann property'']\label{kacschaos}
A sequence $\mu_N \in \cP\left(\O^{N}\right)$ of symmetric probabilities on $\Om^{N}$ is $\mu-$chaotic, for a given $\mu \in \cP(\O)$, if for any $k \in \bbN$ one has the weak convergence 
$$ P_k\mu_N\longrightarrow \mu^{\otimes k}\,,\qquad N\to\infty.$$
\end{definition}
A key step in implementing Kac's program is to construct a correspondence between probability measures on $\O$ and probability measures on $\O^N$. In our setting this can be formulated as follows. 

\begin{definition}[Canonical tensor product]\label{def:canonical}
Given a probability measure $p \in \cP(\Om)$ and an admissible sequence of density vectors $\r_N$, we let $$\gamma(p,\r_N):=p^{\otimes N}\left(\,\cdot\,|\,\Om_{\r_N}\right)$$ be the tensor product of $p$ conditioned on $\Om_{\r_N}$. When $\r_N$ is given by $\rhoN$ as in \eqref{rhoseq}, where the $\pi_{i,x}=p(\si_i=x)$ are the marginals of $p$, we use the notation $\g_N(p):= \gamma(p,\rhoN)$, and call it the {\em canonical tensor product}.  
\end{definition}
To avoid degeneracies we sometimes assume the following property. 
\begin{definition}[Irreducibility]\label{def:irred}
A probability measure $p \in \cP(\Om)$ is called {\em irreducible} if for any $i\in[n]$, any $x\in\{1,\dots,q_i\}$, there exists $\chi\in\O$ such that $p(\si_i=x,\si_{j}=\chi_{j}\;\forall j\neq i)>0$ and $p(\si_i=0,\si_{j}=\chi_{j}\;\forall j\neq i)>0$. 
\end{definition}
If $p$ is  irreducible and the sequence $\r_N$ is sufficiently close to the marginals of $p$, then the local central limit theorem guarantees that the $k$-particle marginals of the symmetric measures $\gamma(p,\r_N)$ converge to the product $p^{\otimes k}$ as $N\to\infty$, that is $\gamma(p,\r_N)$ is $p$-chaotic, see Theorem \ref{th:eqens} for a precise statement. In particular, for the canonical tensor product $\g_N(p)$, we will see that 
\begin{align}\label{eq:chaos0}
\|P_k\g_N(p)-p^{\otimes k}\|_\tv \leq \frac{C_0\,k}{N},
\end{align}
for some constant  $C_0=C_0(p)$, where $\|\cdot \|_\tv$ denotes the total variation distance.
Clearly, our reference product measure $\pi\in\cP(\O)$ is irreducible. In fact, $\gamma(\pi,\r_N)$ is  the {\em uniform probability measure on} $\Om_{\r_N}$, for any admissible sequence $\r_N$. In particular, it follows that its $k$-particle marginals converge to $\pi^{\otimes k}$ as $N\to\infty$. One can also show that $\gamma(p,\r_N)$ is {\em entropically $p$-chaotic} in the sense defined in \cite{entfound}, namely that on top of the convergence of marginals one also has
 \begin{align}\label{entropicchaos}
\lim_{N\to\infty} \frac1N\,H_N(\gamma(p,\r_N)\tc \gamma(\pi,\r_N))=H(p\tc\pi),
\end{align}
where $H(\cdot\tc\cdot)$, $H_N(\cdot\tc\cdot)$ denote respectively the relative entropy %or KL divergence 
for probability measures on $\O$ and on $\O^N$, see Proposition \ref{th:entchaos} below. Let us recall the following standard definition.

\begin{definition}[Propagation of chaos]\label{kacschaosprop}
Let $\mu_{N,t}=\mu_{N}e^{t\cL_N}$, $t\geq 0$, denote the evolution of an initial symmetric distribution $\mu_{N}\in\cP(\O^N)$ under the Markov process generated by $\cL_N$. Suppose that $\mu_{N}$ is $p$-chaotic for some $p\in\cP(\O)$ and let $p_t$ denote the evolution of the initial datum $p$ under the nonlinear process \eqref{boltzeq}. If $\mu_{N,t}$ is $p_t$-chaotic for all fixed $t\geq 0$, then we say that {\em propagation of chaos} holds. If the weak convergence $P_k\mu_{N,t}\longrightarrow p_t^{\otimes k}$, $N\to\infty$, holds uniformly in $t\geq 0$ we say that propagation of chaos holds {\em uniformly in time}.
\end{definition}
An %straightforward 
adaptation of well known arguments, see e.g.\ \cite{KAC,sznitman}, shows that the propagation of chaos (at fixed times) holds in our setting. In fact, the proof of this does not require the assumption that $\nu$ is separating.

\subsection{Uniform in time propagation of chaos}
Our first main result  concerns the validity of propagation of chaos uniformly in time, with quantitative bounds on the convergence as $N\to\infty$. 
\begin{theorem}\label{th:unifintime}
Assume that $\nu$ is separating. The propagation of chaos holds uniformly in time, that is for any $p\in\cP(\O)$, if $\mu_{N}$ is $p-$chaotic, then for all fixed $k \in \bbN$, as $N\to\infty$, $$P_k\mu_{N,t}\longrightarrow p_t^{\otimes k},\qquad \; \text{uniformly in $t\geq 0$},$$ where $p_t$ is the solution to the nonlinear equation \eqref{boltzeq} with initial datum $p_0=p$. 
 Moreover, if $\mu_{N}$ is the canonical tensor product $\mu_{N}=\gamma_N(p)$, and $p\in\cP(\O)$ is irreducible, then 
 \begin{align}\label{eq:chaos1}
\|P_k\mu_{N,t}-p_t^{\otimes k}\|_\tv \leq \frac{C}{\sqrt{N}},
\end{align}
for some constant $C=C(k,p)>0$ independent of $t,N$. 
\end{theorem}

\begin{remark}
Concerning the dependency on $N$ it may be that the optimal decay in \eqref{eq:chaos1} is $O(1/N)$ rather than $O(1/\sqrt N)$. This seems natural in light of our estimate \eqref{eq:chaos0} at time zero. 
Moreover, that would be  in agreement with the recent results in \cite{lacker}, where the $O(1/N)$ bound is obtained  for a class of interacting diffusion processes at fixed times. The value of the constant $C$ in \eqref{eq:chaos1} can be in principle obtained from our more detailed results in Theorem \ref{unifpropchaos} below. However, we have not tried to optimize the dependency of $C$ on $k,p$.

\end{remark}

There are by now several results for kinetic models and for mean field diffusions establishing uniform in time propagation of chaos, see \cite{DelMoralGuionnet,cattiaux,elementary,revchaos}.  
However, the adaptation to our setting of the different techniques used in these works does not seem to be straightforward. The proof of Theorem \ref{th:unifintime} is based on some new contractive estimates for the nonlinear model that allow us to implement the main strategy developed in the groundbreaking work of Mischler and Mouhot \cite{kacsprog}, see Section \ref{sec:unifintime}.

\subsection{Entropy production for generalized random transpositions}
Our second main result is about quantitative estimates on the  decay to equilibrium for the particle system introduced above. The key feature is that these estimates hold {\em uniformly in the number of particles} $N$. We shall actually derive such estimates in the context of the {\em generalized random transposition process} defined as follows. 

Let $\cS_{N,n}=S_N^n$ denote the $n$-fold product of the symmetric group $S_N$ of the permutations of $[N]=\{1,\dots,N\}$. Then $\eta\in \cS_{N,n}$ is a matrix $\eta_i(j)$, $i\in[n]$, $j\in[N]$, where each $\eta(j)=(\eta_1(j),\dots,\eta_n(j))\in[N]^n$ is seen as a particle, and each $\eta_i=(\eta_i(1),\dots,\eta_i(N))\in S_N$ is a permutation of $[N]$. Note that $\cS_{N,n} = \O_{\r}$ where $\O_{\r}$ is defined in \eqref{omegarhosets} when we take the extreme case $q_i\equiv N-1$, $\r_{i,x}\equiv 1/N$ for all $i=1,\dots,n$. 

The generalized random transposition (GRT) process is defined as the process generated by the operator  $\cL_N$ in this setup, namely the GRT process is the continuous time Markov process with state space $\cS_{N,n}$ described as follows:  every pair of particles $\{\eta(l),\eta(j)\}$ collides with rate $1/N$ independently, and when a collision occurs, a new set $A$ is sampled according to $\nu$ and the $A$-content of $\eta(l),\eta(j)$ is exchanged.  

This setting is convenient for proving functional inequalities since by restricting to classes of functions with suitable symmetries we then recover all possible cases of processes on $\O_{\r_N}$ with generator $\cL_N$, for all  admissible $\r_N$, see Remark \ref{rem:opt} below for more details.  As an example, consider the case $q_i\equiv 1$ and suppose that $N(\r_N)_{i,1} = N_i$ for some positive integers $N_i$, $i=1,\dots, n$. Here the process generated by $\cL_N$ can be seen as the GRT process restricted to functions $f:\cS_{N,n}\mapsto\bbR$ such that, for each $i$, $f$ only depends on $(\eta_i(1),\dots,\eta_i(N))$ through the unordered set $\{\eta_i(1),\dots,\eta_i(N_i)\}$. 
 This can be seen as a generalized Bernoulli-Laplace process \cite{DS}.
If no confusion arises we continue to write $\cL_N$ for the generator of the GRT.

We remark that when $n=1$, GRT is just the usual random transposition process \cite{DiacShash}, and that when $\nu$ gives positive weight only to $A\subset [n]$ such that $|A|=1$, it describes $n$ independent random transposition processes. However, in the general case, the recombination measure $\nu$ dynamically couples the permutations and the GRT becomes a nontrivial generalization of the standard random transpositions.

In order to guarantee the irreducibility of the GRT process, we make the following assumption on the recombination measure $\nu$, which is easily seen to be  stronger than the separation assumption $r(\nu)<1$; see also Remark \ref{rem:nonreg}.
\begin{definition}\label{def:strict}
We say that $\nu$ is {\em strictly separating} if for all $i\in[n]$ there exists $A\subset [n]$ such that $i\in A$ and such that both $A$ and $A\setminus\{i\}$ have positive $\nu$-probability.  
\end{definition}
Note that the uniform crossover and the one-point crossover are both strictly separating. 
Let $\pi_N$ denote the uniform distribution on $\cS_{N,n}$. The GRT process is reversible with respect to $\pi_N$ and if the measure $\nu$ is strictly separating, then it is also irreducible, and  any initial distribution converges to  $\pi_N$ as $t\to\infty$. To quantify this statement we consider the Dirichlet form of the GRT, defined by
 \begin{align}\label{operatorG}
\cE_{N,n}(f,g)=\frac{1}{2N}\sum_{1\leq j<l\leq N}\sum_{A \subset [n]}\nu(A)
\sum_{\eta\in\cS_{N,n}}\pi_N(\eta)\left(f(\eta^{j,l,A})-f(\eta)\right)\left(g(\eta^{j,l,A})-g(\eta)\right),
\end{align}
where  $f,g:\cS_{N,n} \mapsto\bbR$. 
The {\em entropy production rate} is measured by the constant
\begin{align}\label{deltaoperatorG}
\a(N,n) = \inf_{f>0} \frac{\cE_{N,n}(f,\log f)}{\ent(f)},
\end{align}
where the infimum is over $f:\cS_{N,n} \mapsto\bbR_+$ such that $\ent(f)\neq 0$ and 
$$\ent(f) = \pi_N(f\log f)-\pi_N(f)\log\pi_N(f)$$ is the entropy of $f$ w.r.t.\ $\pi_N$.
 Equivalently, $\a(N,n)$ is the best constant 
$\a$ such that  the inequality 
 \begin{align}\label{gapoperatorG}
\ent(e^{t\cL_N}f)\leq e^{-\a t}\ent(f)
\end{align}
holds for all functions $f>0$; see e.g.\ \cite{DS,BT}. Note that when $\pi_N(f)=1$ then, for all $t\geq 0$, $\ent(e^{t\cL_N}f)$ coincides with the relative entropy %one has $H_N(\mu_N\tc\pi_N) = \ent(f)$ and 
$H_N(\mu_{N,t}\tc\pi_N)$ 
%= \ent(e^{t\cL_N}f)$ 
where %$\mu_N=f\pi_N$ and 
$\mu_{N,t}
%=\mu_Ne^{t\cL_N}
=(e^{t\cL_N}f)\pi_N$.

We also consider the entropy production rate restricted to the set of symmetric functions defined as follows. Let $\bbS$ denote the set of $f:\O^N\mapsto \bbR$ such that $$f(\eta)=\frac1{N!}\sum_{\t\in S_N} f(\t\circ\eta),$$
where the sum runs over all permutations $\t\in S_N$ and $\t\circ\eta$ denotes the configuration with particles exchanged according to $\t$, that is $\t\circ\eta(j) = \eta(\t(j))$.
From the point of view of Kac's program \cite{KAC}, $\bbS$ is the relevant space of observables in the particle system. We call $\a_{\bbS}(N,n)$ the constant defined as in \eqref{deltaoperatorG}, with the infimum restricted to positive  functions $f\in\bbS$. 

Our main results for the GRT process are the following estimates independent of $N$.

\begin{theorem}\la{th:ent}
Fix $n\in\bbN$ and assume that $\nu$ is strictly separating. Then there exists $\a(\nu)>0$ such that for any $N\in\bbN$, $N\geq 2$,
 \begin{equation}
\label{entao}
\a(N,n)\geq \a(\nu).
\end{equation}
Moreover, if $\nu$ is the one-point crossover, then 
 \begin{equation}
\label{enta}
\a(N,n)\geq \frac{1}{4(n+1)}\,,
\end{equation}
and if $\nu$ is the uniform crossover, then 
 \begin{equation}
\label{entas}
\a(N,n)\geq \frac{1}{4n}\,,\qquad \a_\bbS(N,n)\geq \frac1{2(n+2)}\,.
\end{equation}
%Finally, restricting to symmetric functions one has $\a_\bbS(N,n)\geq \frac1{2(n+2)}$.  
\end{theorem}

\begin{remark}
\label{rem:opt}
Consider the 
entropy production rate $\a(\O_{\r_N})$ for the process on $\O_{\r_N}$ associated to any admissible density sequence $\r_N$. The uniform distribution on $\O_{\r_N}$ is the image of the probability $\pi_N$ (the uniform distribution on $\cS_{N,n}$) under a simple symmetrization procedure, and the  quantity $\a(\O_{\r_N})$ can be defined as in 
\eqref{deltaoperatorG} by restricting to invariant classes of functions with suitable symmetries; see e.g.\ \cite[Section 4.2.3]{barthe_et_al} for a version of this simple projection argument. 
The estimates of Theorem \ref{th:ent} then immediately provide the lower bound
 \begin{equation}
\label{lowbos}
\a(\O_{\r_N})\geq \a(N,n)\geq \a(\nu)\,,
\end{equation}
for any admissible density sequence $\r_N$.
%In particular, in the case of uniform crossover one finds 
%$\a(\O_{\r_N})\geq 1/4n$, for any 
%admissible $\r_N$. 
\end{remark}
\begin{remark}
\label{rem:nonreg}
The statement in Theorem \ref{th:ent} implies exponentially fast convergence to stationarity for the GRT under the strict separation assumption, see \eqref{gapoperatorG}. In particular, it implies irreducibility. If the recombination measure $\nu$ is only assumed to be separating, then the GRT process may fail to be irreducible. Therefore, some assumption such as the  strict separation defined above is necessary for the statement in Theorem \ref{th:ent}. For an example of non irreducible process with separating $\nu$  consider $N=2,n=4$ and suppose $\nu(A)=\frac16$ for all $A\subset[4]$ with $|A|=2$. Clearly, $\nu$ is separating, but if we consider the initial configuration $\eta$ with $\eta(1)=0000, \eta(2)=1111$, then the number of $1$'s in each particle remains even at all times.  
\end{remark}

The proof of Theorem \ref{th:ent}  will be based on some new functional inequalities for permutations which imply a modified logarithmic Sobolev inequality for the GRT. We refer to Section \ref{sec:entdec}.

Concerning upper bounds on the constant $\a(N,n)$ we establish an estimate valid for arbitrary $\nu$, which essentially shows that $\a(N,n)$ cannot be larger than $4/n$ for $n$ large, provided $N$ is taken large enough, possibly depending on $n$. 
\begin{proposition}\label{prop:upbo}
For any $n\in\bbN$, any distribution $\nu$ on $[n]$, 
\begin{equation}
\label{entalow}
\limsup_{N\to\infty}\a(N,n)\leq \frac4n + O\left(\frac1{n^2}\right).
\end{equation}
\end{proposition}
In this sense, the bounds in \eqref{enta} and \eqref{entas} can be considered to be optimal. % up to constants.  

\subsection{Kac's program completed}
One of the main motivations behind Kac's program is the derivation of quantitative bounds on the speed of convergence to equilibrium for the nonlinear equation. In our setting, as a corollary of our analysis we obtain the following relative entropy estimates. We refer to \cite{Vil,entfound} for a discussion of related entropy decay estimates in the context of kinetic models. In particular, in our setup, one can say that Cercignani's conjecture holds true. See also \cite{erbar} for related results in a discrete setting under positive curvature assumptions. 
\begin{theorem}\label{entprodthm}
Assume that $\nu$ is strictly separating. For any $p\in\cP(\O)$, let $p_t$ denote the solution of \eqref{boltzeq} with $p_0=p$ and let $\pi=\otimes_{i=1}^n p_i$ denote the associated equilibrium.  Then for all $t\geq 0$,
\begin{align}\label{entproduct}
H(p_t\tc\pi) \leq e^{-\a(\nu) t}H(p\tc\pi),
\end{align}
where $\a(\nu)>0$ is the constant in Theorem \ref{th:ent}. In particular, $\a(\nu)\geq 1/4(n+1)$ for one-point crossover, and $\a(\nu)\geq 1/2(n+2)$ for uniform crossover. 
\end{theorem}
It is interesting to note that the constant $\a(\nu)$ does not depend on the initial datum $p$ in any way. We point out that, in the case of  the one-point crossover and uniform crossover, the above estimates were already obtained in \cite{entprod} by direct analysis of the entropy production functional of the nonlinear equation, with a slightly better constant actually: $\a(\nu)\geq 1/(n+1)$ in both cases. Moreover,  \cite{entprod} also shows that the $1/n$ decay of the constant $\a(\nu)$ in these cases is optimal up to a constant independent of $n$. Besides extending the bounds of \cite{entprod} to all strictly separating distribution $\nu$, an interesting feature of Theorem \ref{entprodthm} is that its proof takes a completely different route. Namely, it is based on the implementation in our setting of Kac's original idea. More precisely, \eqref{entproduct} is derived from the uniform control on entropy production provided by Theorem \ref{th:ent}, see also Remark \ref{rem:opt}, together with the approximation, as $N\to\infty$, of both $H(p\tc\pi) $ and $H(p_t\tc\pi)$ in terms of the corresponding entropies for the $N$-particle system. 

We also obtain the following general bounds on the convergence to equilibrium for the nonlinear chain.  Recall the definition \eqref{def:rnu} of the constant $r(\nu)\in(0,1)$.
\begin{theorem}\label{prop:tvbo}
Assume that $\nu$ is separating. For any $p\in\cP(\O)$, let $p_t$ denote the solution of \eqref{boltzeq} with $p_0=p$ and let $\pi=\otimes_{i=1}^n p_i$ denote the associated equilibrium.  Then for all $t\geq 0$,
\begin{align}\label{entproductnoncontr}
H(p_t\tc\pi) \leq \tfrac12\,n(n-1)\,H(p\tc\pi)\,e^{-D(\nu)\,t},
\end{align}
where $D(\nu):=1-r(\nu)$.
Moreover, for the total variation distance we have
\begin{align}\label{tvproduct}
\|p_t-\pi\|_\tv \leq \tfrac12\,n(n-1)\,\|p-\pi\|_\tv\,e^{-D(\nu)\, t}.
\end{align}
\end{theorem}
We note that an estimate similar to \eqref{tvproduct} was obtained in \cite{Sinetal2} for a discrete time version of the nonlinear process. To prove Theorem  \ref{prop:tvbo} we use a coupling argument similar to that of \cite{Sinetal2}, together with an explicit construction of the continuous time solution $p_t$ in terms of all possible collision histories, which goes back to the pioneering works of Wild \cite{Wild} and McKean \cite{mckean,mckean1966speed}, see also \cite{carlenwild}. It is interesting to note that in the case of uniform crossover one has $r(\nu)=1/2$ and thus \eqref{entproductnoncontr} provides an exponential decay which is much faster, as $n$ becomes large, than the one provided by \eqref{entproduct}. Moreover, as mentioned, the $1/n$ rate is known to be optimal up to a constant independent of $n$ for the estimate  \eqref{entproduct}. This mismatch can be explained by observing that, because of the possibly large prefactor, \eqref{tvproduct} only provides information about the large time behavior while  
\eqref{entproduct} expresses a contraction property of the relative entropy at all times, and that some particular initial distributions $p$ may have a slow start in the relative entropy decay; see Lemma \ref{lem:entprod} below for a concrete example.  

Finally, we remark that the rate of exponential decay $D(\nu)=1-r(\nu)$ in Theorem \ref{prop:tvbo} is optimal, in the sense that $t^{-1}\log \|p_t-\pi\|_\tv$, as $t\to\infty$, cannot be smaller than $-D(\nu)$, see Remark \ref{rem:Dopt}. We refer to \cite{dolera_etal} for a related  result on the optimal rate of decay in the context of Kac model. 

\subsection{Organization of the paper} In Section \ref{sec:clt_chaos} we present the main preliminary facts concerning the local central limit theorem and its applications to the proof of chaos results.
In Section \ref{sec:unifintime} we prove the uniform in time propagation of chaos stated in Theorem \ref{th:unifintime}. This section also contains the proof of Theorem \ref{prop:tvbo}. Section \ref{sec:entdec} is devoted to the proofs of Theorem \ref{th:ent}, Proposition \ref{prop:upbo}, and Theorem \ref{entprodthm}. In the appendix we give the detailed proof of the local central limit theorem statement used in the main text. 

\subsection*{Acknowledgements} We would like to thank Arnaud Guillin, Cyril Labb\'e, Hubert Lacoin and Alistair Sinclair, for insightful conversations related to this work.

\section{Local Central Limit Theorem and Chaos}\label{sec:clt_chaos}
A probability measure $p\in\cP(\O)$ induces a probability $\mu$ on $X:=\{0,1\}^K$ where $K=\sum_{i=1}^n q_{i}$, via the map
\begin{align}\label{induce}
\si\in\O \mapsto \xi_{i,x} = \ind(\si_i=x)\,,\qquad i=1,\dots,n\,;\; x\in \{1,\dots,q_i\}.
\end{align}
That is, $\mu$ is the push forward of $p$ by the above map. 
Note that we did not include the indicator variable $ \ind(\si_i=0)$ since this is uniquely determined as the indicator of the event $\xi_{i,x}=0$ for all $x\in \{1,\dots,q_i\}$. When $q_i=1$ for all $i$, then $\O=\{0,1\}^n$ can be identified with $X$, $\si$ with $\xi$, and $\mu$ with $p$.

\subsection{Central limit theorem} 
The next results are concerned with the behavior of the sum of independent copies $\xi(1),\dots,\xi(N)$ of a random variable $\xi$ with values in $X$ and distribution $\mu\in\cP(X)$:
$$
S_N = \sum_{j=1}^N\xi(j).
$$
Thus, $S_N$ is a random vector in $\{0,\dots,N\}^K$.  We use the notation 
$\scalar{t}{s}=\sum_{i,x}t_{i,x}s_{i,x}$ if $t$ and $s$ are indexed by $i=1,\dots,n$ and $x=1,\dots,q_i$. 
We call $V_1$ the covariance matrix of $\mu$, %namely 
$$V_1(i,x;i',x')=\mu(\xi_{i,x}\xi_{i',x'})-\mu(\xi_{i,x})\mu(\xi_{i',x'}),\qquad i=1,\dots,n\,,\; x=1,\dots,q_i$$
Thus $V_1$ is a symmetric nonnegative definite $K\times K$ matrix. If $\det V_1\neq 0$ we say that $\mu$ is {\em nondegenerate}. The central limit theorem asserts that if $\mu $ is nondegenerate, then as $N\to\infty$ one has the weak convergence %the random variable 
 \begin{align}\label{eq:CLT}
\frac1{\sqrt N} \;V_1^{-1/2}\left(S_N - \mu^{\otimes N}(S_N)\right)\;
\longrightarrow\; N(0,\ind_K)\,,
\end{align}
where $\mu^{\otimes N}(S_N)\in[0,N]^K$ is the mean of the vector $S_N$ under the product measure $\mu^{\otimes N}$, and $\ind_K$ denotes the $K\times K$ identity matrix, so that $N(0,\ind_K)$ is the standard normal in $K$ dimensions. Note that when $\mu$ is induced by a measure $p\in\cP(\O)$ as described in \eqref{induce}, then $\mu^{\otimes N}(S_N)_{i,x}=N\pi_{i,x}$ for all $i,x$, where $\pi_{i,x}$ are the marginals of $p$. 

The statement \eqref{eq:CLT} clearly requires that $\mu$ is nondegenerate. However, one can obtain similar statements in the case of degenerate measures, provided one reduces to the nondegenerate modes by eliminating the degenerate ones. More precisely, one can take the eigenvectors of $V_1$ with nonzero eigenvalues as the new variables. A simple example is obtained if e.g.\ $q_i\equiv 1$ and $\mu$ gives probability $1/2$ to all $1$'s and  probability $1/2$ to all $0$'s. Here one simply removes all variables but one. 

\subsection{Local  central limit theorem}
We will need a local version of the central limit theorem. For this we assume the following stronger notion of nondegeneracy, which we refer to as irreducibility.

\begin{definition}\label{def:irr}
A measure $\mu\in\cP(X)$ is called {\em irreducible} if for all $i=1,\dots, n$, for all $x\in\{0,\dots,q_i\}$, there exists $\xi\in X$ such that $\mu(\xi)$ and $\mu(\xi^{(i,x)})$ are both positive, where $\xi^{(i,x)}$ denotes the vector $\xi$ with the $(i,x)$-th coordinate flipped, that is $\xi^{(i,x)}_{j,y}=\xi_{j,y}$ for all $(j,y)\neq (i,x)$, and $\xi^{(i,x)}_{i,x}=1-\xi_{i,x}$. 
\end{definition} 
It is immediate to check that if $p\in\cP(\O)$ is irreducible in the sense of Definition \ref{def:irred} then the measure $\mu$ induced on $X$ by $p$ as in \eqref{induce} is irreducible in the sense of Definition \ref{def:irr}.

\begin{proposition}
\label{prop:LCLT}
Suppose $\mu\in\cP(X)$ is irreducible. Then there exists a finite constant $C=C(\mu)$ such that for all $N\in\bbN$, 
 \begin{align}\label{eq:LCLT}
\max_{M_N}\;
\left |\;\mu^{\otimes N}\left(S_N=M_N\right) \;-\; \frac{e^{-\frac12\,\scalar{z_N}{z_N}}}{(2\pi N) ^{K/2}%^{\frac{K}{2}}
 \sqrt {\det V_1} 
}\;
\right| \leq \frac{C}{ N^{(K+1)/2}},
\end{align}
where $$z_N:=\frac1{\sqrt N} \;V_1^{-1/2}\left(M_N - \mu^{\otimes N}(S_N)\right),$$ and the maximum is over all possible values $M_N\in\{0,\dots,N\}^K$.
\end{proposition}
Noting that $\O_{\r_N}=\{S_N=M_N\}$ with $M_N=N\r_N$, and that in this case $$\scalar{z_N}{z_N} 
=  N \scalar{\r_N-\pi}{V_1^{-1}(\r_N-\pi)},$$ the following is an immediate corollary of Proposition \ref{prop:LCLT}.

\begin{corollary}\label{coro2}
Suppose $\mu\in\cP(X)$ is irreducible, and let $\r_N$ be an admissible sequence  
such that 
\begin{align}\label{hypo1}
%\lim_{N\to\infty}
\scalar{\r_N-\pi}{\r_N-\pi} =O(1/N),
\end{align}
where $\pi = \mu(\xi)$ is the vector of the expected values of $\mu$.
Then there exists a constant $c=c(\mu)>0$ such that for $N$ sufficiently large
\begin{align}\label{lim1}
\mu^{\otimes N}\left(\Om_{\r_N}\right)\geq \frac{c}{N^{K/2}}.
\end{align}  
In particular, %for such sequences $\r_N$ one has
\begin{align}\label{loglim}
\lim_{N\to\infty} \frac1N\,\log
\mu^{\otimes N}\left(\Om_{\r_N}\right)= 0.
\end{align}  
\end{corollary}
We note that the condition \eqref{hypo1} corresponds to ``normal" fluctuations $\scalar{z_N}{z_N}  = O(1)$, and that the corollary applies, in particular, to the canonical sequence $\rho_N=\r^\pi$ defined in \eqref{rhoseq}, since $\scalar{\r^\pi-\pi}{\r^\pi-\pi} =O(1/N^2)$ in that case. 

The proof of Proposition \ref{prop:LCLT} will be given in the appendix. Here we pause for some remarks on the assumptions we made, and then discuss the main applications to chaos. 

\begin{lemma}\label{lem:nondeg}
If $\mu$ is irreducible then it is nondegenerate. The converse does not hold.
\end{lemma}
\begin{proof}
If $\mu$ is degenerate, then for any fixed $(i,x)$, the variable $\xi_{i,x}$ can be written $\mu$-a.s.\ as a nontrivial linear combination of the other variables $\xi_{j,y}$, $(j,y)\neq (i,x)$. In particular, the value of $\xi_{i,x}$ is $\mu$-a.s.\ determined by the other variables. But this is not possible if $\mu$ is irreducible since by assumption there is always at least one value of all the other variables for which both values $\xi_{i,x}=0,1$ happen with positive $\mu$ probability. This proves the first assertion. To violate the converse, consider the following example: $n=3$, $q_i\equiv1$, so that $X=\{0,1\}^3$ and suppose that $\mu$ gives probability $1/4$ to the following four configurations $101,110,011,000$, and probability 0 to the four remaining configurations. Then one checks that $V_1=\frac14\ind_3$. In particular, $\mu$ is nondegenerate.  However, $\mu$ is not irreducible since the condition in Definition \ref{def:irr} is violated at $i=1$. %$\xi_1$ is fully determined when $\xi_2,\xi_3$ are given. 
\end{proof}

Let us remark that some irreducibility assumption is necessary for the local CLT statement in Proposition \ref{prop:LCLT}. Consider the same counterexample from the proof of Lemma \ref{lem:nondeg}. In this case one checks easily that if the first component of $S_N$ is even, then the sum of the remaining two components must be even as well. This shows that the event $S_N=M_N$ has probability zero for  many admissible sequences such that Corollary \ref{coro2} %the estimate \eqref{eq:LCLT} 
would predict $\mu^{\otimes N}\left(S_N=M_N\right) >0$. Thus, Proposition \ref{prop:LCLT} does not hold for all nondegenerate $\mu$. The next lemma elucidates the role of the irreducibility assumption.   
\begin{lemma}\label{charfunc}
Suppose $\mu$ is irreducible.  
Then there exists a constant $c=c(\mu)>0$ such that the characteristic function $\psi(t)=\mu(e^{i\scalar{t}{\xi}})$, $t\in\bbR^K$, satisfies 
\begin{align}
|\psi(t)|\leq 
e^{-c
\scalar{t}{t}
}\,,\qquad \text{for all\;\;} %$t$ that satisfies 
t\in[-\pi,\pi]^K.
\end{align}
\end{lemma}
\begin{proof}%[Proof of lemma \ref{charfunc}]
We write
\begin{align}
|\psi(t)|^2 = \left|\mu\left[e^{i\scalar{t}{\xi}}\right]\right|^2
&=\mu\left[\cos\scalar{t}{\xi}\right]^2 + \mu\left[\sin\scalar{t}{\xi}\right]^2 
%\\ &
=\sum_{\xi,\xi' \in X}\mu(\xi)\mu(\xi')\cos\scalar{t}{\xi-\xi'},
\end{align}
where the last equation uses the identity $\cos(\alpha-\beta) = \sin(\alpha)\sin(\beta) + \cos(\alpha)\cos(\beta).$
If $|\theta|\leq \pi$, then $\cos(\theta)\leq 1-2\theta^2/\pi^2$, and therefore, 
%Note that, if $|t|_{\infty}\leq \pi,$ then $<t,\vec{x}-\vec{y}> \leq \pi|\vec{x}-\vec{y}|_1,$ so we can estimate
%
\begin{equation}
\cos\scalar{t}{\xi-\xi'} \leq
\begin{cases}
1-\frac{2\scalar{t}{\xi-\xi'}^2}{\pi^2} & \text{if} \ |\xi-\xi'|_1 = 1 \\
1 & \text{if} \ |\xi-\xi'|_1 \neq 1 
\end{cases}
\end{equation}
where $|\xi-\xi'|_1=\sum_{i,x}|\xi_{i,x}-\xi_{i,x}'|$. Since $|\xi-\xi'|_1=1$ iff $\xi'=\xi^{(i,x)}$ for some $i,x$,  
\begin{align}
|\psi(t)|^2
&\leq 1-\frac{2}{\pi^2}\sum_{i,x}\sum_{\xi\in X}\mu(\xi)\mu(\xi^{(i,x)})\,t_{i,x}^2 
\leq 1-2c\scalar{t}{t},
\end{align}
where $$c:=\frac1{\pi^2} \;\inf_{i,x}\sum_{\xi\in X}\mu(\xi)\mu(\xi^{(i,x)})
.$$
The irreducibility of $\mu$ is equivalent to $c>0$. 
Using $x\leq e^{\frac{1}{2}(x^2-1)}$,  $x \in [0,1]$, with $x=|\psi(t)|$, we conclude
\begin{align}
|\psi(t)|\leq e^{-c\,\scalar{t}{t}}.
\end{align}
\end{proof}
We turn to the applications to Kac chaos and entropic chaos. 

\subsection{Kac chaos}
Recall the definition of $\g(p,\r_N)$ and of the canonical tensor product $\g_N(p)$ %the canonical tensor product $\g_N(p)$ 
in 
Definition \ref{def:canonical}. 
\begin{theorem}\label{th:eqens}
Suppose $p\in\cP(\O)$ is irreducible and let $\r_N$  be an admissible sequence  
such that 
\begin{align}\label{hypo10}
%\lim_{N\to\infty}
\scalar{\r_N-\pi}{\r_N-\pi} =O(1/N).
\end{align}
Then for all $k=1,\dots,N$,  
\begin{align}\label{eq:chaos00}
\|P_k\g(p,\r_N)-p^{\otimes k}\|_\tv \leq \frac{C\,k}{\sqrt N},
\end{align}
for some constant  $C=C(p)$. 
Moreover, when $\r_N=\rhoN$, the canonical tensor product $\g_N(p)$ satisfies 
the stronger estimate
\begin{align}\label{eq:chaos000}
\|P_k\g_N(p)-p^{\otimes k}\|_\tv \leq \frac{C\,k}{ N}.
\end{align}
\end{theorem}
\begin{proof}
We prove \eqref{eq:chaos00} first, and then show how to obtain \eqref{eq:chaos000}.
%We start with some notation. We write 
%$\xi_{i,x}(j)=\ind_{\si_{i,x}(j)=x}$ for the indicator that the particle $j$ has the letter $x$ at site $i$, and 
Let $\hat \xi_{i,x}(j)=\xi_{i,x}(j)-(\r_N)_{i,x}$.   
%Under $p^{\otimes N}$, the $\hat \xi_{i,x}(j)$, $j=1,\dots,N$, are i.i.d.\ random variables with covariance given by $V_1$,  a $K\times K$ matrix,  where $K=\sum_{i=1}^nq_i$, and 
%$$V_1(i,x;i',x')=p(\si_{i}=x,\si_{i'}=x')-\pi_{i,x}\pi_{i',x'}.$$
We use the shorthand notation $\g_N=\g(p,\r_N)$ and $\mu_N=p^{\otimes N}$. 
For any $f:\O^N\mapsto\bbR$ we have 
\begin{align}\label{eq:chao1}
\g_N(f)-\mu_N(f)=\frac{\mu_N(f(\ind_{\O_{\r_N}}-\mu_N(\O_{\r_N}))}{\mu_N(\O_{\r_N})},
\end{align}
Since $\O_{\r_N}=\{S_N=N\r_N\}$, using the Fourier transform we write 
\begin{align}\label{eq:chao2}
\mu_N(\O_{\r_N})=\frac1{(2\pi)^K}\int_{[-\pi,\pi]^K}d t \,\mu_N\left(e^{i\scalar{t}{\hat S_N}}\right),
\end{align}
where $\hat S_N=\sum_{j=1}^N\hat\xi(j)$.
%we use the notation $\scalar{t}{s}=\sum_{i=1}^n\sum_{x=1}^{q_i-1}t_{i,x}s_{i,x}$. 
Set $V_N:=N V_1$. The product structure of $\mu_N$ and the change of variables $s= V_N^{1/2} t=\sqrt N V_1^{1/2} t\,$  imply 
\begin{align}\label{eq:chao3}
\mu_N(\O_{\r_N})=\frac1{B_N(2\pi)^K}\int_{Q_{N,K}}d s \,\mu_N\left(e^{i\scalar{V_N^{-1/2} s}{\hat\xi(1)}}\right)^N,
\end{align}
where $Q_{N,K}=V_N^{1/2}[-\pi,\pi]^K$ and $B_N = \sqrt {\det V_N}=N^{K/2}\sqrt {\det V_1}$. 

In the same way, for any $f=f(\xi(1),\dots,\xi(k))$, we have %$P_k\g(p,N) (f) = \nu(f)$ is given by 
\begin{align}\label{eq:chao4}
\mu_N(f\,\ind_{\O_{\r_N}})=\frac1{B_N(2\pi)^K}\int_{Q_{N,K}}d s \,\mu_N\left(e^{i\scalar{V_N^{-1/2} s}{\hat\xi(1)}}\right)^{N-k}\mu_N\left(f \,e^{i\scalar{V_N^{-1/2} s}{\hat S_k}}\right).
\end{align}
In conclusion, we have
\begin{align}\label{eq:chao5}
\g_N(f)-\mu_N(f)=\frac{\int_{Q_{N,K}}d s \,\psi_N(s)^{N-k}\mu_N\left(f; \,e^{i\scalar{V_N^{-1/2} s}{\hat S_k}}\right)}{
\int_{Q_{N,K}}d s \,\psi_N(s)^{N}\,},
\end{align}
where $$\psi_N(s)=\mu_N\left(e^{i\scalar{V_N^{-1/2} s}{\hat\xi(1)}}\right),$$ and we use the notation $\mu_N(f; g)=\mu_N(fg)-\mu_N(f)\mu_N(g)$ for the covariance of $f,g$. From Corollary \ref{coro2} we known that \eqref{eq:chao3} is at least $cN^{-K/2}$, and thus the denominator in \eqref{eq:chao5} is at least some constant $c'>0$. 
Therefore, it suffices to show that the numerator is bounded by
  \begin{align}\label{eq:chao6}
\int_{Q_{N,K}}d s \,|\psi_N(s)|^{N-k}\left|\mu_N\left(f ;\,e^{i\scalar{V_N^{-1/2} s}{\hat S_k}}\right)\right|\leq C \,|f|_\infty\,\frac{k}{\sqrt N}. 
\end{align}
From Lemma \ref{charfunc} we know that $|\psi_N(s)|\leq e^{-a \scalar{s}{s}/N}$ for some constant $a=a(p)>0$. Notice that we can assume without loss of generality that  $k\leq N/2$, since otherwise the result \eqref{eq:chaos00} is trivial. Thus $|\psi_N(s)|^{N-k} \leq e^{-a \scalar{s}{s}/2}$ and it is sufficient to show that  
  \begin{align}\label{eq:chao7}
\left|\mu_N\left(f; \,e^{i\scalar{V_N^{-1/2} s}{\hat S_k}}\right)\right|\leq C \,|f|_\infty\,\frac{k}{\sqrt N}\,\scalar{s}{s},
\end{align}
for all $s\in Q_{N,K}$. Recalling that $|e^{i\theta}-1|\leq |\theta|$, $\theta\in\bbR$, and using Schwarz' inequality, %from \eqref{eq:chao9} we have
\begin{align}\label{eq:chao71}
|\mu_N\left(f; \,e^{i\scalar{V_N^{-1/2} s}{\hat S_k}}
\right)|
\leq \frac1{\sqrt N} \,|f|_\infty\,
\mu_N\left(\scalar{V_1^{-1/2} s}{\hat S_k}^2\right).
%%\,\frac{k}N\,\scalar{s}{s},
\end{align}
Now we observe that 
\begin{align}\label{eq:chao72}
\mu_N\left(\scalar{V_1^{-1/2} s}{\hat S_k}^2\right) &= k^2\scalar{V_1^{-1/2} s}{\pi-\r_N}^2 + k \scalar{s}{s}\nonumber
\\
& \leq k^2\scalar{s}{V_1^{-1}s}\scalar{\pi-\r_N}{\pi-\r_N} +  k \scalar{s}{s} \leq Ck \scalar{s}{s}, 
\end{align}
where we use $\mu_N(\hat S_k)=k\,(\pi-\r_N)$, the independence of the  $\hat\xi(j)$, and  $\scalar{\pi-\r_N}{\pi-\r_N}\leq C/k$ which follows from the assumption $\scalar{\pi-\r_N}{\pi-\r_N}=O(1/N)$. 
This proves \eqref{eq:chaos00}.

To prove \eqref{eq:chaos000}, note that it is sufficient to prove \eqref{eq:chao7} with $\sqrt N$ replaced by $N$ in the right hand side. For this, we are going to use the fact that $\scalar{\pi-\r_N}{\pi-\r_N}=O(1/N^2)$ when $\r_N=\rhoN$, see \eqref{rhoseq}. 
 Let us first consider the function $\tilde f = f-g$, where 
  \begin{align}\label{eq:chao8}
g(\xi(1),\dots,\xi(k))=\frac1k\scalar{V_1^{-1}v}{\hat S_k} %= \sum_{i,x}v_{i,x}\hat \xi_{i,x}(1)
,\qquad v = \mu_N(f;\hat S_k).
\end{align}
The function $g$ can be seen as a linear approximation of $f$. Notice that 
  \begin{align}\label{eq:chao9}
\mu_N\left(f-g\, \,;\,\scalar{V_N^{-1/2} s}{\hat S_k}\right)=0.
%g(\xi(1))=\scalar{v}{\hat\xi(1)} = \sum_{i,x}v_{i,x}\hat \xi_{i,x}(1),\qquad v = \sum_{j=1}^k\mu_N(fV_1^{-1}\hat\xi(j)).
\end{align}
Indeed, by independence $\mu_N((\hat S_k)_{i,x};(\hat S_k)_{j,y}) = k V_1(i,x;j,y) $, and therefore for all $s$,
\begin{align}\label{eq:chao10}
\mu_N\left(g \,;\,\scalar{V_N^{-1/2} s}{\hat S_k}\right)
%\sum_{j=1}^k
%\mu_N\left(g \,;\,\scalar{V_N^{-1/2} s}{\hat\xi(j)}\right)&=
%\mu_N\left(g \,;\,\scalar{V_N^{-1/2} s}{\hat\xi(1)}\right)
%\nonumber \\& 
= \scalar{v}{V_N^{-1/2} s}=\mu_N\left(f\,;\,\scalar{V_N^{-1/2} s}{\hat S_k}\right).
\end{align}
Recalling that $|e^{i\theta}-1-i\theta|=|R(\theta)|\leq \frac12\theta^2$, $\theta\in\bbR$, from \eqref{eq:chao9} we have
\begin{align}\label{eq:chao71x}
&\left|\mu_N\left(f-g\,; \,e^{i\scalar{V_N^{-1/2} s}{\hat S_k}}
\right)\right|^2
 \leq \frac1{N^2}\,\left(\var(f) + \var(g)\right)\mu_N\left( \scalar{V_1^{-1/2} s}{\hat S_k}^4\right)\,,
\end{align}
where we use the inequality $|\mu_N(f-g;R)|^2\leq 2\left(\var(f) + \var(g)\right)\mu_N(R^2)$, and we use $\var$ for the variance w.r.t.\ $\mu_N$.  Next, observe that $\var(f)\leq |f|_\infty^2$ and 
$$
\var(g) = \frac1k \scalar{V^{-1}v}{v} \leq C \var(f) \leq C |f|_\infty^2.
$$ 
We are going to show that 
\begin{align}\label{eq:chao7a7}
\mu_N\left( \scalar{V_1^{-1/2} s}{\hat S_k}^4\right)
\leq C k^2\scalar{s}{s}^2\,.
\end{align}
Suppose for a moment that \eqref{eq:chao7a7} holds. Then we conclude that 
 \begin{align}\label{eq:chao7a}
\left|\mu_N\left(f-g\,; \,e^{i\scalar{V_N^{-1/2} s}{\hat S_k}}\right)\right|\leq C \,|f|_\infty\,\frac{k}{ N}\,\scalar{s}{s}.
\end{align}
%
%we claim that both terms in the right hand side above are bounded by $\frac{C\,k}{N} \,|f|_\infty$.
%For the first term observe that by \eqref{eq:chao72} one has
%\begin{align}\label{eq:chao7a1}
%\left|\mu_N\left( f\,;\scalar{V_1^{-1/2} s}{\hat S_k}^2\right)\right|
%\leq 2|f|_\infty\,\mu_N\left( \scalar{V_1^{-1/2} s}{\hat S_k}^2\right)\leq C\,k\, \,|f|_\infty.
%\end{align}
%Moreover, letting $a=V^{-1}v$, writing $\hat S_k = \hat \xi(1) + \hat S'_{k-1}$ and using \eqref{eq:chao72} for $\hat S'_{k-1}$ one finds
%\begin{align}\label{eq:chao7a2}
%\left|\mu_N\left( g\,;\scalar{V_1^{-1/2} s}{\hat S_k}^2\right)\right\|&\leq
%C\scalar{s}{s}\scalar{a}{a}^{1/2}\left(k\scalar{\r_N-\pi}{\r_N-\pi}^{1/2}+ 1\right)\nonumber\\
%& \leq C\scalar{s}{s}\scalar{a}{a}^{1/2} \leq C\, \sqrt k\, |f|_\infty\scalar{s}{s},
%\end{align}
%where we use the fact that for $\r_N=\rhoN$ one has $\scalar{\r_N-\pi}{\r_N-\pi}=O(1/N^2)$,
%and the easily verified estimate $\scalar{a}{a}\leq k|f|^2_\infty$.
%In conclusion, the first term in the right hand side of \eqref{eq:chao71} is bounded by $\frac{C\,k}{N} \,|f|_\infty$. To estimate the second term, we note that $|g|_\infty \leq Ck|f|_\infty$ and therefore it is sufficient to prove that 
%\begin{align}\label{eq:chao7a2}
%\mu_N\left( |\scalar{V_1^{-1/2} s}{\hat S_k}|^3\right)|&\leq
%\frac{C}{N^{3/2}}\scalar{s}{s}
%\end{align}
%where we have used \eqref{eq:chao72}.
This proves $|\g_N(\tilde f)-\mu_N(\tilde f)|\leq  C \,|f|_\infty\,k/N$. 
However, noting that $\g_N(g)=0$, we have
\begin{align}\label{eq:chaoa1}
|\g_N(f)-\mu_N(f)|\leq |\g_N(\tilde f)-\mu_N(\tilde f)|+|\mu_N(g)|.
\end{align}
The desired conclusion $|\g_N(f)-\mu_N(f)|\leq  C \,|f|_\infty\,k/N$ then follows from %\what aggiungere $|g|_\infty\leq  C \,|f|_\infty$ 
\begin{align}\label{eq:chaoap1}
|\mu_N(g)|\leq  |\scalar{V_1^{-1}v}{\pi-\rhoN}|\leq \scalar{V_1^{-1}v}{V_1^{-1}v}^{1/2}\scalar{\pi-\rhoN}{\pi-\rhoN}^{1/2} \leq \frac{C|f|_\infty k}{N},
\end{align}
where we use $\scalar{\pi-\rhoN}{\pi-\rhoN}=O(1/N^2)$, and $\scalar{V_1^{-1}v}{V_1^{-1}v}\leq Ck|f|_\infty$. 
%$$|(V_1^{-1}v)_{i,x}| \leq  \sum_{j=1}^k|\mu_N(f;(V_1^{-1}\hat\xi(j))_{i,x})|\leq k|f|_\infty\var_{\mu_N}((V_1^{-1}\hat\xi(1))_{i,x}) = k|f|_\infty
%(V_1^{-1})_{i,x;i,x}
%$$
%so that by the nondegeneracy assumption one has $\sum_{i,x}(V_1^{-1})_{i,x;i,x}\leq C$. 

Thus, it remains to prove \eqref{eq:chao7a7}. Notice that it is sufficient to prove 
\begin{align}\label{eq:chao7a8}
\mu_N\left( \scalar{\hat S_k}{\hat S_k}^2\right)
\leq C k^2\,.
\end{align}
We write $\tilde \xi %= \xi - \mu(\xi) 
= \hat \xi - \mu(\hat \xi)$ and $\tilde S_k=\hat S_k - \mu_N(\hat S_k)$ for the corresponding sums. 
Then one checks that 
  \begin{align}\label{eq:chao7a9}
 &\scalar{\hat S_k}{\hat S_k}^2 \leq \scalar{\tilde S_k}{\tilde S_k}^2 + Ck\scalar{\tilde S_k}{\tilde S_k}^{3/2}\scalar{\pi-\rhoN}{\pi-\rhoN}^{1/2}+ Ck^2\scalar{\tilde S_k}{\tilde S_k}\scalar{\pi-\rhoN}{\pi-\rhoN}   \nonumber \\
 & \qquad \qquad \qquad + Ck^3\scalar{\tilde S_k}{\tilde S_k}^{1/2}\scalar{\pi-\rhoN}{\pi-\rhoN}^{3/2} + Ck^4\scalar{\pi-\rhoN}{\pi-\rhoN}^2\,.
\end{align}
Since $\scalar{\pi-\rhoN}{\pi-\rhoN}=O(1/N^2)$ we can restrict to prove \eqref{eq:chao7a8} for $\tilde S_k$ instead of $\hat S_k$. Since $\tilde \xi$ are centered the estimate follows easily by expanding $\scalar{\tilde S_k}{\tilde S_k}^2$ and observing that the dominant terms are of the form $\scalar{\tilde\xi(i)}{\tilde\xi(j)}^2$, and their  contribution is of order $k^2$.  \end{proof}

We notice that the trick of replacing $f$ by $f-g$ in the proof of Theorem \ref{th:eqens} allowed us to obtain the decay rate $O(1/ N)$ instead of $O(1/ \sqrt N)$. This idea was used in \cite{asymm} for the proof of a related ``equivalence of ensembles" result. 

\begin{corollary}\label{coro3} 
For any irreducible $p\in\cP(\O)$, if $\eta\in\O^N$ is distributed according to the canonical tensor product $\g_N(p)$, letting $\l_{\eta}=\frac1N\sum_{j=1}^N\d_{\eta(j)}$ denote the corresponding empirical measure, 
\begin{align}\label{eq:empit}
\bbE\Big[\|p-\lambda_{\eta}\|_{\rm TV}\Big] \leq \frac{C}{\sqrt {N}},
\end{align}
for some constant $C=C(p)$. Moreover, the same estimate holds with $\sqrt N$ replaced by $N^{1/4}$ if $\eta $ is distributed according to $\g(p,\r_N)$, for any admissible sequence $\r_N$ satisfying \eqref{hypo1}. 
\end{corollary}

\begin{proof}
We write
\begin{align}\label{eq:empit3}
\bbE\left[\|p-\lambda_{\eta}\|_{\rm TV}\right] & = \frac12\sum_{\si\in\O}\bbE\left[|\lambda_{\eta}(\si)-p(\si)|\right]\nonumber\\&
\leq \frac12\left(\sum_{\si\in\O}p(\si)
\bbE\left[(h_\eta(\si)-1)^2\right]\right)^{1/2}\,,
\end{align}
where  
$
h_\eta(\si) =\frac{\lambda_{\eta}(\si)}{p(\si)}
$ and we have used Schwarz' inequality for the product measure $p\times \bbE$.
Now,
\begin{align}\label{eq:empit1}
\sum_{\si}p(\si)
\bbE\left[(h_\eta(\si)-1)^2\right]
& = -1 + \frac1N
\,\bbE\left[\frac{1 }{p(\eta(1))}\right] 
%\sum_\si \frac{\bbP(\eta(1)=\si)}{p(\si)} 
+ \frac{N(N-1)}{N^2}
\,
\bbE\left[\frac{\ind_{\eta(1)=\eta(2)} }{p(\eta(1))}\right] ,
\end{align}
where we use 
$$
\bbE\left[\frac{1 }{p(\eta(1))}\right]=\sum_\si \frac{\bbP(\eta(1)=\si)}{p(\si)} \,,
\quad \bbE\left[\frac{\ind_{\eta(1)=\eta(2)} }{p(\eta(1))}\right]  = \sum_\si \frac{\bbP(\eta(1)=\eta(2)=\si)}{p(\si)}\,. 
%\leq \frac1{p_*}\,,
$$
Since
$\bbE\left[1/p(\eta(1))\right]\leq 1/p_*$,
where $p_*=\min_{\si:\, p(\si)>0}p(\si)$ we see that the second term in \eqref{eq:empit1} is bounded by $1/(p_*N)$. 
Next, consider the function 
$f(\eta(1),\eta(2))=\ind_{\eta(1)=\eta(2)}/p(\eta(1))$. Note that %$|f|_\infty\leq 1/p_*$ and that 
$p^{\otimes 2}(f) = 1$. Then, Theorem \ref{th:eqens}  implies 
$$
\left|-1+\sum_\si\frac{\bbP(\eta(1)=\eta(2)=\si) }{p(\si)}\right|\leq \frac{C}{p_* N}.
$$  
Since $N(N-1)/N^2=1-1/N$, we have shown that \eqref{eq:empit1} is bounded by $C/N$
for some new constant $C$ depending only on $p$. Together with \eqref{eq:empit3} this concludes the proof of \eqref{eq:empit}. Finally, if instead $\eta $ is distributed according to $\g(p,\r_N)$, for an arbitrary admissible sequence $\r_N$ satisfying \eqref{hypo1} we may repeat all the steps above and use the first part of Theorem \ref{th:eqens} to conclude that \eqref{eq:empit1} this time is bounded by $C/\sqrt N$, which implies the claimed bound with $N^{1/4}$ in place of $\sqrt N$.
\end{proof}

\subsection{Entropic chaos and Fisher chaos}
The next result shows how to use the local CLT to obtain convergence of the relative entropy of tensor products. Following \cite{entfound} we refer to this as entropic chaos. 
\begin{proposition}\label{th:entchaos}
Suppose $p\in\cP(\O)$ is irreducible and let $\r_N$  be an admissible sequence  
such that 
\begin{align}\label{hypoa1}
%\lim_{N\to\infty}
\scalar{\r_N-\pi}{\r_N-\pi} =O(1/N).
\end{align}
Then
\begin{align}\label{eq:entchaos}
\lim_{N\to\infty}
\frac1N\,H_N(\g(p,\r_N)\tc\g(\pi,\r_N))=H(p\tc \pi).
\end{align}
\end{proposition}
\begin{proof}
We write
\begin{align}\label{eq:entchaos2}
H_N(\g(p,\r_N)\tc\g(\pi,\r_N))=\g(p,\r_N)\left[\log\left(\frac{p^{\otimes N}}{\pi^{\otimes N}}\right)\right]
+ \log\left(\frac{\pi^{\otimes N}(\O_{\r_N})}{p^{\otimes N}(\O_{\r_N})}\right).
\end{align}
From Corollary \ref{coro2} we obtain
\begin{align}\label{eq:entchaos3}
\lim_{N\to\infty}
\frac1N\, \log\left(\frac{\pi^{\otimes N}(\O_{\r_N})}{p^{\otimes N}(\O_{\r_N})}\right)=0.
 \end{align}
By symmetry, the first term in \eqref{eq:entchaos2} equals 
$$
N\,P_1\g(p,\r_N)\left[\log\left(\frac{p}{\pi}\right)\right].
$$
Therefore the result follows from Theorem \ref{th:eqens}.
\end{proof}
Another consequence of the local CLT is the following upper semi-continuity property, see \cite{entfound} for a similar statement in the kinetic setting.
\begin{proposition}\label{th:entUPSC}
For each $N,$ let $\mu^{(N)}$ be a symmetric probability  on $\Om_{\r_N}$ and let $\mu_k$ be a probability on $\Om^k$ such that $$P_k\mu^{(N)} \longrightarrow \mu_k$$ weakly for some integer $k$. %Let $\g\in\cP(\O)$ be any irreducible measure with marginals $\g_i=\pi_i$ for all $i \in [n].$ 
Then, for any admissible sequence satisfying \eqref{hypo1}, %the following inequality holds:
\begin{align}
\frac{H(\mu_k\tc\pi^{\otimes k})}{k} \leq \liminf_{N \to \infty}\frac{H_N(\mu^{(N)}|\gamma(\pi,\r_N))}{N}.
\end{align}
\end{proposition}
\begin{proof}
From Shearer inequality one has
\begin{align}
\sum_{A\in\cA}H\left(P_A\mu^{(N)}|P_A\pi^{\otimes N}\right) \leq n_+(\cA)H\left(\mu^{(N)}\tc\pi^{\otimes N}\right),
\end{align}
where $\cA$ is any family of sets covering $[N]=\{1,\dots,N\}$, $n_+(\cA) = \max_{j} \#\{A\in\cA:\,A\ni j\}$, and $P_A\nu$ denotes the marginal on the variables $\{\eta(j),\,j\in A\}$ of a probability $\nu\in\cP(\O^N)$, see Lemma \ref{lem:shearer} below for more details. 
If we take $\cA=\{A\subset [N]:\,|A|=k\}$, then $n_+(\cA) =\binom{N-1}{k-1}$. Moreover, by symmetry 
$P_A\mu^{(N)}=P_k\mu^{(N)}$ for all $A\in\cA$.  Since $\binom{N}{k}/\binom{N-1}{k-1}=N/k$, this proves 
\begin{align}\label{eq:lhs}
H\left(P_k\mu^{(N)}\tc\pi^{\otimes k}\right) \leq \frac{k}N\,H\left(\mu^{(N)}\tc\pi^{\otimes N}\right).
\end{align}
See also Lemma 3.9 in \cite{miclo}, where \eqref{eq:lhs} was derived with $\frac{k}{N}$ replaced by $\frac{k}{N}(1+O(\frac{k}{N}))$  in the right hand side. 
On the other hand, 
\begin{align}\label{eq:entchaos4}
H_N\left(\mu^{(N)}\tc\pi^{\otimes N}\right)
= H_N(\mu^{(N)}\tc\g(\pi,\r_N))
- \log\left(\pi^{\otimes N}(\O_{\r_N})\right).
\end{align}
The left hand side of \eqref{eq:lhs} converges by assumption to $H\left(\mu_k\tc\pi^{\otimes k}\right)$. The desired conclusion then follows from Corollary \ref{coro2}.
 \end{proof}

\begin{remark}\label{rem:anyp}
Both Proposition \ref{th:entchaos} and Proposition \ref{th:entUPSC} hold with $\pi$ replaced by any irreducible $p'\in\cP(\O)$ with the same marginals as $\pi$.
\end{remark}

One can also establish the following analogue of the Fisher chaos property discussed in \cite{onkacschaos}.
Observe that if $\mu_{N,t}=\mu_N e^{t\cL_N}$, for some $\mu_N\in\cP(\O_{\r_N})$, then
\begin{align}\label{eq:fish1}
\frac{d}{dt}\, H\left(\mu_{N,t}|\g(\pi,\r_N)\right)_{\vert_{t=0^+}} = -D_N(f_N)\,,
\end{align}
where $f_N:=\mu_N/\g(\pi,\r_N)$, and 
\begin{align}\label{eq:fish2}
D_N(f_N) = \frac1{2N}\sum_{j<l}\sum_A\nu(A)\,\g(\pi,\r_N)\left[(f_N^{j,l,A}-f_N)\log\frac{f_N^{j,l,A}}{f_N}\right]\,.
\end{align}
Here we use the notation $f_N^{j,l,A}(\eta)=f_N(\eta^{j,l,A})$, where $\eta^{j,l,A}$ is defined in \eqref{def:etajl}.
%Consider the Fisher information of $f_N:\O_{\r_N}\mapsto\bbR_+$ with respect to $\cE$ from \eqref{} and the entropy production $D_\pi(f)$ from \eqref{}.
On the other hand, for the nonlinear equation \eqref{boltzeq} one has 
\begin{align}\label{eq:fish3}
\frac{d}{dt}\, H\left(p_t\tc\pi\right)_{\vert_{t=0^+}} = -D_\pi(f)\,,
\end{align}
where $f:=p/\pi$, and 
\begin{align}\label{eq:fish4}
D_\pi(f) =\frac12 \sum_A\nu(A)\,\pi\left[(f^Af^{A^c}-f)\log\frac{f^Af^{A^c}}{f}\right]\,,
\end{align}
where $f^A(\si)=p_A(\si)/\pi_A(\si)$ is the density of the marginal of $p$ on $A$ with respect to $\pi$.

\begin{proposition}\label{dff2tyhm}
Under the same assumptions of Proposition \ref{th:entchaos},   
\begin{equation}\label{dff2}
\lim_{N\to\infty}\frac{D_N(f_N)}{N} = D_\pi(f), 
\end{equation}
where $f_N:=\gamma(p,\r_N)/\g(\pi,\r_N)$, and $f:=p/\pi$. 
\end{proposition}
%
%\begin{remark}
%If we consider $f_N = p_N/\pi_N,$ where $p_N$ is a symmetric probability measure on $\Om_{\rho_N}$ such that $P_1p_N \to p,$ then it is clear that equation \eqref{dff2} can be seen as a discrete analog of the Fisher information chaos defined e.g. in \cite{onkacschaosonkacschaos}. The previous lemma namely allows us to prove that $\gamma(p,N)$ is Fisher chaotic, if $p$ has the right marginals. 
%\end{remark}
%
\begin{proof}%[Proof of Lemma \ref{dff2tyhm}]
Since 
\begin{align}\label{eq:fisha2}
D_N(f_N) = -\frac1{N}\sum_{j<l}\sum_A\nu(A)\,\g(\pi,\r_N)\left[(f_N^{j,l,A}-f_N)\log{f_N}\right]\,.
\end{align}
and 
\begin{align}\label{eq:fisha4}
D_\pi(f) = - 2\sum_A\nu(A)\,\pi\left[(f^Af^{A^c}-f)\log{f}\right]\,,
\end{align}
it suffices to show that for every $j,l \in [N]$ and $A \subset [n]$ one has
\begin{align}
\lim_{N \to +\infty}&\sum_{\eta \in \Om_{\rho_N}}\g(\pi,\r_N)(\eta)\left(f_N(\eta^{j,l,A})-f_N(\eta)\right)\log\left(%\frac{f_N(\eta_{i,j,A})}
{f_N(\eta)}\right)  \\  
&\qquad = 2\sum_{\sigma \in \Om}(p_A(\sigma_A)p_{A^c}(\sigma_{A^c})-p(\sigma))\log\left(\frac{p(\sigma)}{\pi(\sigma)}\right).
\end{align}
Note that
\begin{align}
& \sum_{\eta \in \Om_{\rho_N}}\pi_N(\eta)\left(f_N(\eta_{j,l,A})-f_N(\eta)\right)\log\left(%\frac{f_N(\eta_{i,j,A})}
{f_N(\eta)}\right) \\
&=\sum_{\eta \in \Om_{\rho_N}} %\left(\right)
\frac{\prod_{\ell \neq j,l}p(\eta(l))}{p^{\otimes N}\left(\Om_{\rho_N}\right)}\left(
p(\eta^{j,l,A}(j))p(\eta^{j,l,A}(l))
%p(\eta_A(j)\eta_{A^c}(l)p(\eta_A(l)\eta_{A^c}(j))
-p(\eta(l))p(\eta(j))\right)\log\prod_{r=1}^N\frac{p(\eta(r))}{\pi(\eta(r))} \\
&=\sum_{\eta(l),\eta(j) }\frac{p^{\otimes (N-2)}\left(\Om^{\eta(j),\eta(l)}_{\rho_N}\right)}{p^{\otimes N}\left(\Om_{\rho_N}\right)}\left(p(\eta^{j,l,A}(j))p(\eta^{j,l,A}(l))
%\eta_A(j)\eta_{A^c}(l)p(\eta_A(l)\eta_{A^c}(j))
-p(\eta(l))p(\eta(j))\right)\log\frac{p(\eta(l))p(\eta(j))}{\pi(\eta(l))\pi(\eta(j))}%\frac{p(\eta(j))}{\pi(\eta(j))}
,
\end{align}
where $\Om^{\eta(j),\eta(l)}_{\rho_N}$ is the event 
\begin{align}
\sum_{\ell\neq j,l} \ind(\eta_i(\ell)=x) = (N-2)(\widetilde{\rho_N})_{i,x}\,,\quad  \forall i,x
\end{align}
and, for any fixed $\eta(j),\eta(l)$,  $\widetilde{\rho_N}$ denotes the density 
$$
(\widetilde{\rho_N})_{i,x} = \frac{N}{N-2}\, (\r_N)_{i,x} -  \frac{1}{N-2}\left(\ind(\eta_i(j)=x)+\ind(\eta_i(l)=x)
\right)\,.
$$
Therefore, it is sufficient to show that
\begin{align}\label{clad}
\frac{p^{\otimes (N-2)}\left(\Om^{\eta(j),\eta(l)}_{\rho_N}\right)}{p^{\otimes N}\left(\Om_{\rho_N}\right)} \;\to\; 1,
\end{align}
for all $i,j \in [N]$, for all fixed values of $\eta(j),\eta(l)\in\O$. 
We note that both $\rho_N,\widetilde{\rho_N}$ satisfy condition \eqref{hypo1}. Moreover, writing 
%using $\widetilde\r_N = (1+O(1/N))\r_N + O(1/N^2)$  
\begin{gather*}
z_N =\sqrt N \;V_1^{-1/2}\left(\r_N - \pi\right),\quad 
\widetilde z_N=\sqrt N \;V_1^{-1/2}\left(\widetilde\r_N -\pi\right), 
\end{gather*}
we see  that $$
\scalar{z_N}{z_N} -\scalar{\widetilde z_N}{\widetilde z_N} \to 0\,,\quad N\to\infty,
$$
for all fixed values of $\eta(j),\eta(l)$. The desired claim \eqref{clad} then follows from Proposition \ref{prop:LCLT}.
\end{proof}

\section{Uniform in time propagation of chaos}
%: Proof of Theorem \ref{unifpropchaos}}
\label{sec:unifintime}
The main goal in this section is to prove Theorem \ref{th:unifintime}. 
We first introduce some notation. Let $\lambda_{\eta}\in \cP(\O)$ be the empirical measure $\lambda_{\eta}:=\frac1N\sum_{i=1}^N\delta_{\eta(i)}$, where $\eta:=\left(\eta(1),\dots,\eta(N)\right)\in\O^N$.
We consider the Wasserstein distance on $\cP(\O)$ associated to the Hamming distance on $\O$: 
\begin{align}\label{def:W}
W(p,q):= \inf_{\G \in \Pi(p,q)}\sum_{\si,\si' \in \Om}\G(\si,\si')\sum_{i=1}^n\ind_{\si_i\neq\si'_i},
\end{align}
where $p,q\in\cP(\O)$, and $\Pi(p,q)$ denotes the set of all couplings $\G$ of $p$ and $q$.
We remark that the total variation distance $\|p-q\|_\tv$ is defined as above with $\sum_{i=1}^n\ind_{\si_i\neq\si'_i}$ replaced by $\ind_{\si\neq\si'}$ and thus one has
\begin{align}\label{eq:tvW}
\|p-q\|_\tv\leq W(p,q) \leq n\,\|p-q\|_\tv.
\end{align}
In contrast with the total variation distance, the distance $W$ has a convenient monotonicity along the evolution, see Lemma \ref{wasscontrlemma} and Remark \ref{rem:mono} below. Recall the definition of the non-separation probability $r(\nu)$ from \eqref{def:rnu}.

\begin{theorem}\label{unifpropchaos}
Assume that $\mu_{N}\in\cP(\O^N)$ is $p-$chaotic and the measure $\nu$ is separating. Then for all $k \in \bbN$ and any function of the form $\phi_k:= \phi^1\otimes\dots\otimes\phi^k:\O^k\mapsto\bbR$, where $ \phi^i:\Om\mapsto \bbR$, and such that $\|\phi_k\|_{\infty}\leq 1$, the following inequality holds
\begin{align}\label{propchaosuniftime}
|P_k\mu_{N,t}(\phi_k)-p_t^{\otimes k}(\phi_k)|&\leq  \frac{2k(k-1)}{N} + \frac{4k^2n^5}{D(\nu)\,N} (1-e^{-D(\nu)\, t}) 
+ 2 k\,
%n^2B\,e^{-Dt}
\mu_N\left[W\left(p,\lambda_{\chi_N}\right)\right],
\end{align} 
where $\chi_N\in\O^N$ has distribution $\mu_{N}$, and
%$$C:=\left(2\frac{\frac{1+r(\nu)}{r(\nu)}}{\left(\left(\frac{1+r(\nu)}{2r(\nu)}\right)^{\frac{2}{1-r(\nu)}}-1\right)^{\frac{1-r(\nu)}{2}}}\right)^2, \ \ \ \ \  
$D(\nu):=1-r(\nu)$. 
\end{theorem}
The proof of Theorem \ref{unifpropchaos} follows the steps of a general approach, the so called ``abstract theorem'', see \cite{grunbaumchaos,kacsprog,anewapp}, see also \cite{revchaos} for a review. 
%and that has been rigorously implemented in the kinetic setting in \cite{kacsprog, anewapp}, see also \cite{revchaos} for a review. 
In our discrete setting many aspects of this approach take a simpler form, and no extra assumption on the initial distribution is needed besides the $p$-chaoticity for some $p\in\cP(\O)$. This general approach  requires however several model-specific inputs.  Our main original contribution here consists in establishing 
the key estimates stated in Theorem \ref{contrlemma2} and Theorem \ref{wasscontrlemma2} below, where we prove new contraction inequalities for both the non linear and the linearized evolutions associated to \eqref{boltzeq}. 
Before proving Theorem \ref{unifpropchaos}, let us show that it implies Theorem \ref{th:unifintime}.

 \subsection{Proof of Theorem \ref{th:unifintime}}
 It is well known  that $$\mu_N\left[W\left(p,\lambda_{\chi_N}\right)\right] \to 0$$ if and only if $\mu_{N}$ is $p$-chaotic; see e.g.\ %The implication from right to left follows e.g. from 
Lemma 3.34 in \cite{revchaos}. This, combined with the estimate from Theorem \ref{unifpropchaos}, proves the uniform in time propagation of chaos asserted in Theorem \ref{th:unifintime}. To prove the quantitative statement \eqref{eq:chaos1}, it suffices  to show that when $\mu_{N}=\gamma_N(p,\r_N)$, then 
\begin{align}\label{eq:empi}
\mu_N\left[W\left(p,\lambda_{\chi_N}\right)\right] \leq \frac{C_1}{\sqrt {N}},
\end{align}
where $C_1$ is a constant depending on $p$ and $n$. 
This statement follows from Corollary \ref{coro3} by noting that  $W(p,q)\leq n\|p-q\|_{\rm TV}$, for any $p,q\in\cP(\O)$, see \eqref{eq:tvW}. We remark that as in Corollary \ref{coro3} one has the same estimate with $\sqrt N$ replaced by $N^{1/4}$ if instead of $\g_N(p)$ we take $\g(p,\r_N)$ with an arbitrary admissible $\r_N$ satisfying \eqref{hypo1}.

We turn to the proof of Theorem \ref{unifpropchaos}. 
We start with some preliminary facts. 

\subsection{Wild sums and McKean trees}
Given $f, g : \Om \rightarrow \bbR,$ we adopt the notation
\begin{align}
&\label{defconv1} (f \circ g)^{A}:= \frac{1}{2}\left(f_A \otimes g_{A^c}+ g_A \otimes f_{A^c} \right), \\
&\label{defconv2}  f \circ g := \sum_{A \subset [n]}\nu(A)(f \circ g)^{A},
\end{align}
where $f_A(\si_{A}) :=  \sum_{\si_{A^c}}f(\si_A\si_{A^c})$.
Then the nonlinear equation \eqref{boltzeq} is given by 
\begin{align}\label{boltzeqc}
\frac{d}{dt}\,p_t = p_t\circ p_t - p_t .
\end{align}
The convolution product defined by $f\circ g$ is commutative and distributive, but not associative.  
Following Wild's original construction \cite{Wild} we write
the solution of \eqref{boltzeqc} with initial datum $p_0=p$, as
\begin{align}\label{wildsums}
p_t = e^{-t}\sum_{k=1}^{\infty}(1-e^{-t})^{k-1}p^{(k)},
\end{align}
where %, given $p \in \cP(\O),$ 
the  $\{p^{(k)}\}_{k \geq 1}$ are probability measures on $\Om$ defined inductively by
\begin{align}\label{eq:wildp}
p^{(1)}=p, \qquad 
p^{(k)} = \frac{1}{k-1}\sum_{j=1}^{k-1}p^{(j)}\circ p^{(k-j)}\,,\quad k\geq 2.
\end{align}
The validity of \eqref{wildsums} can be easily checked by direct inspection, see e.g.\ the argument after \eqref{eq:solve} below for a similar computation. 

Moreover, following McKean \cite{mckean}, we may express $p^{(k)}$ as
a weighted sum over rooted binary trees $\g$ with $k$ leaves. We now recall the details of this construction, and refer the reader to \cite{carlenwild} for further background. Let $\G(k)$ denote the set of rooted binary trees with $k$ leaves and call $\alpha_k(\gamma)$, $\g\in\G(k)$,  the probability over $\G(k)$ obtained by the following procedure. $\G(1)=\{\g_1\}$ is just the empty tree with only the root with $\a_1(\g_1)=1$, and $\G(2)=\{\g_2\}$ where $\g_2$ is the unique tree obtained by adding two children to the root, with $\a_2(\g_2)=1$.  Then, recursively, for $k\geq 2$, for any $\g_{k-1}\in\G(k-1)$, consider all possible trees $\g_{k-1}^i\in\G(k)$, $i=1,\dots,k-1$,   obtained by adding two children to the $i$-th leaf of $\g_{k-1}$, and for any $\g\in\G(k)$,  set 
\[
\a_k(\g)=\sum_{\g_{k-1}\in\G(k-1)}\a_{k-1}(\g_{k-1})\sum_{i=1}^{k-1} \frac{\ind(\g=\g_{k-1}^i)}{k-1}\,.
\]
%
%$\a_k(\g_k^i)=\frac1{k-1}\a_{k-1}(\g_{k-1})$ for each one of them. 
This defines a probability $\a_k$ on $\G(k)$, for any $k\in\bbN$, and one checks, recursively, that %$\sum_{\g\in\G(k)}\a_k(\g)=1$ for all $k\geq 1$ and that 
for all $k\geq 2$,
\begin{align}\label{eqreca}
\a_k(\g)=\frac1{k-1}\a_j(\g_l)\a_{k-j}(\g_r),
\end{align}
where $\g_l$ and $\g_r$ denote respectively the subtree of $\g$ rooted at the left child of the root and
the subtree of $\g$ rooted at the right child of the root, while $j$ denotes the number of leaves in $\g_l$.  
 Then,
by induction over $k$ it follows that for all $k\in\bbN$,
\begin{align}\label{eq00}
p^{(k)} = \sum_{\g\in \Gamma(k)}\alpha_k(\g)C_{\gamma}(p),
\end{align}
where $C_{\gamma}(p)\in\cP(\O)$ is described as follows. Each internal node of $\g$ represents a collision and the tree $\g$ describes the collision history. Then $C_\g(p)$ represents the distribution obtained at the root after all collisions from $\g$ have been performed, starting with the distribution $p$ at each leaf of $\g$.

%%%%FIGURA%%%%
\begin{figure}
\center

\begin{subfigure}

\begin{tikzpicture}[scale=0.8]
    
    \draw  (1,-1) -- (2,1);
    \draw  (3,-1) -- (2,1);
   \draw  (2,1) -- (4,3);
    \draw  (5,-1) -- (6,1);
    \draw  (7,-1) -- (6,1);
    \draw  (6,1) -- (4,3);
    
    \node[shape=circle, draw=black, fill = white, scale = 1.5]  at (4,3) {}; 
    \node[shape=circle, draw=black, fill = white, scale = 1.5]  at (2,1) {};
    \node[shape=circle, draw=black, fill = gray, scale = 1.5]  at (1,-1) {};
    \node[shape=circle, draw=black, fill = gray, scale = 1.5]  at (3,-1) {};
    \node[shape=circle, draw=black, fill = white, scale = 1.5]  at (6,1) {};
    \node[shape=circle, draw=black, fill = gray, scale = 1.5]  at (5,-1) {};
    \node[shape=circle, draw=black, fill = gray, scale = 1.5]  at (7,-1) {};
    
    \node at (1/4,-1) {$p$};
    \node at (9/4,-1) {$p$};
    \node at (17/4,-1) {$p$};
    \node at (25/4,-1) {$p$};
    \node at (1,1) {$p \circ p$};
    \node at (20/4,1) {$p \circ p$};
    \node at (8/4,3) {$(p\circ p) \circ (p \circ p)$};

\end{tikzpicture}
\end{subfigure}
\hspace{2 cm}
\begin{subfigure}

\begin{tikzpicture}[scale=0.8]

    \draw  (1,-3) -- (4,3);
    \draw  (3,-3) -- (2,-1);
    \draw  (4,-1) -- (3,1);
    \draw  (5,1) -- (4,3);
    
    \node[shape=circle, draw=black, fill = white, scale = 1.5]  at (3,1) {}; 
    \node[shape=circle, draw=black, fill = white, scale = 1.5]  at (2,-1) {};
    \node[shape=circle, draw=black, fill = gray, scale = 1.5]  at (1,-3) {};
    \node[shape=circle, draw=black, fill = gray, scale = 1.5]  at (3,-3) {};
    \node[shape=circle, draw=black, fill = gray, scale = 1.5]  at (4,-1) {};
    \node[shape=circle, draw=black, fill = white, scale = 1.5]  at (4,3) {};
    \node[shape=circle, draw=black, fill = gray, scale = 1.5]  at (5,1) {};
    
    \node at (1/4,-3) {$p$};
    \node at (9/4,-3) {$p$};
    \node at (1,-1) {$p \circ p$};
    \node at (13/4,-1) {$p$};
    \node at (17/4,1) {$ p$};
    \node at (6/4,1) {$(p\circ p) \circ p$};
    \node at (8/4,3) {$((p\circ p) \circ p)\circ p$};

\end{tikzpicture}
\end{subfigure}
\caption{Two possible trees $\g,\g'\in\G(4)$, and the corresponding distributions  $C_{\gamma}(p) = (p\circ p) \circ (p \circ p)$ and $C_{\gamma'}(p) = ((p\circ p) \circ p)\circ p.$}
\label{fig1}
\end{figure}
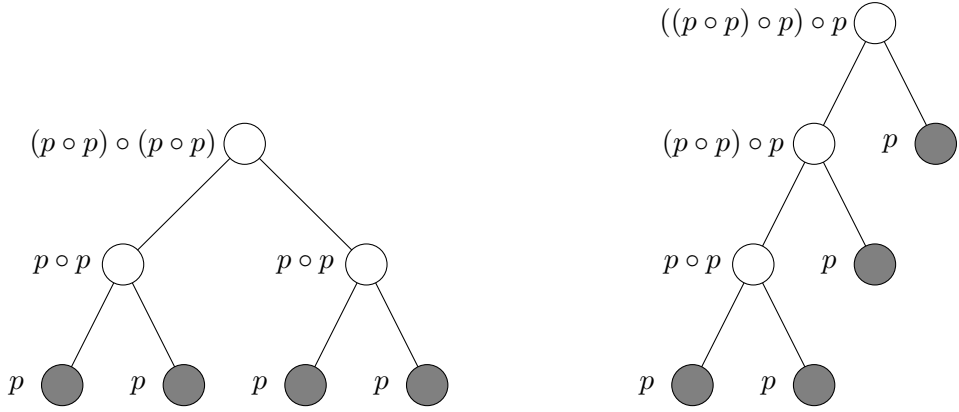

%%%%%%%%%%%

For example, if $k=3$, there are only two trees $\g,\g'\in\G(3)$, with $\a_3(\g)=\a_3(\g')=1/2$ and 
 \begin{align}\label{eq2}
C_{\gamma}(p)=(p \circ p)\circ p=p\circ(p \circ p)= C_{\gamma'}(p).
\end{align}
We refer to Figure \ref{fig1} for two examples of $C_{\gamma}(p)$ for $\g\in\G(4)$.

Since each collision is of the form $q\circ q'$ for some $q,q'\in\cP(\O)$ and since by \eqref{defconv1} one has 
\begin{align}\label{lab:nons}
&p \circ q = \sum_{A \subset [n]}\bar\nu(A)\,p_A \otimes 	q_{A^c},
\end{align}
where $\bar\nu(A)= \frac12(\nu(A)+\nu(A^c))$, one can write the following expansion for the resulting measure at the root:
\begin{align}\label{eq5}
C_{\gamma}(p)= \sum_{\vec A\in \cV_n^{k-1}}\nu(\vec A) C^{\vec A}_{\gamma}(p),
\end{align}
where $\cV_n$ is the set of subsets of $[n]$, so that $\vec A\in \cV_n^{k-1}$ represents a pattern $\vec A=(A_1,\dots,A_{k-1})$, 
\begin{equation}\label{eq:nua}
\nu(\vec A)= \prod_{i=1}^{k-1}\bar\nu(A_i)\end{equation} and $C^{\vec A}_{\gamma}(p)\in\cP(\O)$ is the distribution computed as follows. Each internal node $v_i$, $i=1,\dots,k-1$, in $\gamma$ (in some fixed order) is associated with the set $A_i$, and we attach the mark $A_i$ to the edge connecting $v_i$ to its left child and the mark $A_i^c$ to  the edge connecting $v_i$ to its right child in $\gamma$.  In this way we obtain a tree $\gamma$ with marks on all its edges. Let $d_i$ denote the depth of the $i$-th leaf of $\gamma$ (e.g.\ counting from the leftmost leaf), and consider %define $p^{\vec A}(i):=
$p_{V_i(\vec A)}$, the marginal of the measure $p$ on the subset 
\begin{align}\label{eqa5}
V_i(\vec A):=\cap_{j=1}^{d_i} A^i_j,
\end{align} where $A^i_1,\dots,A^i_{d_i}$ are the marks encountered %(in that order) 
along the edges of the unique path from the root to the $i$-th leaf. Then $ C^{\vec A}_{\gamma}(p)$ is given by
\begin{align}\label{eq6}
C^{\vec A}_{\gamma}(p)%=\otimes _{i=1}^k \, p^{\vec A}(i) 
= p_{V_1(\vec A)}\otimes\cdots\otimes p_{V_k(\vec A)}.
\end{align}
Note that some of the $V_i(\vec A)$ may be empty. However, they  form a partition of $[n]$, namely  $V_i(\vec A)\cap V_j(\vec A)=\emptyset$ for $i\neq j$ and  $\cup_{i=1}^kV_i(\vec A)=[n]$, which can be seen as the result of a fragmentation process $$[n]\to (A_1,A_1^c) \to (A_1\cap A_2,A_1\cap A_2^c,A_1^c)\to\cdots\to (V_1(\vec A),\cdots,V_k(\vec A)).$$ 
As an example, in the case \eqref{eq2} we have $\vec A=(A_1,A_2)$, and 
\begin{align}(p \circ p)\circ p = \sum_{A_1,A_2 \subset[n]}\bar\nu(A_1)\bar\nu(A_2)p_{A_1\cap A_2}\otimes p_{A_1\cap A_2^c} \otimes p_{A_1^c} .
\end{align}
The proof of \eqref{eq00}-\eqref{eq5} can be easily done by induction, by splitting the tree $\g\in\G(k)$ into the left and right subtrees $\g_l,\g_r$ and using the relation \eqref{eqreca}, see the proof of \eqref{eq:idto} below for a closely related explicit computation. 
  
In conclusion, from \eqref{wildsums}, \eqref{eq00} and \eqref{eq5} we  obtain the following representation of the distribution $p_t$ as a convex combination of distributions $C^{\vec A}_{\gamma}(p)$: %probability measures 
\begin{align}\label{eq0}
p_t=\sum_{k=1}^{\infty}\b_t(k)\sum_{\g\in \Gamma(k)}\alpha_k(\g)\sum_{\vec A\in \cV_n^{k-1}}\nu(\vec A) \,C^{\vec A}_{\gamma}(p),
\end{align}
where $\b_t(k) := e^{-t}(1-e^{-t})^{k-1}$ is a probability on $\bbN$ for each $t\geq 0$, $ \alpha_k(\g)$ is a probability on $\G(k)$ for all $k$, and $\nu(\vec A)$ is a probability on $\cV_n^{k-1}$ for all $k$.

\subsection{Contractive estimates for the nonlinear semigroup}
It is convenient to use the notation $S_t(p)=p_t$ for the solution of \eqref{boltzeqc} with initial datum $p_0=p$, so that $S_{t+s}=S_tS_s=S_sS_t$, for all $s,t\geq 0$. This defines the nonlinear semigroup $\{S_t, \,t\geq 0\}$. We show a contraction property for $S_t$ when the two initial distributions $p,q$ have the same marginals, that is $p_i=q_i$ for all $i \in [n]$. 
\begin{theorem}\label{contrlemma2}
For any probability measure $p,q \in \cP(\O)$ such that $p_{i} = q_{i}$ for all $i \in [n]$, all $t\geq 0$,
%and $\phi \in \bbR^{\Om}$ one has 
\begin{align}
&\label{nonlincont1} \|S_t(p)-S_t(q)\|_\tv \leq \tfrac12\,n(n-1)\,\|p-q\|_\tv \,e^{-D(\nu) t}
\end{align}
where $D(\nu)=1-r(\nu)$.
\end{theorem}

\begin{proof} %[Proof of \eqref{nonlincont3}]
From \eqref{eq0} we write 
\begin{align}\label{eqb17}
S_t(p) - S_t(q) 
=\sum_{k=1}^\infty \b_t(k) s^{(k)},
\end{align}
where
\begin{align}\label{eqba17}
s^{(k)}=\sum_{\gamma\in\Gamma(k)}\alpha_k(\gamma)
\sum_{\vec A\in \cV_n^{k-1}}\nu(\vec A) \,\left[C^{\vec A}_{\gamma}(p)-C^{\vec A}_{\gamma}(q)\right]\,.
\end{align}
We introduce some further notation in order to handle more explicitly  the difference of probability measures $C^{\vec A}_{\gamma}(p)-C^{\vec A}_{\gamma}(q)$. For any given $\g\in\G(k)$, and $\vec A\in \cV_n^{k-1}$, $\phi:\O\mapsto\bbR$, we may write
\begin{align}\label{eqb28}
\!\!\!\!\!\!\!\scalar{C^{\vec A}_{\gamma}(p)-C^{\vec A}_{\gamma}(q)}{\phi}
=\!\!\!\!\!\!\!\!
\sum_{(x,y)\in\Omega^k\times \Omega^k}\mu(x,y)  \sum_{z\in\Omega} \phi(z)\left(\ind(z=u(y,\vec A))-\ind(z=u(x,\vec A))
 \right)
\end{align}
where $x=(x^i,i=1,\dots,k)\in\O^k$ stands for the configurations sampled with $q^{\otimes k}$ over the leaves of $\g$,  $y=(y^i,i=1,\dots,k)\in\O^k$ stands for the configurations sampled with $p^{\otimes k}$ over the leaves of $\g$, and $\mu(x,y)=\prod_{i=1}^k\mu_1(x^i,y^i)$ denotes a coupling of these two product measures, such that 
for every $i=1,\dots,k$, $\mu_1(x^i,y^i)$ is the optimal coupling of $(q,p)$: 
\begin{align}\label{eq29}
\|p-q\|_{\rm TV} = \sum_{(x,y)\in\Omega^k\times\Omega^k}\mu(x,y) \ind(x^i\neq y^i)=\sum_{x^i,y^i\in\Omega}\mu_1(x^i,y^i) \ind(x^i\neq y^i).
\end{align}
The notation $u(x,\vec A)$ in 
\eqref{eqb28} is defined as follows. Note that $x^i\in\O$ is a vector $x^i=(x^i_1,\dots,x^i_n)$, for every $i$, with $x^i_\ell\in\{0,\dots,q_\ell\}$, and write $x^i_A=(x^i_\ell)_{\ell\in A}$ for the content of $x^i$ on $A\subset[n]$. 
With this notation, given $x\in\Omega^k$ and $\vec A\in\cV_n^{k-1}$,  $u(x,\vec A)\in\O$ is defined as the unique configuration  such that 
$$u(x,\vec A)_{V_i(\vec A)} = x^i_{V_i(\vec A)}$$ for every $i=1,\dots,k$. In words, for each $i\in[k]$, the content of the configuration $u(x,\vec A)$ on the set $V_i(\vec A)\subset V$ is taken from the configuration $x^i$ at the $i$-th leaf. The validity of \eqref{eqb28} is thus a consequence of \eqref{eq6}.

Let $\partial\g$ denote the set of leaves of $\g$. 
Notice that 
\begin{gather*}
\ind(z=u(x,\vec A)) =
\prod_{i\in \partial \g}
\ind(z_{V_i(\vec A)}=x^i_{V_i(\vec A)})
%\\
%\ind(z=u(x,y,\vec A,j)) =\ind(z_{V_j(\vec A)}=y^j_{V_j(\vec A)} )
%\prod_{i\in \partial \g:\; i\neq j}
%\ind(z_{V_i(\vec A)}=x^i_{V_i(\vec A)}).
\end{gather*}
Let 
$F=F(\vec A,\g)\subset \partial \g$ denote the set of leaves $i\in \partial \g$ such that $|V_i(\vec A)|> 1$, and write
$$
\ind(z=u(x,\vec A)) = w_F (x,z,\vec A) w_{F^c}(x,z,\vec A)\,, $$
where for any $S\subset \partial \g$
we write
\begin{equation}\label{defwS}
w_S(x,z,\vec A) = \prod_{i\in S}
\ind(z_{V_i(\vec A)}=x^i_{V_i(\vec A)}).
\end{equation}
We write $X=(X^i)_{i=1,\dots,k}$ and $Y=(Y^i)_{i=1,\dots,k}$ for the random variables with distribution $q^{\otimes k}$ and $p^{\otimes k}$ respectively. Using the fact that $\mu$ is a product over the leaves, and the fact that $p,q$ have  the same marginals, for fixed $z\in\O$ we have 
\begin{align}\label{eq282a}
\mu[w_{F^c} (X,z,\vec A)]&=\sum_{(x,y)}\mu(x,y)  \prod_{i\in F^c}
\ind(z_{V_i(\vec A)}=x^i_{V_i(\vec A)}) \nonumber\\
&= \sum_{(x,y)}\mu(x,y)  \prod_{i\in F^c}
\ind(z_{V_i(\vec A)}=y^i_{V_i(\vec A)}) = 
\mu[w_{F^c} (Y,z,\vec A)], 
\end{align}
and 
\begin{align}\label{eq282}
& 
\mu[w_{\partial \g}(Y,z,\vec A)]=\sum_{(x,y)}\mu(x,y) 
\ind(z=u(y,\vec A)) =\mu[w_{F^c} (X,z,\vec A)]\mu[w_F (Y,z,\vec A)].
\end{align}
With this notation we  rewrite \eqref{eqb28} as
\begin{align}\label{eqbb28}
\scalar{C^{\vec A}_{\gamma}(p)-C^{\vec A}_{\gamma}(q)}{\phi}
&=
%\sum_{(x,y)\in\Omega^k\times \Omega^k}
\sum_{z\in\Omega} \phi(z) \left[
\mu[w_{\partial \g}(Y,z,\vec A)] - \mu[w_{\partial \g}(X,z,\vec A)\right],
\\
&= %\sum_{(x,y)\in\Omega^k\times \Omega^k}
\sum_{z\in\Omega} \phi(z) 
\mu[w_{F^c} (X,z,\vec A)]\,\mu\left[w_F (Y,z,\vec A)-w_F (X,z,\vec A)
\right].
\end{align}
Note that $w_F (Y,z,\vec A)]\neq w_F (X,z,\vec A)$ implies that there exists $i\in F$ such that $X^i\neq Y^i$. Therefore, from \eqref{eq29} we see that
\begin{align}\label{eqbb2a8}
&%\sum_{(x,y)\in\Omega^k\times \Omega^k}
\sum_{z\in\Omega} \mu[w_{F^c} (X,z,\vec A)]\,
\left|\mu\left[w_F (Y,z,\vec A)]-w_F (X,z,\vec A)
\right]\right|\\& \qquad \leq %\sum_{(x,y)\in\Omega^k\times \Omega^k}
\sum_{i\in F}\sum_{z\in\Omega} \mu[w_{F^c} (X,z,\vec A)]\,\mu\left[|w_F (Y,z,\vec A)-w_F (X,z,\vec A)|\ind(X^i\neq Y^i)
\right] \\& \qquad \leq 2\sum_{i\in F}
%\sum_{(x,y)\in\Omega^k\times \Omega^k}
\mu\left[X^i\neq Y^i
\right] = 2|F|\,\|p-q\|_\tv.
\end{align}
%Next, observe that $|F|\leq n/2$, 
%since $\cup_{i\in F}V_i(\vec A)\subset [n]$ and $|V_i(\vec A)|\geq 2$ for all $i\in F$. 
Therefore, we have shown that \eqref{eqba17} satisfies
\begin{align}\label{eq2aa86}
& |\scalar{s^{(k)}}{\phi}|\leq 2\,\|\phi\|_\infty\,\|p-q\|_\tv \sum_{\gamma\in\Gamma(k)}\alpha_k(\gamma)
\sum_{\vec A\in \cV_n^{k-1}}\nu(\vec A)\,|F|.
\end{align}
To estimate \eqref{eq2aa86}, let us call $\cA_{j}=\cA_j(\gamma)$ the set of $\vec A\in\cV_n^{k-1}$ such that $|V_j(\vec A)|\geq 2$, and let $\cA=\cup_{j=1}^k\cA_j$. 
%Note that these sets may depend on $\gamma$. 
If $\vec A\notin\cA$ then all sets  $V_j(\vec A)$ are either empty or a single site and therefore $|F|=0$. Moreover, $|F|=\sum_{j=1}^k\ind(\cA_{j})$.
Therefore we write
\begin{align}
\sum_{\vec A\in \cV_n^{k-1}}\nu(\vec A)\,|F|= \sum_{j=1}^k
\nu(\cA_j).
\end{align}
To estimate the probability $\nu(\cA_j)$, observe that the event $\cA_j$ implies that there exists $i_1<i_2\in[n]$ such that $\{i_1,i_2\} \in V_j(\vec A)$. Moreover, recalling \eqref{eqa5},
\begin{align*}
\nu\left(\{i_1,i_2\} \in V_j(\vec A)\right)& = \nu\left(\{i_1,i_2\} \in \cap_{l=1}^{d_j} A^j_l\right) %\\&
= \bar \nu\left(\{i_1,i_2\} \in A\right)^{d_j}
\leq \left(\frac{r(\nu)}{2}\right)^{d_j},
\end{align*}
where we use the fact that for any $i_1<i_2$, the probability under $\bar\nu$ that $\{i_1,i_2\} \in A$ is bounded by $r(\nu)/2$. Indeed, conditionally on not being separated the probability that both $i_1,i_2$ belong to $A$ under $\bar\nu$ is $1/2$ by symmetry. 
A union bound then shows that
$$\nu(\cA_j)\leq \tfrac12\,n(n-1)\left(\frac{r(\nu)}{2}\right)^{d_j}.$$ 
In conclusion,
if we define 
\begin{align}\label{eqomegagamma}
\o(\gamma):=\sum_{j=1}^k\left(\frac{r(\nu)}{2}\right)^{d_j},
\end{align} 
we obtain, for all $k\in\bbN$,
\begin{align}\label{eq10}
& |\scalar{s^{(k)}}{\phi}|\leq{n(n-1)}\,\|\phi\|_\infty\,\|p-q\|_\tv \sum_{\gamma\in\Gamma(k)}\alpha_k(\gamma)
\,\o(\g).
\end{align}
From \cite[Lemma 1.4]{carlenwild}, one has, for all $\e>0$,
\begin{align}\label{carlenwildestimation}
\sum_{\gamma \in \Gamma(k)}\alpha_k(\gamma)\o(\gamma) \leq B_\e e^{-a_\e\log k}.
\end{align}
with $a_\e=(1-r(\nu))/(1+\e)$ and $B_\e=B_\e(r(\nu),\e)$. One could use this estimate to obtain an almost optimal decay rate $a_\e$.  However, here we obtain the optimal exponential decay rate by computing  the expectation of $\o(\g)$.
\begin{lemma}\label{lemma:omega}
For any $t\geq 0$,
\begin{equation}\label{lem:omega}
\sum_{k=1}^{\infty}\b_t(k)\sum_{\g\in \Gamma(k)}\alpha_k(\g)\,\o(\gamma) = e^{-(1-r(\nu))t}.
\end{equation}
\end{lemma}
Assuming the validity of \eqref{lem:omega}, and using $\|\mu-\mu'\|_\tv = \tfrac12 \max_{f:\,\|f\|_\infty\leq 1} |\mu(f)-\mu'(f)|$, from \eqref{eqb17} and \eqref{eq10} we obtain
\begin{align}\label{eqabb17}
\|S_t(p) - S_t(q)\|_\tv 
\leq \tfrac12\,n(n-1)\,\|p-q\|_\tv \,e^{-(1-r(\nu))t},
\end{align}
which concludes the proof of Theorem \ref{contrlemma2}.

It remains to prove \eqref{lem:omega}. For any $\g\in\cup_{k=1}^\infty\G(k)$, define $\bbP_t(\g) = \b_t(k)\alpha_k(\g)$, where $k$ is such that $\g\in\G(k)$. Notice that for any $\g\in\cup_{k=1}^\infty\G(k)$, 
\begin{equation}\label{lem:omega1}
\bbP_t(\g) = \ind_{\g=\emptyset}\,e^{-t} + \ind_{\g\neq \emptyset}\int_0^t e^{-s}\bbP_{t-s}(\g_l)\bbP_{t-s}(\g_r)ds,
\end{equation}
where $\emptyset$ denotes the empty tree with one leaf (given by the root), while for $\g\neq \emptyset$ we write $\g_l,\g_r$ for the left and right subtrees. Indeed,
for $\g=\emptyset$ \eqref{lem:omega1} is immediate, while for $\g\neq\emptyset$, supposing $\g\in\G(k)$, and recalling \eqref{eqreca} one has
\begin{align*}
\bbP_{t}(\g) &= e^{-t}(1-e^{-t})^{k-1}\a_k(\g) =  e^{-t}\int_0^{t}e^{-(t-s)}(1-e^{-(t-s)})^{k-2}(k-1)\a_k(\g)ds\\ &=
\int_0^t e^{-s}e^{-2(t-s)}(1-e^{-(t-s)})^{k-2}\a_j(\g_l)\a_{k-j}(\g_r)ds
=\int_0^t e^{-s}  \bbP_{t-s}(\g_l)\bbP_{t-s}(\g_r)ds.
\end{align*}
To prove Lemma \ref{lemma:omega}, note that the left hand side of \eqref{lem:omega} is given by $\sum_\g \bbP_t(\g)\o(\g)$. Using \eqref{lem:omega1} and $\o(\emptyset)=1$,
we see that
\begin{align*}
\z(t):&=\sum_\g\bbP_{t}(\g)\o(\g) =e^{-t} +\frac{r(\nu)}2
\sum_{\g\neq\emptyset}\int_0^t e^{-s}  \bbP_{t-s}(\g_l)\bbP_{t-s}(\g_r)(\o(\g_l)+\o(\g_r))ds\\
&=e^{-t} +r(\nu)
\sum_{\g}\int_0^t e^{-s}  \bbP_{t-s}(\g)\o(\g)ds=e^{-t} +r(\nu)
\int_0^t e^{-s}  \z(t-s)ds.
\end{align*}
Differentiating, and integrating by parts the resulting expression one finds $\dot\z(t) = -(1-r(\nu))\z(t)$. Since %one and the only solution to this equation with 
$\z(0)=1$ we obtain $\z(t)=e^{-(1-r(\nu))t}$, $t\geq 0$. 
%
%We conclude the proof of the theorem by observing that \eqref{eq60}, \eqref{eq10} and \eqref{lem:omega} imply, by convexity,
%\begin{align}\label{eq:tow}
%W(S_t(p),S_t(q)) \leq \tfrac12\,n(n-1)\,W(p,q)e^{-(1-r(\nu))t}. 
%\end{align}
%It remains to show that
%\begin{align}\label{enwild}
%\sum_{k=1}^{\infty}\b_t(k)e^{-a\log k}\leq B'e^{-a\,t},
%\end{align}
%for some new constant $B'=B'(r(\nu),\e)$.
%To see this, observe that for any $k_0\in\bbN$ the sum on $k\leq k_0$ is estimated by 
%$$ e^{-t}\sum_{k=1}^{k_0} k^{-a} \leq e^{-t} C_a k_0^{1-a},$$ 
%where $C_a$ only depends on $a\in(0,1)$. The sum on $k> k_0$ is estimated by $k_0^{-a}$, and therefore \eqref{enwild} follows by taking $k_0 = \lfloor e^{t}\rfloor$. 
\end{proof}

With a similar argument %As a consequence of the above result 
we obtain Theorem \ref{prop:tvbo}.

\begin{proof}[{\bf Proof of Theorem \ref{prop:tvbo}}]
The statement about the total variation distance is contained in Theorem \ref{contrlemma2}.
%follows immediately from \eqref{eq:tvW} and \eqref{nonlincont1}.
To prove the statement \eqref{entproductnoncontr} about the relative entropy, we observe that by 
%\eqref{wildsums}, 
\eqref{eq5}, \eqref{eq0}, and convexity of the relative entropy,
\begin{align}
H(p_t\tc\pi) \leq \sum_{k=1}^{\infty}\b_t(k)\sum_{\gamma \in \Gamma(k)}\alpha_k(\gamma)\sum_{\vec A\in \cV_n^{k-1}}\nu(\vec A) \,H\left(C_{\gamma}^{\vec{A}}(p)\tc\pi\right).
\end{align}
%\sum_{\vec A\in \cV_n^{k-1}}\nu(\vec A) C^{\vec A}_{\gamma}(p)
Since $\pi$ is a product measure, using \eqref{eq6} we see that 
%the standard tensorization property of the relative entropy, see e.g.\ Lemma \ref{lem:shearer} below for a more general statement, shows that 
for any $\g\in\G(k)$,
\begin{align}
H\left(C_{\gamma}^{\vec{A}}(p)\tc \pi\right) = \sum_{j=1}^k
H\left( p_{V_j(\vec A)}\tc\pi_{V_j(\vec{A})}\right) = 
\sum_{j : |V_j(\vec{A})| > 1}H\left( p_{V_j(\vec A)}\tc \pi_{V_j(\vec{A})}\right),
\end{align}
where use the fact that if $|V_j(\vec A)|= 1$ then $H( p_{V_j(\vec A)}\tc\pi_{V_j(\vec{A})})=0$ since $p,\pi$ have the same marginals. 
Furthermore, the monotonicity property of entropy implies that 
\begin{align}
H\left( p_{V_j(\vec A)}\tc \pi_{V_j(\vec{A})}\right) \leq H(p\tc \pi).
\end{align}
Therefore, repeating the argument in \eqref{eq10}-\eqref{lem:omega} we obtain
\begin{align}
H(p_t\tc \pi) &\leq  
\tfrac12\,n(n-1)\,H(p\tc \pi) \,e^{-(1-r(\nu))\,t}.
%B\sum_{k=1}^{\infty}\b_t(k)e^{-a\log k}
%\leq H(p\tc \pi)\,n^2B'\,e^{-D_\e\,t}.
\end{align}
\end{proof}

\begin{remark}\label{rem:Dopt}
The constant $D(\nu)=1-r(\nu)$ in Theorem \ref{prop:tvbo} is optimal in the sense that, for any recombination measure $\nu$, there are initial distributions $p$ with the same marginals as $\pi$ such that
\begin{align}\label{eq:Dopt}
\liminf_{t\to\infty}\frac1t \,\log\|p_t-\pi\|_\tv \geq -D(\nu). 
\end{align}
To see this, pick $i_1,i_2\in[n]$ such that $r(\nu)=\text{Pr}_\nu\left(A\text{ does not separate $i_1$ and $i_2$}\right)$, see \eqref{def:rnu}. For simplicity, take $\O=\{0,1\}^n$, $\pi$ uniform over $\O$ and $p=\frac12\d_{\underline 0}+ \frac12\d_{\underline 1}$. Consider the event $\{\si_{i_1}=\si_{i_2}\}$, and let $\cB_t$ denote the event that $i_1$ and $i_2$ are not separated by the fragmentation process at time $t$. From \eqref{eq0} we write $p_t$ as an average over the fragmentation process and note that conditionally on the event $\cB_t$ one has $p_t(\si_{i_1}=\si_{i_2}\tc\cB_t)=1$, while conditionally on the event $\cB^c_t$ one has $p_t(\si_{i_1}=\si_{i_2}\tc\cB^c_t)=1/2$. Moreover $\pi(\si_{i_1}=\si_{i_2})=1/2$ and therefore 
\begin{align*}
\|p_t-\pi\|_\tv &\geq p_t(\si_{i_1}=\si_{i_2}) -  \pi(\si_{i_1}=\si_{i_2}) 
 = \bbP(\cB_t) + \tfrac12\bbP(\cB_t^c)  - \tfrac12 = \tfrac12\bbP(\cB_t).
\end{align*}
On the other hand, with the notation from \eqref{eqomegagamma} and \eqref{lem:omega}, $$\bbP(\cB_t) = \sum_\g \bbP_t(\g) \o(\g)= e^{-D(\nu)\,t},$$ 
which implies \eqref{eq:Dopt}.
\end{remark}

\subsection{Contraction for the linearized equation}
Consider the symmetric bilinear form $\hat{Q}(f,g)$ defined by 
\begin{align}
& \hat{Q}(f,g)(\eta)= %\\ \label{bilboltz}
%& 
\frac{1}{2}\sum_{A,\si}%\substack{A \subset [n] \\ \sigma \in \Om}}
\nu(A)\left(f(\sigma_{A} \eta_{A^c}) g(\eta_{A}\sigma_{A^c}) + g(\sigma_{A} \eta_{A^c})f(\eta_{A}\sigma_{A^c}) - f(\sigma)g(\eta)- g(\sigma)f(\eta)\right),
\end{align}
where $\eta \in \Om$, $f,g:\O\mapsto\bbR$, and the sum extends to all $A \subset [n]$ and $\si \in \O$. 
If $p \in \cP(\O)$, then $\hat{Q}(p,p) = Q(p),$ where $$Q(p)=\sum_A\nu(A)(p_A\otimes p_{A^c} - p)$$ is the generator associated to \eqref{boltzeq}, that is the non linear semigroup $\{S_t, t \geq 0\}$ satisfies $S_0 (p)= p$ and $\partial_t S_t (p)= Q (S_t (p))$, $t\geq 0$. 
Next, consider the differential equation
\begin{align}\label{linsem}
&\partial_t h_t = 2\hat{Q}(q_t,h_t), \ \ \ \;\;\;  h_0=h, % h_t|_{t=0}=h,
\end{align}
where $q_t = S_t(q)$ for some fixed $ q \in \cP(\O)$. Note that it is linear in $h$. 
We write its unique solution as $$h_t=\bar{S}_t(q)(h),$$ that is $\bar{S}_t(q)(h)$ verifies $\partial_t \bar{S}_t(q)(h) = 2\hat{Q}(q_t,\bar{S}_t(q)(h))$,  $\bar{S}_0(q)(h)=h$ and $$\bar{S}_t(q)(\bar{S}_s(q)(h)) = \bar{S}_s(q)(\bar{S}_t(q)(h)) = \bar{S}_{t+s}(q)(h).$$ 
Our main result concerning the linearized equation reads as follows. 
\begin{theorem}\label{wasscontrlemma2}
For all $p,q \in \cP(\O)$ such that $p_{i} = q_{i}$ for all $i \in [n]$, and $\phi:\O\mapsto\bbR$, for all $t>0$,
\begin{align}
&\label{nonlincont2} \scalar{\bar{S}_t(q)(p-q)}{\phi} \leq n(n-1)\|\phi\|_{\infty}e^{-D(\nu) \,t}\,\|p-q\|_\tv, \\
&\label{nonlincont3} \scalar{S_t(p)-S_t(q)-\bar{S}_t(q)(p-q)}{\phi} \leq \tfrac1{16}\,n^3(n-1)(n-2) \|\phi\|_{\infty}e^{-D(\nu) \,t}\,\|p-q\|_\tv^2
\end{align}
where %$B=B(\e,r(\nu))$ depends only on $\e,r(\nu)$, 
$D(\nu)= 1-r(\nu)$.
\end{theorem}
The proof of Theorem \ref{wasscontrlemma2} requires several steps. We start by giving an explicit representation of the solution $\bar{S}_t(q)(p-q)$ to \eqref{linsem} when $h=p-q$. To this end let us define, for any $k\geq 2$, $\g\in\G(k)$, the distribution $C_{\gamma, p}(q)$ as the measure defined in \eqref{eq5} with the only difference that we take the distribution $p$ instead of $q$ at the rightmost leaf of $\g$, all other leafs remaining with the distribution $q$.   Formally,
\begin{align}\label{eqab5}
C_{\gamma,p}(q)= \sum_{\vec A\in \cV_n^{k-1}}\nu(\vec A) \,C^{\vec A}_{\gamma,p}(q),\qquad C^{\vec A}_{\gamma,p}(q)%=\otimes _{i=1}^k \, p^{\vec A}(i) 
= q_{V_1(\vec A)}\otimes\cdots\otimes q_{V_{k-1}(\vec A)}\otimes p_{V_{k}(\vec A)}.
\end{align}
%We also use the symmetric version of \eqref{eqab5} obtained as in \eqref{eq:symme}
%with the only difference that we take the distribution $p$ instead of $q$ at the rightmost leaf of $\g$, all other leafs remaining with the distribution $q$, that is
%\begin{align}\label{eqab5sy}
%C_{\gamma,p}(q)= \sum_{\vec A\in \cV_n^{k-1}}\nu(\vec A) \,\hat C^{\vec A}_{\gamma,p}(q),
%\end{align}
%where, with reference to \eqref{lab:rec}, $\hat C^{\vec A}_{\gamma,p}(q)$ is defined recursively by
%\begin{align*}
%& \hat{C}^{\vec A}_{\gamma,p}(q) = \left(\hat{C}^{A_2,\dots,A_{j}}_{\gamma_l}(q)\circ \hat{C}^{A_{j+1},\dots,A_{k-1}}_{\gamma_r,p}(q)\right)_{A_{1}},
%\end{align*}
%and $\hat{C}^{A}_{\g_2,p}(q)=(p\circ q)_A$, when $\g_2\in\G(2)$ is the tree with two leafs. 

In what follows we denote by  $d_r(\g)$ the depth of the rightmost leaf in $\g$. 
\begin{lemma}\label{lem:wildrep}
For all $p,q \in \cP(\O)$, $t\geq 0$,
\begin{align}\label{eq:topid}
\bar{S}_t(q)(p-q) = \sum_{k=1}^\infty\b_t(k) \sum_{\gamma \in \Gamma(k)}\alpha_k(\gamma)(C_{\gamma, p}(q)-C_{\gamma}(q))2^{d_r(\gamma)}.
\end{align}
\end{lemma}
\begin{proof}
Fix $p,q \in \cP(\O)$. If $h_t=\bar{S}_t(q)(p)$, then 
\begin{align*}
\partial_t \sum_{\sigma}h_t(\sigma) = 2\sum_{\sigma}\hat{Q}(q_t,h_t)(\si) = 0.
\end{align*}
Thus $ \sum_{\sigma}h_t(\sigma) = 1$, and  $2\hat{Q}(q_t,h_{t})=2q_t \circ h_{t} - q_t - h_{t}$.
Therefore $h_t$ satisfies 
\begin{align}
h_t = qe^{-t} + \int_{0}^te^{-(t-s)}\left(2q_{s} \circ h_{s}-q_{s}\right)ds.
\end{align}
By linearity, $ \bar{S}_t(q)(p-q) =\bar{S}_t(q)(p) - \bar{S}_t(q)(q)$ satisfies $\bar{S}_t(q)(p-q) =u_t-v_t$, where 
\begin{align}\label{eq:solve}
u_t = pe^{-t} + \int_{0}^te^{-(t-s)}2q_{s} \circ u_{s}ds,\qquad v_t = qe^{-t} + \int_{0}^te^{-(t-s)}2q_{s} \circ v_{s}ds.
\end{align}
Wild's construction then shows that
\begin{align}\label{eq:defuv}
u_t =  e^{-t}\sum_{k=1}^{\infty}(1-e^{-t})^{k-1}u^{(k)},\qquad v_t =  e^{-t}\sum_{k=1}^{\infty}(1-e^{-t})^{k-1}v^{(k)},
\end{align}
where $u^{(1)}= p$, $v^{(1)}= q$, and for $k\geq 2$,
\begin{align*}
u^{(k)}= \frac{1}{k-1}\sum_{j=1}^{k-1}2q^{(j)}\circ u^{(k-j)} \,,\qquad
v^{(k)}= \frac{1}{k-1}\sum_{j=1}^{k-1}2q^{(j)}\circ v^{(k-j)}.
\end{align*}
respectively, and $q^{(j)}$ is defined by \eqref{eq:wildp} with $p$ replaced by $q$. To check that this representation of the solution holds, let $u_t$ be defined by \eqref{eq:defuv}. Then 
\begin{align}\label{eq:defuv1}
\int_{0}^te^{-(t-s)}2q_{s} \circ u_{s}ds
&= \sum_{\ell,k=1}^{\infty}
\int_{0}^t\b_s(k)\b_s(\ell)e^{-(t-s)}2q^{(\ell)} \circ u^{(k)}ds\nonumber\\
&  =\sum_{\ell<k}
\int_{0}^te^{-2s}(1-e^{-s})^{k-2}e^{-(t-s)}2q^{(\ell)} \circ u^{(k-\ell)}ds\nonumber\\
&  =\sum_{k\geq 2}
\int_{0}^t(k-1)e^{-s}(1-e^{-s})^{k-2}e^{-t} u^{(k)}ds \nonumber\\
& =\sum_{k\geq 2}
(1-e^{-t})^{k-1}e^{-t} u^{(k)} = u_t - e^{-t} u^{(1)} =  u_t - e^{-t} p.
\end{align}
Thus $u_t$ solves \eqref{eq:solve}. The same applies to $v_t$. This proves \eqref{eq:defuv}.

Next, 
we prove that
\begin{align}\label{eq:idto}
u^{(k)}=\sum_{\gamma \in \Gamma(k)}\alpha_k(\gamma)C_{\gamma,p}(q)2^{d_r(\gamma)}\,,\qquad 
v^{(k)}=\sum_{\gamma \in \Gamma(k)}\alpha_k(\gamma)C_{\gamma}(q)2^{d_r(\gamma)}
\end{align}
for all $k \geq 2$.
Let us prove this by using induction over $k$. If $k=2$ then $u^{(2)} = 2 q\circ p$ and thus the claimed identity holds.  If the formula holds for $j \leq k-1,$ then using \eqref{eq5} and \eqref{eqreca},
\begin{align*}
 u^{(k)}&=\frac{1}{k-1}\sum_{j=1}^{k-1}2q^{(j)}\circ u^{(k-j)} \\
 &=\frac{1}{k-1}\sum_{j=1}^{k-1}\sum_{\substack{\gamma_l \in \Gamma(j) \\ \gamma_r \in \Gamma(k-j)}}\alpha_j(\gamma_l)\alpha_{k-j}(\gamma_r) C_{\gamma_l}(q) \circ C_{\gamma_r, p}(q)2^{d_r(\gamma_r)+1} \\
 %&=\frac{1}{k-1}\sum_{j=1}^{k-1}\sum_{\gamma \in \Gamma(k)}\alpha_k(\gamma)C_{\gamma,p_0}(q_0)2^{d_r(\gamma)}\\
 &=\sum_{\gamma \in \Gamma(k)}\alpha_k(\gamma)C_{\gamma,p}(q)2^{d_r(\gamma)}.
\end{align*}
This proves the identity \eqref{eq:idto} for $u$, and the same argument proves the one for $v$. 
The identities \eqref{eq:defuv} and \eqref{eq:idto} imply \eqref{eq:topid}.                      
\end{proof}

Next, for any $k\geq 2$, $\g\in\G(k)$, $j\in[k]$, we define 
\begin{align}\label{eq11}
C_{\gamma,p,j}(q)= \sum_{\vec A\in \cV_n^{k-1}}\nu(\vec A) C^{\vec A}_{\gamma,p,j}(q)\,,\qquad C^{\vec A}_{\gamma,p,j}(q):=  p_{V_j(\vec A)} \otimes _{\ell\neq j} q_{V_\ell(\vec A)},
\end{align}
which denotes the distribution obtained from the tree $\gamma\in\G(k)$ when all leaves are given the distribution $q$ except for the $j$-th leaf which takes the distribution $p$. When $j=k$, that is $j$ is the rightmost leaf, then we have $C_{\gamma,p,k}(q) = C_{\gamma,p}(q)$. 
%Note that the latter distribution coincides with $ C_{\gamma,p}(q)$ defined above. 

\begin{lemma}\label{lem:dr}
	For any $k\geq 2$, for all $p,q$: 
	\begin{align}\label{eq18}
	\sum_{\gamma\in\Gamma(k)}\alpha_k(\gamma)2^{d_r}C_{\gamma,p}(q)
	=\sum_{j=1}^k\sum_{\gamma\in\Gamma(k)}\alpha_k(\gamma)C_{\gamma,p,j}(q)
	\,.
	\end{align}
\end{lemma}
\begin{proof}
	By symmetry, it holds for $k=2$, since in this case \eqref{eq18} says $2q\circ p =  q\circ p+p\circ q$. Assume that \eqref{eq18} holds for all $j\leq k-1$. Then 
	\begin{align}\label{eq19}
	\sum_{\gamma\in\Gamma(k)}\alpha_k(\gamma)2^{d_r(\gamma)}C_{\gamma,p}(q) &= \frac2{k-1}\sum_{j=1}^{k-1}
	\sum_{\gamma\in\Gamma(j)}\alpha_j(\gamma)C_{\gamma}(q)\circ\left( 
	\sum_{\gamma'\in\Gamma(k-j)}\alpha_{k-j}(\gamma')2^{d_r(\gamma')}C_{\gamma',p}(q) \right)
	\nonumber\\
	& = 
	\frac2{k-1}\sum_{j=1}^{k-1}\sum_{\gamma\in\Gamma(j)}\alpha_j(\gamma)C_{\gamma}(q)\circ\left( 
	\sum_{i=1}^{k-j}\sum_{\gamma'\in\Gamma(k-j)}\alpha_{k-j}(\gamma')C_{\gamma',p,i}(q) \right)
	\nonumber\\
	& =  2
	%\frac2{k-1}\sum_{j=1}^{k-1}
	\sum_{\gamma\in\Gamma(k)}\alpha_k(\gamma)\sum_{i=j(\gamma)+1}^{k}C_{\gamma,p,i}(q) ,
	\end{align}
	where $j(\gamma)$ is the number of leaves in the left subtree of $\gamma$, and we have used \eqref{eqreca}. 
	By symmetry, 
	%for every $i\neq j$:
	% \begin{align}\label{eq20}
	%\sum_{\gamma\in\Gamma(k)}\alpha_k(\gamma)C_{\gamma,p,i}(q)
	% =
	%\sum_{\gamma\in\Gamma(k)}\alpha_k(\gamma)C_{\gamma,p,j}(q)\,.
	%\end{align}
	%Similarly,
	\begin{align}\label{eq201}
	\sum_{\gamma\in\Gamma(k)}\alpha_k(\gamma)\sum_{i=j(\gamma)+1}^{k}C_{\gamma,p,i}(q)
	=
	\sum_{\gamma\in\Gamma(k)}\alpha_k(\gamma)\sum_{i=1}^{j(\gamma)}C_{\gamma,p,i}(q)\,.
	\end{align}
	Therefore, 
	\begin{align}\label{eq202}
	2
	\sum_{\gamma\in\Gamma(k)}\alpha_k(\gamma)\sum_{i=j(\gamma)+1}^{k}C_{\gamma,p,i}(q)\nonumber
	%\\
	%&\qquad 
	&=
	\sum_{\gamma\in\Gamma(k)}\alpha_k(\gamma)\sum_{i=1}^{j(\gamma)}C_{\gamma,p,i}(q)
	+\sum_{\gamma\in\Gamma(k)}\alpha_k(\gamma)\sum_{i=j(\gamma)+1}^{k}C_{\gamma,p,i}(q) \nonumber\\
	&  =
	\sum_{i=1}^k\sum_{\gamma\in\Gamma(k)}\alpha_k(\gamma)C_{\gamma,p,i}(q).
	\end{align}
	The desired identity follows from \eqref{eq19} and \eqref{eq202}.
\end{proof}

We can now turn to the proof Theorem of \ref{wasscontrlemma2}.

\begin{proof}[{\bf Proof of \eqref{nonlincont2}}]
From Lemma \ref{lem:wildrep} and  Lemma \ref{lem:dr}, we write
\begin{align}\label{eqbi17}
\bar S_t(q)(p-q) 
=\sum_{k=1}^\infty \b_t(k) \,\bar s^{(k)},
\end{align}
where
\begin{align}\label{eqbia17}
\bar s^{(k)}=\sum_{j=1}^k\sum_{\gamma\in\Gamma(k)}\alpha_k(\gamma)
\sum_{\vec A\in \cV_n^{k-1}}\nu(\vec A) \,\left[C^{\vec A}_{\gamma,p,j}(q)-C^{\vec A}_{\gamma}(q)\right]\,.
\end{align}
Let $\cA_{j}$ denote the set of $\vec A\in\cV_n^{k-1}$ such that $|V_j(\vec A)|\geq 2$. Since $p,q$ have the same marginals, arguing as in the proof of Theorem \ref{contrlemma2}, we obtain, or all $\phi:\O\mapsto\bbR$,
\begin{align}\label{eqbas17}
|\scalar{C^{\vec A}_{\gamma,p,j}(q)-C^{\vec A}_{\gamma}(q)}{\phi}|\leq 2\|\phi\|_\infty \|p-q\|_\tv\, \ind(\cA_j)\,.
\end{align}
Therefore, as in \eqref{eq10} we obtain
\begin{align}\label{eqbas18}
|\scalar{\bar s^{(k)}}{\phi}|\leq {n(n-1)}\,\|\phi\|_\infty\,\|p-q\|_\tv 
\sum_{\gamma\in\Gamma(k)}\alpha_k(\gamma)
\o(\g).
\end{align}
From \eqref{eqbi17} and Lemma \ref{lemma:omega} we conclude the proof of \eqref{nonlincont2}.
\end{proof}

\begin{proof}[{\bf Proof of \eqref{nonlincont3}}]
The proof of \eqref{nonlincont3} requires a bit more work. 
Let us define $$r_t: = S_t(p) - S_t(q) - \bar{S}_t(q)(p-q).$$
From Lemma \ref{lem:wildrep} and  Lemma \ref{lem:dr} we see that
\begin{align}\label{eqa17}
r_t=\sum_{k=1}^\infty \b_t(k) r^{(k)},
\end{align}
where
\begin{align}\label{eqaa17}
r^{(k)}=\sum_{\gamma\in\Gamma(k)}\alpha_k(\gamma)\left[C_{\gamma}(p)-C_{\gamma}(q)-\sum_{i=1}^k(C_{\gamma,p,i}(q)-C_{\gamma}(q))\right]\,,
\end{align}
We are going to prove that for any $\phi:\Omega\mapsto\bbR$, %with $\|\phi\|_\infty\leq 1$, 
\begin{align}\label{eq27}
\scalar{r^{(k)}}{\phi}& %\leq \scalar{p^{(k)}-q^{(k)}-h^{(k)}}{\phi}
\leq \tfrac1{16}\,n^3(n-1)(n-2)\|\phi\|_\infty \|p-q\|^2_{\rm TV}\sum_{\gamma\in\Gamma(k)}\alpha_k(\gamma)\o(\g).
\end{align}
By the argument in \eqref{eqabb17} and Lemma \ref{lemma:omega} this is sufficient to end the proof.

With the notation from the proof of Theorem \ref{contrlemma2} we write 
\begin{align}\label{eq28}
& \scalar{r^{(k)}}{\phi}=\sum_{\gamma\in\Gamma(k)}\alpha_k(\gamma)
\sum_{(x,y)\in\Omega^k\times \Omega^k}\sum_{\vec A\in\cV_n^{k-1}}\mu(x,y)\nu(\vec A) \,\times  \\
&\times \sum_{z\in\Omega} \phi(z)\left(\ind(z=u(y,\vec A))-\ind(z=u(x,\vec A))
- \sum_{i=1}^k\left(\ind(z=u(x,y,\vec A,i) -\ind(z=u(x,\vec A)) \right)\right), 
\nonumber
\end{align}
where we call $u(x,y,\vec A,j)$ the unique configuration in $\Omega$ such that 
$$u(x,y,\vec A,j)_{V_j(\vec A)} = y^j_{V_j(\vec A)}\,,\qquad u(x,y,\vec A,j)_{V_i(\vec A)} = x^i_{V_i(\vec A)}\,,\;\;\text{for every }\; i\neq j.$$ That is, $u(x,y,\vec A,j)$ coincides with $u(x,\vec A)$ except that on $V_j(\vec A)$ its content (if not empty) is taken from $y^j$. 
It follows that 
\begin{gather*}
\ind(z=u(x,\vec A)) =
\prod_{i\in \partial \g}
\ind(z_{V_i(\vec A)}=x^i_{V_i(\vec A)})\,,\\
\ind(z=u(x,y,\vec A,j)) =\ind(z_{V_j(\vec A)}=y^j_{V_j(\vec A)} )
\prod_{i\in \partial \g:\; i\neq j}
\ind(z_{V_i(\vec A)}=x^i_{V_i(\vec A)})\,,
\end{gather*}
where $\partial\g$ denotes the set of leaves of $\g$. 

As in the proof of Theorem \ref{contrlemma2}, we let $F=F(\vec A,\g)\subset \partial \g$ denote the set of leaves $i\in \partial \g$ such that $|V_i(\vec A)|> 1$, and recall that since $\mu$ is a product over the leaves, and $p,q$ have  the same marginals one has the identities 
\eqref{eq282a} and \eqref{eq282}.
With this notation we  rewrite \eqref{eq28} as
\begin{align}\label{eq2a8}
& \scalar{r^{(k)}}{\phi}=\sum_{\gamma\in\Gamma(k)}\alpha_k(\gamma)
\sum_{\vec A\in\cV_n^{k-1}} \nu(\vec A)\sum_{z\in\Omega} \phi(z)\times \\ 
&  \times \left[
\mu[w_{\partial \g}(Y,z,\vec A)] - \mu[w_{\partial \g}(X,z,\vec A)] 
- \sum_{i=1}^k\left(\mu[w_{\partial \g}(X(i,y),z,\vec A)] - \mu[w_{\partial \g}(X,z,\vec A)]
\right)\right],\nonumber
\end{align}
where $X(i,y)\in\O^k$ denotes the vector whose $i$-th component is $Y^i$ while all other components are $X^j$.  
Moreover, for a given choice of $\g,\vec A, z$, the square bracket in \eqref{eq2a8} is also equal to 
\begin{align}\label{eq2a81}
&\mu[w_{\partial \g}(Y,z,\vec A)] - \mu[w_{\partial \g}(X,z,\vec A)] 
- \sum_{i=1}^k\left(\mu[w_{\partial \g}(X(i,y),z,\vec A)] - \mu[w_{\partial \g}(X,z,\vec A)]
\right) \\
& = \mu[w_{F^c} (X,z,\vec A)]\,\,\mu\left[w_F (Y,z,\vec A)-w_F (X,z,\vec A)-\sum_{i\in F}
\left(w_{F}(X(i,y),z,\vec A) - w_{F}(X,z,\vec A)
\right)\right].
\nonumber
\end{align}
In particular, if $\g,\vec A$ are such that $F=F(\g,\vec A)=\emptyset$, then  \eqref{eq2a81} vanishes. 
Moreover, \eqref{eq2a81} vanishes also when $F$ is a single leaf. Indeed, if e.g.\  $F=\{j\}$, 
then 
\begin{align}\label{eq2a83}
&w_F (Y,z,\vec A)-w_F (X,z,\vec A)-\sum_{i\in F}
\left(w_{F}(X(i,y),z,\vec A) - w_{F}(X,z,\vec A)
\right)\nonumber \\ &\qquad \qquad   =w_F (Y,z,\vec A)-
w_{F}(X(j,y),z,\vec A) = 0. 
\end{align}
Thus, we may restrict to the case of $|F|>1$. Therefore, 
\begin{align}\label{eq2a85}
& |\scalar{r^{(k)}}{\phi}|\leq \|\phi\|_\infty\sum_{\gamma\in\Gamma(k)}\alpha_k(\gamma)
\sum_{\vec A\in\cV_n^{k-1}} \nu(\vec A)\ind(|F|>1) 
\,\mu\left[\G(\g,\vec A,X,Y)\right],
\end{align}
where we use the notation 
\begin{align}
&\G(\g,\vec A,X,Y) \\& = \sum_{z\in\O} w_{F^c} (X,z,\vec A)\left|w_F (Y,z,\vec A)-w_F (X,z,\vec A)-\sum_{i\in F}
\left(w_{F}(X(i,y),z,\vec A) - w_{F}(X,z,\vec A)
\right)\right|.
\end{align}
Clearly, if $X=Y$ the expression within absolute values above vanishes. Moreover, the same applies if $(X,Y)$ is such that $X^i=Y^i$ for all $i\in F$ except one leaf $j\in F$ where $X^j\neq Y^j$. Indeed, in this case $X(i,y)=X$ for all $i\in F$, $i\neq j$, and $X(j,y)=Y$.
Given $X,Y$ we write $E=E(X,Y)\subset F$ for the set of leaves $i\in F$ where $X^i\neq Y^i$. Then 
\begin{align}\label{abbao}
%&\sum_z w_{F^c} (x,z,\vec A)\left|w_F (y,z,\vec A)-w_F (x,z,\vec A)-\sum_{i\in F}
%\left(w_{F}(x(i,y),z,\vec A) - w_{F}(x,z,\vec A)\right)\right|
%\nonumber \\
&\G(\g,\vec A,X,Y) \nonumber\\&\quad=\sum_z w_{E^c\cup F^c} (X,z,\vec A)\left|w_E (Y,z,\vec A)-w_E (X,z,\vec A)-\sum_{i\in E}
\left(w_{E}(X(i,y),z,\vec A) - w_{E}(X,z,\vec A)\right)\right|\nonumber \\
&\quad \leq \sum_z w_{E^c} (X,z,\vec A)\left|w_E (Y,z,\vec A)-\sum_{i\in E}
w_{E}(X(i,y),z,\vec A)\right| \ind(|E|>1)+ (|E|-1)\ind(|E|>1)\nonumber \\
&\quad \leq (|E|+1)\ind(|E|>1)+ (|E|-1)\ind(|E|>1)
=  2|E|\ind(|E|>1) \leq 2|F|\ind(|E|>1).
\end{align}
Next, we use $|F|\leq n/2$, 
since $\cup_{i\in F}V_i(\vec A)\subset [n]$ and $|V_i(\vec A)|\geq 2$ for all $i\in F$. 
Since $\mu$ is a product over leaves and on each leaf it satisfies \eqref{eq29}, 
a union bound over the set of pairs in $F$ shows that
\begin{align}\label{eqa128}
\mu\left[|E|>1\right]\leq \tfrac12\,|F|(|F|-1)\|p-q\|_{\rm TV}^2. %\leq \tfrac{1}8\,n(n-2)\,\|p-q\|_{\rm TV}^2.
\end{align}
Therefore,
\begin{align}\label{eq2a86}
& |\scalar{r^{(k)}}{\phi}|\leq \|\phi\|_\infty \|p-q\|_{\rm TV}^2\sum_{\gamma\in\Gamma(k)}\alpha_k(\gamma)
\nu(|F|^2(|F|-1))
.
\end{align}
Next, observe that $|F|\leq n/2$, 
since $\cup_{i\in F}V_i(\vec A)\subset [n]$ and $|V_i(\vec A)|\geq 2$ for all $i\in F$. Thus, if $|F|>1$ one has $|F|^2(|F|-1)\leq n^2(n-2)/8$.
Summarizing,
\begin{align}\label{eq2a87}
& |\scalar{r^{(k)}}{\phi}|\leq \tfrac{1}8\,n^2(n-2)\,\|\phi\|_\infty\,\|p-q\|_{\rm TV}^2 \sum_{\gamma\in\Gamma(k)}\alpha_k(\gamma)\nu(|F|>1).
\end{align}
For a given tree $\g$, the argument in \eqref{eq10} shows that 
%a union bound over leafs and over pairs of letters $i_1,i_2\in[n]$ shows that
$$
%\sum_{\gamma\in\Gamma(k)}\alpha_k(\gamma)
\nu(|F|>1)
\leq %\sum_{\gamma\in\Gamma(k)}\alpha_k(\gamma)
\nu(|F|\geq 1) \leq  
\tfrac12\,n(n-1)\sum_{i=1}^k\left(\frac{r(\nu)}{2}\right)^{d_i(\g)} = \tfrac12\,n(n-1)\,\,\o(\g) ,
$$
where we use the fact that the event $F\geq 1$ coincides with $\cup_{i=1}^k\cA_i$ from the proof of \eqref{eq10}. 
This ends the proof of \eqref{eq27}, which completes the proof of  \eqref{nonlincont3}.
\end{proof}

\subsection{Monotonicity of $W$ along the nonlinear evolution}
Before proving Theorem \ref{unifpropchaos}, we
show that the distance $W$ appearing in that statement is monotone along the  semigroup. We refer to \cite{tanaka} for a similar argument in the case of kinetic models. 
It will be convenient to work with a more symmetric version of \eqref{eq5}.
Since $p\circ q = \sum_A \bar\nu(A) (p\circ q)^A$ we may rewrite \eqref{eq5} as
\begin{align}\label{eq:symme}
C_{\gamma}(p)= \sum_{\vec A\in \cV_n^{k-1}} \nu(\vec A)
%\prod_{i=1}^{k-1}\bar \nu(A_i) 
\hat{C}^{\vec A}_{\gamma}(p)
\end{align}
where $\nu(\vec A)$ is defined as in \eqref{eq:nua}, and $\hat{C}^{\vec A}_{\gamma}(p)$ represents the symmetric convolutions associated to the sampled sets $(A_1,\dots, A_{k-1})$. In other words, 
$\hat{C}^{\vec A}_{\gamma}(p)$ is defined recursively
by the following relations.
If $\gamma \in \Gamma(k)$ then, decomposing $\g$ into the left and right subtrees $\g_l,\g_r$ as in \eqref{eqreca},  one has
\begin{align}\label{lab:rec}
& \hat{C}^{\vec A}_{\gamma}(p) = (\hat{C}^{\vec{A}_{l}}_{\gamma_l}(p)\circ \hat{C}^{\vec{A}_{r}}_{\gamma_r}(p))^{A_1} 
\end{align}
where $A_1$ is the set attached to the root, and $\vec{A}_{l} :=\left(A_2,\dots,A_j\right)$, $\vec{A}_{r} :=\left(A_{j+1},\dots,A_{k-1}\right)$ are the sets associated to the left and right subtrees respectively.
For example, in the case where $C_{\gamma}(p) = ((p \circ p)\circ p)\circ p$ as in Figure \ref{fig1}, one has 
\begin{align}
C_{\gamma}(p) = \sum_{A_1,A_2,A_3 \subset [n]}\bar\nu(A_1)\bar \nu(A_2)\bar \nu(A_3)\hat{C}^{A_1,A_2,A_3}_{\gamma}(p)
\end{align}
where
$\hat{C}^{A_1,A_2,A_3}_{\gamma}(p) = (((p\, \circ\, p)^{A_3}\circ \,p)^{A_2})\circ\, p)^{A_1}$.
Therefore, as in \eqref{eq0} one obtains the decomposition 
\begin{align}\label{eq60}
p_t=\sum_{k=1}^{\infty}\b_t(k)\sum_{\g\in \Gamma(k)}\alpha_k(\g)\sum_{\vec A\in \cV_n^{k-1}}\nu(\vec A) \,\hat C^{\vec A}_{\gamma}(p).
\end{align}

\begin{lemma}\label{wasscontrlemma}
For any  $p, q \in \cP(\O)$, any $k\in\bbN$, any $\g\in\G(k)$, and any $\vec A\in \cV_n^{k-1}$, one has
\begin{align}\label{eq:mon1}
W\left(\hat C^{\vec A}_{\gamma}(p),\hat C^{\vec A}_{\gamma}(q)\right) \leq W\left(p,q\right).
\end{align}
In particular, $W\left(S_t(p),S_t(q)\right) \leq W\left(p,q\right)$ for any $t\geq 0$.
\end{lemma}

%We already know that thanks to \ref{gronwall1} the inequality holds with constant dependent on time. 
\begin{proof}%[Proof of lemma \ref{wasscontrlemma}]
First we show
\begin{align}\label{wildineq1}
W\left((p_1 \circ p_2),(q_1 \circ q_2)\right)\leq \tfrac{1}{2}\,W\left(p_1,q_1\right) + \tfrac{1}{2}\,W\left(p_2,q_2\right) 
\end{align}
for all $p_1,p_2,q_1,q_2 \in \cP(\O).$
For a fixed $A \subset [n]$ and configurations $x_1, x_2 \in \Om$ we define
\begin{align*}
\Pi_{x_1,x_2,A}(z) := \tfrac{1}{2}\left(\IND(z_A = x_{1,A}, z_{A^c}=x_{2,A^c})+\IND(z_A = x_{2,A}, z_{A^c}=x_{1,A^c})\right), \ \ z \in \Om.
\end{align*} 
Now we choose $\{X_1,Y_1\}, \{X_2,Y_2\}$ random variables such that
\begin{itemize}
\item  $\bbE\left[d\left(X_1,Y_1\right)\right] = W(p_1,q_1), \bbE\left[d\left(X_2,Y_2\right)\right] = W(p_2,q_2).$
\item $X_i$ is distributed with $p_i$ and $Y_i$ is distributed with $q_i, i=1,2.$
\item $\{X_1, Y_1\}$ and $\{X_2, Y_2\}$ are independent.
\end{itemize}
Note that 
\begin{align}
\bbE\left[\Pi_{X_1,X_2,A}\right] = \left(p_1 \circ p_2\right)^A.
\end{align}
Now, given $x_1, x_2, y_1, y_2 \in \Om$ and $A \subset [n],$ we consider the following probability measure
\begin{align*}
\pi(z_1,z_2):= &\tfrac{1}{2}\big(\IND(z_{1,A}=x_{1,A}, z_{1,A^c}=x_{2,A^c},z_{2,A}=y_{1,A}, z_{2,A^c}=y_{2,A^c}) \\
& + \IND(z_{1,A}=x_{2,A},z_{1,A^c}=x_{1,A^c},z_{2,A}=y_{2,A}, z_{2,A^c}=y_{1,A^c})\big).
\end{align*}
Note that this is a coupling of $\Pi_{x_1,x_2,A}$ and $\Pi_{y_1,y_2,A}.$  Therefeore
\begin{align}
W\left(\Pi_{x_1,x_2,A},\Pi_{y_1,y_2,A}\right)&\leq \sum_{z_1,z_2}\pi(z_1,z_2)d(z_1,z_2) \nonumber \\ 
&=\tfrac{1}{2}\left(d(x_{1,A}x_{2,A^c},y_{1,A}y_{2,A^c})+d(x_{2,A}x_{1,A^c},y_{2,A}y_{1,A^c})\right) \nonumber  \\
&=\tfrac{1}{2}\left(d(x_1,y_1)+d(x_2,y_2)\right). \label{wassineq}
\end{align}
Moreover, by convexity
\begin{align}
&W\left((p_1\circ p_2),(q_1 \circ q_2)\right)  \leq \sum_{A \subset [n]}\nu(A)W\left((p_1\circ p_2)^A,(q_1 \circ q_2)^A\right)
\end{align}
and, for each $A \subset [n]$,
\begin{align}\label{wildineq1overA}
W\left((p_1\circ p_2)^A,(q_1 \circ q_2)^A\right)  & \leq \bbE\left[W\left(\Pi_{X_1,X_2,A},\Pi_{Y_1,Y_2,A}\right)\right]  %\nonumber\\
%&
=\tfrac{1}{2}\,W(p_1,q_1) + \tfrac{1}{2}\,W(p_2,q_2).
\end{align}
Combining the last two inequalities we obtain \eqref{wildineq1}. 

We now prove \eqref{eq:mon1} by induction over $k\geq 1$. Clearly, $k=1$ is trivial. The case $k=2$ follows by \eqref{wildineq1overA}. Suppose that \eqref{eq:mon1} holds for any $j\leq k-1$, $\g\in\G(j)$, $\vec A\in \cV_n^{j-1}$ and let $A_1$ be the set attached to the root of $\gamma$.
If $\gamma \in \Gamma(k)$ then we decompose $\g$ into the left and right subtrees $\g_l,\g_r$ as in \eqref{eqreca},  so that 
\begin{align*}
& \hat{C}^{\vec A}_{\gamma}(p) = (\hat{C}^{\vec{A}_{l}}_{\gamma_l}(p)\circ \hat{C}^{\vec{A}_{r}}_{\gamma_r}(p))^{A_1} \\
&\hat{C}^{\vec A}_{\gamma}(q) = (\hat{C}^{\vec{A}_{l}}_{\gamma_l}(q)\circ \hat{C}^{\vec{A}_{r}}_{\gamma_r}(q))^{A_1} .
\end{align*}
where $\vec{A}_{l} :=\left(A_2,\dots,A_j\right)$ and $\vec{A}_{r} :=\left(A_{j+1},\dots,A_{k-1}\right),$ and $j$ is the number of leaves of $\g_l.$
Then, using  \eqref{wildineq1overA} again and the inductive hypothesis we have
\begin{align*}
W(\hat{C}^{\vec A}_{\gamma}(p),\hat{C}^{\vec A}_{\gamma}(q))  &= W( (\hat{C}^{\vec{A}_{l}}_{\gamma_l}(p)\circ \hat{C}^{\vec{A}_{r}}_{\gamma_r}(p))^{A_1},(\hat{C}^{\vec{A}_{l}}_{\gamma_l}(q)\circ \hat{C}^{\vec{A}_{r}}_{\gamma_r}(q))^{A_1}) \\
& \leq \tfrac{1}{2}\,W(\hat{C}^{\vec{A}_{l}}_{\gamma_l}(p),\hat{C}^{\vec{A}_{l}}_{\gamma_l}(q)) + \tfrac{1}{2}\,W(\hat{C}^{\vec{A}_{r}}_{\gamma_r}(p), \hat{C}^{\vec{A}_{r}}_{\gamma_r}(q)) 
\leq W(p,q).
\end{align*}
This proves \eqref{eq:mon1}. 

By \eqref{eq:symme} and convexity, %it also shows that
\begin{align}\label{wildineq3}
W(C_{\gamma}(p),C_{\gamma}(q))\leq W(p,q).
\end{align}
Finally, using the expression \eqref{eq60} for $S_t(p)=p_t$, again by convexity we obtain $$  W(S_t(p),S_t(q))\leq W(p,q),$$ for all $t\geq 0$.
\end{proof}

\begin{remark}\label{rem:mono}
We note that \eqref{wildineq1} shows in particular that 
\begin{align}\label{wildineq21}
W\left(p\circ p,q\circ q\right) \leq W\left(p,q\right).
\end{align}
This monotonicity  is not satisfied by the total variation distance. For example let us consider $\Om=\{0,1\}^2,$ $p:=\IND_{(1,1)}, q:=\frac 12(\IND_{(1,1)}+\IND_{(0,0)})$ and suppose $\nu$ is the uniform crossover. Then one can check that $p \circ p = \IND_{(1,1)}$ and $q \circ q = \frac{3}8(\IND_{(1,1)}+\IND_{(0,0)}) + \frac{1}8(\IND_{(1,0)}+\IND_{(0,1)}),$ and therefore
$\|p \circ p - q \circ q\|_\tv = \frac 58 > \frac 12 = \|p-q\|_\tv.$ 
\end{remark}

\subsection{Proof of Theorem \ref{unifpropchaos}}
%We define $F_t^N\in\cP(\cP(\O))$ and $\bar{F}_t^N\in\cP(\cP(\O))$ by 
%\begin{align}
%F_t^N(\nu)= \mu_{N,t}(\lambda_{X_t^N}=\nu)\,,\qquad \bar F_t^N(\nu)= \mu_{N}(S_t\lambda_{X_0^N}=\nu),
%\end{align}
For $k \in \bbN,$ the moment measures $F_t^{k,N}\in\cP(\O^k)$ and $\bar{F}_t^{k,N}\in\cP(\O^k)$ are defined by
\begin{align}
&F_t^{k,N}
%\left(\phi_k\right)
%:=\sum_{\eta'\in\Om}
%\mu_{N,t}(\lambda_{X_t^N}=\l_{\eta'})(\lambda_{\eta'})^{\otimes k}\left(\phi_k\right) = 
=\sum_{\eta\in\Om^N}
\mu_{N,t}(\eta)(\lambda_{\eta})^{\otimes k}
%\left(\phi_k\right)
, \qquad
\bar F_t^{k,N}
%\left(\phi_k\right) 
%:=\int_{\cP(\O)}\bbP_{\mu_N}(S_t\lambda_{X_0^N}\in d\nu)\nu^{\otimes k}\left(\phi_k\right)
= \sum_{\eta\in\Om^N}
\mu_{N}(\eta)(S_t\lambda_{\eta})^{\otimes k}.
%\left(\phi_k\right)
\end{align} 
%for all functions $\phi_{k} :\Om^k\rightarrow\bbR.$

%\begin{proof}[Proof of Theorem \ref{unifpropchaos}]
By the triangular inequality we have
\begin{align}\label{labo}
&|P_k\mu_{N,t}(\phi_k)-p_t^{\otimes k}(\phi_k)|\leq \mathcal{T}_1 + \mathcal{T}_2 + \mathcal{T}_3\,,
\end{align}
where 
\begin{align}\label{labt}
\cT_1:=|P_k\mu_{N,t}(\phi_k)-F_t^{k,N}(\phi_k)| \,,\quad\cT_2:= |F_t^{k,N}(\phi_k)-\bar{F}_t^{k,N}(\phi_k)|\,,\quad\cT_3:= |\bar{F}_t^{k,N}(\phi_k)-p_t^{\otimes k}(\phi_k)| 
\end{align}
We are going to estimate each term separately. 

The term $\mathcal{T}_1$ can be estimated by using a simple combinatorial argument, and one obtains
\begin{align}\label{labt1}
\mathcal{T}_1 \leq \frac{2k(k-1)}{N}\|\phi_k\|_{\infty}.
\end{align}
Since the proof is identical to that in \cite[ Lemma 3.11]{revchaos} we omit the details. 

The last term $\mathcal{T}_3$ is estimated using the initial chaos. 
Note that 
\begin{align}
p_t^{\otimes k}\left(\phi_k\right)-\bar{F}_t^{k,N}\left(\phi_k\right)
&
=\sum_{\eta \in \Om^N}\mu_N(\eta)\left(\prod_{i=1}^k\scalar{S_t(p)}{\phi^i}
\,-\,\prod_{i=1}^k\scalar{S_t(\lambda_{\eta})}{\phi^i}\right).
\end{align}
Using $\prod_{i=1}^ka_i -\prod_{i=1}^kb_i = \sum_{i=1}^k(\prod_{1 \leq j < i}a_j)(a_i-b_i)(\prod_{i<j\leq k}b_j)$,
%\begin{align}\label{prodeq}
%\prod_{i=1}^ka_i -\prod_{i=1}^kb_i = \sum_{i=1}^k\left(\prod_{1 \leq j < i}a_j\right)(a_i-b_i)\left(\prod_{i<j\leq k}b_j\right),
%\end{align}
we obtain
\begin{align*}
\mathcal{T}_3 &
\leq\sum_{\eta \in \Om^N}\mu_N(\eta)\sum_{i=1}^k\Bigg(\prod_{1 \leq j < i}|\scalar{S_t(p)}{\phi^j}|\Bigg)\left|\scalar{S_t(p)-S_t(\lambda_{\eta})}{\phi^i}\right| 
\Bigg(\prod_{i<j\leq k}|\scalar{S_t(\lambda_{\eta})}{\phi^j}|\Bigg) \\
%&\leq k\|\phi_k\|_{\infty}\sum_{\eta \in \Om^N}\mu_N(\eta)W_1\left(S_t(p_0),S_t(\lambda_{\eta})\right) \\
&\leq 2k\|\phi_k\|_{\infty}\,\mu_N \left[W\left(S_t(p),S_t(\lambda_{\chi^N})\right)\right]\leq 
2k\|\phi_k\|_{\infty}\,\mu_N \left[W\left(p,\lambda_{\chi^N}\right)\right],
\end{align*}
where $\chi^N$ has distribution $\mu_N$, and we have used 
\begin{align}\label{eq:sta}
|\scalar{\mu-\mu'}{f}|\leq 2\|f\|_\infty \|\mu-\mu'\|_\tv\leq 2\|f\|_\infty W(\mu,\mu'),
\end{align} 
for any $\mu,\mu'\in\cP(\O)$, $f:\O\mapsto\bbR$, and  the monotonicity from Lemma \ref{wasscontrlemma}.

The estimate of the second term $\cT_2$ is more delicate. Here we use the contraction proved in Theorem \ref{contrlemma2} and Theorem \ref{wasscontrlemma2}. 
Define, for $s\in[0,t]$,
$$\Psi^k_s = \sum_{\eta\in\O^N}\mu_{N,t-s}(\eta)(S_s(\l_{\eta}))^{\otimes k}.$$
Then $\Psi^k_s\in\cP(\O^k)$ for all $s\in[0,t]$, and 
$$F_t^{k,N}(\phi_k)=\Psi^k_0(\phi_k)\,,\qquad \bar{F}_t^{k,N}(\phi_k)=\Psi^k_t(\phi_k).$$ 
Therefore,
\begin{align}\label{eq:inpri}
\bar{F}_t^{k,N}(\phi_k) -F_t^{k,N}(\phi_k)= \int_0^t \partial_s \left[\Psi^k_s(\phi_k)\right] \,ds\,.
\end{align}
Now, 
\begin{align}\label{eq:inpria}
& \partial_s \left[\Psi^k_s(\phi_k)\right] =\sum_{\eta\in\O^N} (\partial_s\mu_{N,t-s}(\eta)) (S_s(\l_{\eta}))^{\otimes k}(\phi_k) + \sum_{\eta\in\O^N} \mu_{N,t-s}(\eta)\partial_s (S_s(\l_{\eta}))^{\otimes k}(\phi_k)\\
 &= -\sum_{\eta\in\O^N} \mu_{N,t-s}(\eta)\cL_N\prod_{\ell=1}^k\scalar{S_s(\l_{\eta})}{\phi^\ell} + 
 \sum_{\eta\in\O^N} \mu_{N,t-s}(\eta)
 \sum_{i=1}^k\scalar{Q(S_s(\l_{\eta}))}{\phi^i}\prod_{\ell\neq i}\scalar{S_s(\l_{\eta}}{\phi^\ell} .
  \nonumber\end{align}
Next, we show that 
\begin{align}\label{eq:lintos}
 Q(S_s(\l_{\eta}))
 =\frac1{N} \sum_{i<j}\sum_A\nu(A)\bar S_s(\l_\eta)(\l_{\eta^{i,j,A}} -\l_{\eta}).
\end{align}
Consider the linearized equation \eqref{linsem}. Taking $q\in\cP(\O)$ and $h= Q(q)$ one has that the solution $h_t=\bar S_t(q)(h)$ satisfies $h_t = Q\left(S_t(q)\right)$, that is
\begin{align}\label{eq:qqq}
\bar S_t(q)(Q(q)) = Q\left(S_t(q)\right)\,,\qquad t\geq 0,
\end{align}
for all $q\in\cP(\O)$. 
To see this note that for all $t\geq 0$,
\begin{align*}
\sum_{\sigma \in \Om}S_t(q)(\sigma) = 1\,,\qquad \sum_{\sigma \in \Om}Q(S_t(q))(\sigma) = 0\,,
\end{align*}
and therefore
\begin{align*}
&\partial_t Q\left(S_t(q)\right)(\sigma') 
= \sum_{\substack{A \subset [n] \\ \sigma \in \Om}}\nu(A)\bigg(Q(S_t(q))(\sigma_{A^c}  \sigma'_{A}) S_t(q)(\sigma_A  \sigma'_{A^c}) \\
&\qquad \qquad \qquad\qquad \qquad \qquad+S_t(q)(\sigma_{A^c}  \sigma'_{A})Q(S_t(q))(\sigma_A  \sigma'_{A^c}) - Q(S_t(q))(\sigma')S_t(q)(\sigma)\bigg) \\
&\qquad= \sum_{\substack{A \subset [n] \\ \sigma \in \Om}}\nu(A)\bigg(Q(S_t(q))(\sigma_{A^c}  \sigma'_{A}) S_t(q)(\sigma_A  \sigma'_{A^c}) +S_t(q))(\sigma_{A^c}  \sigma'_{A})Q(S_t(q))(\sigma_A  \sigma'_{A^c}) \\
&\qquad \qquad\qquad \qquad \qquad\qquad - Q(S_t(q))(\sigma')S_t(q)(\sigma)-Q(S_t(q))(\sigma)S_t(q)(\sigma')\bigg) \\
&\qquad=2\hat{Q}(S_t(q),Q(S_t(q)))(\sigma').
\end{align*}
This proves \eqref{eq:qqq}. If $q=\l_{\eta}$ for some $\eta\in\O^N$, then
\begin{align}
Q(\l_{\eta}) = \sum_A\nu(A) \left((\l_{\eta})_A\otimes(\l_{\eta})_{A^c} - \l_{\eta}\right).
\end{align}
Moreover, for any $A\subset [n]$,
\begin{align}
(\l_{\eta})_A\otimes(\l_{\eta})_{A^c}- \l_{\eta}&= \frac1{2N^2} \sum_{i,j=1}^N(\ind_{\eta^{i,j,A}(i)} +  \ind_{\eta^{i,j,A}(j)} - \ind_{\eta(i)} -  \ind_{\eta(j)})\\
&=  \frac1{N} \sum_{i<j}(\l_{\eta^{i,j,A}} - \l_{\eta}).
\end{align}
It follows that 
\begin{align}
Q(\l_{\eta}) = \frac1{N} \sum_{i<j}\sum_A\nu(A) (\l_{\eta^{i,j,A}} - \l_{\eta}).
\end{align}
Thus, using \eqref{eq:qqq} and the linearity of $\bar S_s(\l_\eta)(\cdot)$ we obtain
\eqref{eq:lintos}.

On the other hand, 
\begin{align}
\cL_N\prod_{\ell=1}^k\scalar{S_s(\l_{\eta})}{\phi^\ell}   = 
\frac1{N} \sum_{i<j}\sum_A\nu(A) \left(\prod_{\ell=1}^k\scalar{S_s(\l_{\eta^{i,j,A}})}{\phi^\ell}    -\prod_{\ell=1}^k\scalar{S_s(\l_{\eta})}{\phi^\ell}  \right).
\end{align}
%Therefore, 
%\begin{align}
%\partial_s \Psi_s=-\frac1{N} \sum_{i<j}
%\mu_{N,t-s}[S_s(\l_{\eta^{i,j,A}})) -S_s(\l_{\eta}) - \bar S_s(\l_\eta)(\l_{\eta^{i,j,A}} -\l_{\eta})].
%\end{align}
Suppose we can show that 
\begin{align}\label{Ks}
&\Big|\prod_{\ell=1}^k\scalar{S_s(\l_{\eta^{i,j,A}})}{\phi^\ell}    -\prod_{\ell=1}^k\scalar{S_s(\l_{\eta})}{\phi^\ell}  \nonumber \\ & \quad - \sum_{u=1}^k\scalar{\bar S_s(\l_\eta)(\l_{\eta^{i,j,A}} -\l_{\eta})}{\phi^u}\prod_{\ell\neq u}\scalar{S_s(\l_{\eta}}{\phi^\ell} 
\Big|\leq K_s,
\end{align}
for some function $K_s$, $s\in[0,\infty)$, independent of $\eta,i,j,A$ . Then by \eqref{eq:inpria} and \eqref{eq:lintos} we would have $\cT_2\leq \frac{N}2\int_0^tK_s ds$. Thus the proof will be completed by proving \eqref{Ks} for a suitable $K_s$. 

 Observe that
 \begin{align}\label{eq:inpriaa3}
 &\prod_{\ell=1}^ka_\ell -\prod_{\ell=1}^kb_\ell - \sum_{u=1}^kc_u\prod_{j\neq u}b_j
  = \sum_{u=1}^k\left[(\prod_{1 \leq j < u}a_j)(a_u-b_u)
- (\prod_{1 \leq j < u}b_j)c_u\right]\prod_{j> u}b_j\\
&\qquad  = \sum_{u=1}^k(a_u-b_u-c_u)\prod_{j\neq u}b_j + 
\sum_{u=1}^k\left[(\prod_{1 \leq j < u}a_j)
- (\prod_{1 \leq j < u}b_j)\right](a_u-b_u)\prod_{j> u}b_j
\\
&\qquad  = \sum_{u=1}^k(a_u-b_u-c_u)\prod_{j\neq u}b_j + 
\sum_{u=1}^k\sum_{v=1}^{u-1}(\prod_{j < v}a_j)(\prod_{v<j< u}b_j)(a_v-b_v)(a_u-b_u)\prod_{j> u}b_j.
\end{align}
Define $a_\ell = \scalar{S_s(\l_{\eta^{i,j,A}})}{\phi^\ell} $, $b_\ell=\scalar{S_s(\l_{\eta})}{\phi^\ell}  $, and $c_\ell = \scalar{\bar S_s(\l_\eta)(\l_{\eta^{i,j,A}} -\l_{\eta})}{\phi^\ell}$.
Noticing that
\begin{align}
\|\l_{\eta^{i,j,A}} -\l_{\eta}\|_\tv\leq \frac2N,
\end{align}
and that $\l_{\eta^{i,j,A}},\l_{\eta}$ have the same marginals, it follows from Theorem \ref{wasscontrlemma2} %and \eqref{eq:lintos} 
that
%for any $\phi^i:\O\mapsto\bbR$,
\begin{align}\label{eq:abc1}
|a_u-b_u-c_u|\leq n^5\,\|\phi^u\|_\infty e^{-D(\nu)\, s} \frac{4}{N^2}.
%\leq \frac{2}{N}\,n^5B\,\|\phi^i\|_\infty e^{-D_\e s},
 \end{align} 
 Moreover, Theorem \ref{contrlemma2}
 shows that for all $u,v$,
\begin{align}\label{eq:abc2}
|a_v-b_v||a_u-b_u|\leq n^4\,\|\phi^v\|_\infty\,\|\phi^u\|_\infty e^{-2D(\nu)\, s} \frac{4}{N^2}.
%\leq \frac{2}{N}\,n^5B\,\|\phi^i\|_\infty e^{-D_\e s},
 \end{align} 
 These bounds hold uniformly in $\eta\in\O^N,1\leq i<j\leq n,A\subset[n]$.  This proves \eqref{Ks}  with 
   \begin{align}\label{eq:inprinc}
  K_s =ke^{-D(\nu)\, s}( 
  n^5 + (k-1)n^4e^{-D(\nu)\, s}) \frac{\|\phi_k\|_\infty }{N^2}.
   \end{align}
   Therefore,
  \begin{align}\label{eq:inprinc2}
   \cT_2\leq \frac{N}2\int_0^tK_s \leq \frac{k^2n^5}{D(\nu) N}\,\|\phi_k\|_\infty (1-e^{-D(\nu)\, t}).
\end{align}
The proof of Theorem \ref{unifpropchaos} is complete.
 
\section{Relative entropy decay %in the particle system
}\label{sec:entdec}
The goal of this section is to prove Theorem \ref{th:ent}, Proposition \ref{prop:upbo}, and Theorem \ref{entprodthm}. We start by setting up a convenient notation and by gathering some preliminary ingredients of the proof. To simplify our notation we write $\mu=\pi_N$ for the uniform distribution over $\cS_{N,n}$. For any $A\subset [n]$, $f:\cS_{N,n}\mapsto\bbR$, we are going to use the notation $$\mu_A f =\mu(f\tc \eta_{A^c}),$$ 
where $\mu(\cdot\tc \eta_{A^c})$ denotes the conditional expectation given the variables $$\eta_{A^c}:=\{\eta_i(j), \;i\in A^c, \;j=1,\dots,N\}.$$ The function $\mu_Af$ thus depends on $\eta$ only through the variabes $\eta_{A^c}$. When $A=[n]$ we have the global expectation $\mu_{[n]} f = \mu(f)$. With this notation, $\mu_A$ is the  orthogonal projection, in $L^2(\cS_{N,n},\mu)$, onto the space of functions that do not depend on the $A$-component of $\eta\in\cS_{N,n}$. %are symmetric w.r.t.\ permutations restricted to the subset $A$.
Notice the relations $\mu(f)=\mu(\mu_A f)$, and $\mu_B(f)=\mu_B(\mu_A f)$, for all $A\subset B\subset [n]$. 
 We also use the notation $\ent_A(f)$ for the entropy of $f:\cS_{N,n}\mapsto\bbR_+$ with respect to $\mu_A$, that is 
 \begin{equation}
\label{entA}
\ent_A(f) = \mu_A\left[f\log(f/\mu_A f)\right]\,.
\end{equation}
Thus $\ent_A(f)$ is a function that depends on $\eta$ through the variables $\eta_{A^c}$ only. Its expectation satisfies
  \begin{equation}
  \label{entA2}
\mu\left[\ent_A(f)\right] = \ent(f) - \ent[\mu_Af]\,,
\end{equation}
where $\ent(f) = \ent_{[n]}(f) = \mu\left[f\log(f/\mu (f))\right]$.
In our setting the well known Shearer inequality can be formulated as follows. 
\begin{lemma}\label{lem:shearer}
For any distribution $\nu$ on subsets of $[n]$, for any $f:\cS_{N,n}\mapsto\bbR_+$, 
  \begin{equation}
  \label{shearer}
  \sum_{A\subset [n]}\nu(A) \,\mu\left[\ent_A(f)\right] \geq \g(\nu)\, \ent(f)\,,
\end{equation}
where $\g(\nu) = \min_{i\in[n]}\sum_{A:\,A\ni i}\nu(A)$. 
\end{lemma}
\begin{proof}
Reasoning as in \cite[Proposition 4.3]{entprod}, the estimate is a consequence of the product structure of $\mu$ and the weighted version of the classical Shearer inequality for Shannon entropy. We refer e.g.\  to \cite{Virag} for a proof of the latter. 
\end{proof} 
When $A=\{i\}$, a single site, we write $\mu_i$ for $\mu_{\{i\}}$ and $\ent_i$ for $\ent_{\{i\}}$. An immediate consequence of Lemma \ref{lem:shearer} is the following well known tensorization property:
  \begin{equation}
  \label{tensor}
\ent(f) \leq\sum_{i=1}^n \mu\left[\ent_{i}(f) \right].
\end{equation}
\subsection{The case $n=1$}
A key ingredient of our proof is the control of the base case $n=1$. Here the problem reduces to standard random transpositions and one can use a well known bound that was first derived in \cite{Goel,GQ}. In our setting it can be summarized as follows, see \cite[Theorem 1]{GQ} or \cite[Corollary 3.1]{Goel} for a proof. Note that for $n=1$ we have $\cS_{N,n} =S_N$.
 \begin{lemma}\la{GQlemma} for all $N\geq 2$, for all $g:S_N\mapsto\bbR_+$, for any $i=1,\dots,n$,
\begin{align*}
%\frac1{N!}
\sum_{\t\in S_N}
&g(\t)\log(g(\t)/\bar g) %\\&
\leq \frac1N\sum_{j<\ell}
%\frac1{N!}
\sum_{\t\in S_N}
\left(g(\t^{j,\ell}) - g(\t)\right)\log \frac{g(\t^{j,\ell})}{g(\t)},
\end{align*}
where $\bar g = \frac1{N!}\sum_{\t\in S_N}g(\t)$, and $\t^{j,\ell}$ denotes $\t$ composed with the transposition at $\{j,\ell\}$.
%
%\begin{equation}
%\label{gq1}
%\frac1{N}\sum_{1\leq j<j'\leq N}\mu_i\left[(f^{j,j',\{i\}}-f)\log\frac{f^{j,j',\{i\}}}{f}\right]
%\geq \ent_i(f).
%\end{equation}
%
\end{lemma}
The proof of Theorem \ref{th:ent} is based on Lemma \ref{lem:shearer}, Lemma \ref{GQlemma} and the following analysis of the entropy associated to partial random permutations of the particle configuration, which is the main novelty in this section.  

\subsection{Permutation entropies}
%Let $S_N$ denote the set of all permutations of $[N]$, and f
For any $A\subset [n]$, $g:\cS_{N,n}\mapsto\bbR$, define  $P_Ag:\cS_{N,n}\mapsto \bbR$, as
 \begin{align}\label{defpA}
P_Ag(\eta) = \frac1{N!}\sum_{\t\in S_N}g((\t\eta)_A\,\eta_{A^c}),
\end{align}
where $\t\eta:=\t\circ\eta$ denotes the element of $\cS_{N,n}$ obtained from $\eta$ by permuting the particle labels according to $\t$:
$$
(\t\eta)(j) = \eta(\t(j)).
$$ 
The linear operator $P_A$ is the orthogonal projection, in $L^2(\cS_{N,n},\mu)$, onto the space of functions that are symmetric w.r.t.\ permutations restricted to the subset $A$. When $A=\{i\}$ we write $P_{\{i\}}=P_i$, and note %With this notation one has 
%also the identities 
%\begin{align}\label{ent06}
%P_Af(\eta) = \mu\left(f\tc \bar\eta_A\eta_{A^c}\right).
%\end{align}
%Moreover, 
that $P_{i}g = \mu(g\tc \eta_{\{i\}^c})=\mu_i g$ for all $i$. Note also that $P_A,P_B$ do not commute for general $A,B\subset[n]$, but if $A\cap B=\emptyset$ then $P_AP_B=P_BP_A$. Moreover, one easily checks that  if $A\subset B\subset [n]$, then
 \begin{align}\label{commPA}
P_AP_B = P_{B\setminus A}P_A = P_AP_{B\setminus A}=P_BP_A,\qquad \;A\subset B.
\end{align}
Also, observe that the orthogonal projection $\mu_A$ defined above satisfies %writing $\mu_Af(\eta)=\mu(f\tc \eta_{A^c})$, one has that
 \begin{align}\label{muAPA}
 \mu_A = \prod_{i\in A}P_i.
 \end{align}
Notice that $P_{[n]}g = g$ iff $g\in\bbS$, and that in this case
\begin{align}\label{enta+7}
P_Ag = P_{A^c}g = P_AP_{A^c}g=P_{A^c}P_{A}g\,,\qquad g\in\bbS.
%\frac1{N!}\sum_{\t} f((\t\si)_A\si_{A^c})\,,\qquad \si\in\O^N ,
\end{align}
For any fixed $f\geq 0$, and $A\subset[n]$, define % $A\mapsto\phi(A;f)$ by
\begin{align}\label{phia}
\phi(A;f) = \mu\left[f\log(f/P_Af)
\right].
\end{align}
The function $\phi$ is a suitable conditional entropy of $f$ on $A$. Namely, $\phi(A;f)$ is the expected value with respect to $\mu$ of the entropy 
$$
P_A\left[f\log(f/P_Af)\right].
$$
In particular, $\phi(A;f) \geq 0$.
 We call $\phi(A;f) $ the {\em permutation entropy} of $f$ on $A$.
%, and it coincides with the entropy $\ent_i(f)$ if $A=\{i\}$, for some $i\in[n]$.  On the other hand $\phi([n])=0$ if $f\in\bbS$. 
When there is no risk of confusion we write simply $\phi(A)$ for $\phi(A;f)$. In general, one has 
\begin{lemma}\la{propphiA} 
Fix a function $f\geq 0$. For any $i\in[n]$, $\phi(\{i\}) =  \mu\left[\ent_i(f)\right]$,
and for all $A\subset[n]$:
\begin{equation}
\label{phia2}
\phi(A) = \mu\left[\ent_A(f) \right] - \mu\left[\ent_A(P_Af) \right] \,.
\end{equation}
Moreover,  for all $A\subset[n]$
\begin{equation}
\label{phiaGQ}
\phi(A) \leq  2\cE_A(f,\log f),
\end{equation} 
where 
\begin{align}\label{enta+5}
\cE_A(f,\log f)=\frac12\frac1N\sum_{j<\ell}\mu\left[\left(f^{j,\ell,A} - f\right)%\left(
\log \frac{f^{j,\ell,A}}{f} % - \log f%\right)
\right].
\end{align}
\end{lemma}
\begin{proof}
If $A=\{i\}$, then $P_A f = \mu_i f$ %P_{\{i\}}f  = \mu(f\tc \eta_{\{i\}^c}) $ 
and therefore $\phi(\{i\}) =  \mu\left[\ent_i(f)\right]$. In general, for any $A$, %writing $\mu_Af(\eta)=\mu(f\tc \eta_{A^c})$ 
\begin{align*}
\mu\left[\ent_A(f) \right]&= \mu\left[f\log(f/\mu_Af)\right] \nonumber\\& 
=\mu\left[f\log(f/P_Af)\right] + \mu\left[f\log(P_Af/\mu_Af)\right].
\end{align*}
The second term above satisfies 
$$
 \mu\left[f\log(P_Af/\mu_Af)\right]=  \mu\left[P_A f\log(P_Af/\mu_A P_Af)\right] = \mu\left[\ent_A(P_Af) \right] \,,
$$
where we have used  that $\mu_BP_A=\mu_B$ for any $A\subset B\subset[n]$ since $P_A$ is an orthogonal projection. This proves \eqref{phia2}.

To prove \eqref{phiaGQ}, observe that for any $A\subset [n]$, $f\geq 0$, any fixed $\eta\in\cS_{N,n}$, taking $g(\t):=f((\t\circ\eta)_A\eta_{A^c})$, one has $\bar g = P_Af$, and therefore by Lemma \ref{GQlemma}, 
\begin{align*}
&P_A\left[f\log(f/P_Af)\right] (\eta)=\frac1{N!}\sum_{\t\in S_N}
g(\t)\log(g(\t)/\bar g)\\&
\leq \frac1N\sum_{j<\ell}\frac1{N!}\sum_{\t\in S_N}\left(f((\t^{j,\ell}\eta)_A\eta_{A^c}) - f((\t\eta)_A\eta_{A^c})\right)\log \frac{f((\t^{j,\ell}\eta)_A\eta_{A^c})}{f((\t\eta)_A\eta_{A^c})}\\&=
\frac1N\sum_{j<\ell}P_A\left[\left(f^{j,\ell,A} - f\right)\log \frac{f^{j,\ell,A}}f
\right](\eta).
\end{align*}
Taking the expectation and using $\mu P_A=\mu$ one obtains \eqref{phiaGQ}.
\end{proof}
Next, we compare the permutation entropy $\phi(A)$ with the usual entropy $\mu\left[\ent_A(f) \right]$. 
The previous lemma in particular shows that $\phi(A)\leq \mu\left[\ent_A(f) \right]$.
The next lemma allows us to give a bound in the opposite direction. Notice that such a bound needs some care since if e.g.\ $f\in\bbS$ is not a constant then $\phi([n];f)=0$ while $\mu\left[\ent_{[n]}(f) \right]=\ent(f)\neq 0$. 
  \begin{lemma}\la{keyphi} 
Fix  $A\subset [n]$ such that $0\leq |A|\leq n-1$, and $V\subset A^c$.
Then, 
\begin{equation}
\label{phia5}
\phi(A) + \sum_{i\in V}\phi(A\cup\{i\})\geq  \mu\left[\ent_{V}f \right].
\end{equation}
\end{lemma}
\begin{proof}
Write $V=\{x_1,\dots,x_k\}\subset [n]$, $k=|V|$, and define $A_0=A$, and
$A_i=A\cup\{x_i\}$, $i=1,\dots,k$. In general $P_{A_i}$ and $P_{A_j}$ do not commute for $i,j=1,\dots,k$, but using \eqref{commPA} one has $P_{A_0}P_{A_i} = P_{A_0}P_{x_i}= P_A\mu_{x_i}$ and for any $\ell=1,\dots,k$, %one checks that the following holds: 
$$
P^\ell:=P_{A_0}P_{A_1}\cdots P_{A_\ell} = P_A\prod_{i=1}^\ell \mu_{x_i}= \left(\prod_{i=1}^\ell \mu_{x_i} \right)P_A %\mu_{V}=\mu_VP_A
=\prod_{i=0}^\ell P_{A_i}\,,
$$
where the last identity holds regardless of the order of the multiplications. 
In other words,  the operators $P_{A_i}$ commute thanks to the presence of $P_{A_0}$. In particular, $P^k=P_A\mu_V=\mu_VP_A$.

Consider the entropy
\begin{align}\label{eq:summa}
\mu\left[f\log\frac{f}{P^kf}\right] = \mu\left[f\log\frac{f}{\mu_Vf}\right] + \mu\left[f\log\frac{\mu_Vf}{P_A\mu_Vf}\right].
\end{align}
The first term is $\mu\left[\ent_Vf\right]$. The second term satisfies
\begin{align}\label{eq:summaa}
 \mu\left[f\log\frac{\mu_Vf}{P_A\mu_Vf}\right] =  \mu\left[\mu_Vf\log\frac{\mu_Vf}{\mu_VP_Af}\right]
 = \phi(A;\mu_Vf)\geq 0. 
\end{align}
On the other hand, 
$$
\log\frac{f}{P^kf}=\log\frac{f}{P_{A_0} f} + \log\frac{P_{A_0}f}{P_{A_0}P_{A_1}f}+\cdots+\log\frac{P_{A_0}P_{A_1}\cdots P_{A_{k-1}} f}{P_{A_0}P_{A_1}\cdots P_{A_k} f},
$$
and therefore 
\begin{align}\label{eq:summ}
\mu\left[f\log\frac{f}{P^kf}\right] = \phi(A;f)+\sum_{\ell=0}^{k-1}\mu\left[f\log\frac{P^\ell
f}{P^{\ell+1}f}\right].
\end{align}
Next, we show that 
\begin{equation}
\label{phia05}
\mu\left[f\log\frac{P^\ell
f}{P^{\ell+1}f}\right]
%\phi\left(A_{\ell+1};P^\ell f\right)
\leq \phi\left(A_{\ell+1};f\right),
\end{equation}
for all $\ell=0,\dots,k-1$. Combined with \eqref{eq:summa}-\eqref{eq:summ}, this implies 
the desired conclusion:
\begin{equation}
\label{phia5x}
\phi(A) + \sum_{i\in V}\phi(A\cup\{i\})\geq \mu\left[\ent_{V}f \right] + \phi(A;\mu_V f)\geq \mu\left[\ent_{V}f \right].
\end{equation}
It remains to prove \eqref{phia05}.
We first observe that %The sum in \eqref{eq:summ} can be rewritten as
\begin{equation}
\label{phia5xx}
\mu\left[f\log\frac{P^\ell
f}{P^{\ell+1}f}\right]=\mu\left[P^\ell
f\log\frac{P^\ell f}{P^{\ell+1}f}\right]=\phi\left(A_{\ell+1};P^\ell f\right).
\end{equation}
The well known variational principle for the relative entropy implies  that for any $B\subset [n]$
$$
P_B\left(f\log\frac{f}{P_Bf}\right)\geq P_B(fg)\,,
$$
for any function $g$ such that $P_B(e^g)=1$. %For any $C, B\subset[n]$, c
Choosing $g= \log\frac{P^\ell f}{P_BP^\ell f}$ shows that 
$$
P_B\left(f\log\frac{f}{P_Bf}\right)\geq P_B\left(f\log\frac{P^\ell f}{P_BP^\ell f}\right)\,.
$$
Taking the expectation one finds 
$$
\phi(B;f)\geq\mu\left[f\log\frac{P^\ell f}{P_BP^\ell f}
\right]\,.
$$
If $P_B,P^\ell $ commute, then 
$$
\mu\left[f\log\frac{P^\ell f}{P_BP^\ell f}
\right]= \mu\left[f\log\frac{P^\ell f}{P^\ell P_Bf}
\right]=\mu\left[P^\ell f\log\frac{P^\ell f}{P^\ell P_Bf}
\right] = \phi(B;P^\ell f).
$$
Therefore,
$$
\phi(B;f)\geq\phi(B;P^\ell f),
$$
whenever $P_B,P^\ell $ commute. %Repeating the same argument with $P_C$ replaced by $P^{\ell}$, 
Taking $B=A_{\ell+1}$, and using the  fact that $P_{A_{\ell+1}}$ and $P^{\ell}$ commute, one obtains
$\phi\left(A_{\ell+1};P^\ell f\right)\leq \phi\left(A_{\ell+1}; f\right)$. Together with \eqref{phia5xx}, this implies \eqref{phia05}. 
\end{proof}
%
%Once we have Lemma \ref{propphiA} 
% and Lemma \ref{keyphi}  the proof of the main results is not difficult. 
\subsection{Proof of Theorem \ref{th:ent}}
%From Lemma \ref{propphiA} and 
From the strict separation assumption, it follows that for some $\d(\nu)>0$, 
$$
\cE_{N,n}(f,\log f)
\geq \frac{\d(\nu)}n\sum_{i=1}^n\left(\cE_{A_i}(f,\log f)+\cE_{A_i\setminus\{i\}}(f,\log f)\right),
%\geq \sum_{A}\nu(A) \phi(A). 
$$
where, for every $i$,  $A_i\subset[n]$ is such that $A_i\ni i$ and $\min\{\nu(A_i),\nu(A_i\setminus\{i\})\}\geq \d(\nu)$.  
Therefore, from Lemma \ref{propphiA},
 \begin{equation}
\label{phia5a1}
\cE_{N,n}(f,\log f)
\geq \frac{\d(\nu)}{2n}\sum_{i=1}^n\left(\phi(A_i) +\phi(A_i\setminus\{i\})\right).
\end{equation}
Lemma \ref{keyphi}  then implies 
\begin{equation}
\label{phia5a2}
\cE_{N,n}(f,\log f)
\geq \frac{\d(\nu)}{2n}\sum_{i=1}^n\mu\left[\ent_i f\right].
\end{equation}
Using \eqref{tensor} it follows that $\a(N,n)\geq \d(\nu)/{2n}$. This proves the bound \eqref{entao} with $\a(\nu)=\d(\nu)/{2n}$. 

To prove the lower bound for one-point crossover, notice that the above argument can be repeated but this time we can directly estimate
$$
\cE_{N,n}(f,\log f)\geq \frac14\frac1{n+1}\sum_{i=1}^n (\phi(J_i)+\phi(J_{i-1}))\geq  
\frac1{4(n+1)}\sum_{i=1}^n\mu\left[\ent_{i}f \right],
$$
where $J_0=\emptyset$, and $J_i=\{1,\dots,i\}$, $i\geq 1$.
The lower bound $\a(N,n)\geq 1/4(n+1)$ thus follows again by the tensorization \eqref{tensor}.
%$$
%\cE_{N,n}(f,\log f)
%= \frac1{n+1}\sum_{i=1}^n\cE_{J_i}(f,\log f)
%\geq \frac1{n+1}\sum_{i=1}^n \phi(J_i). 
%$$
%By Lemma \ref{keyphi}, we have $ \phi(J_i) + \phi(J_{i-1})\geq \mu\left[\ent_{i}f \right]$, for all $i=1,\dots, n$. Therefore 
%$$
%\cE(f,\log f)\geq \frac12\frac1{n+1}\sum_{i=1}^n (\phi(J_i)+\phi(J_{i-1}))\geq  \frac12\frac1{n+1}\sum_{i=1}^n\mu\left[\ent_{i}f \right],
%$$
%and the conclusion follows by tensorization.
% 

Next, we prove the lower bound for the case of uniform crossover $\nu(A)=2^{-n}$
for all $A\subset[n]$. From Lemma \ref{propphiA} 
$$
\cE_{N,n}(f,\log f)=2^{-n}\sum_A \cE_A(f,\log f)\geq 2^{-n-1}\sum_A \phi(A). 
$$
By Lemma \ref{keyphi},
\begin{align*}
\sum_{i=1}^n\sum_{A}\phi(A)\ind(i\in A)&=  \sum_{i=1}^n\sum_{A: \,|A|\leq n-1}\phi(A\cup\{i\})\ind(i\notin A)
\\&
\geq  \sum_{A: \,|A|\leq n-1} \left(\mu\left[\ent_{A^c}f \right] - \phi(A)\right). 
\end{align*}
Therefore,
\begin{align}\label{enta+43}
 \sum_A \phi(A)&= \frac1n\sum_{i=1}^n\sum_{A}\phi(A)(\ind(i\notin A) +\ind(i\in A))
 \\&\geq \frac1n\sum_{A}|A^c|\phi(A) + \frac1n\sum_{A: \,|A|\leq n-1}\left(\mu\left[\ent_{A^c}f \right] - \phi(A)\right) \\&=\frac1n\sum_{A:\, |A|\leq n-1}(|A^c|-1)\phi(A) + \frac1n\sum_{A}\mu\left[\ent_{A^c}f \right] .
\end{align}
In particular,
$$
2^{-n-1}\sum_A \phi(A)\geq \frac{2^{-n}}{2n}\sum_A\mu\left[\ent_{A^c}f \right]. 
$$
From Lemma \ref{lem:shearer} it follows that 
$$
2^{-n-1}\sum_A \phi(A)\geq \frac{1}{4n}\ent f. 
$$
This proves the desired bound $\a(N,n)\geq \frac{1}{4n}$. 

Finally, in the case of symmetric functions one has  
$\phi(A)=\phi(A^c)$ for all $A\subset[n]$, see \eqref{enta+7}. Therefore, the previous computation now shows that
$$
\sum_A \phi(A)\geq \frac2n\sum_{A}\mu\left[\ent_{A^c}f \right] - \frac2n \sum_A \phi(A).
$$
Rearranging terms yields the bound $\a_\bbS(N,n)\geq \frac{1}{2(n+2)}$. This concludes the proof of Theorem \ref{th:ent}.

\subsection{Proof of Proposition \ref{prop:upbo}}
The proof of the upper bound on $\a(N,n)$  is based on Proposition \ref{dff2tyhm}, Proposition \ref{th:entchaos}, and the following entropy production estimate  for the nonlinear equation that was derived in \cite{entprod}.  
\begin{lemma}\label{lem:entprod}
Let $\O=\{0,1\}^n$, and let $p=p(n)\in \cP(\O)$ be defined as 
$$
p = w^2\d_{\underline 1} + (1-w)^2\d_{\underline 0} + 2w(1-w) \,\cU,
$$
where $w=2^{-n}$ and $\cU$ is the uniform distribution on $\O$, and $\d_{\underline 1}$ and $\d_{\underline 0}$ denote the Dirac mass at the ``all one" and ``all zero" configuration respectively.
Then,  taking $f=p/\pi$, with $\pi=\otimes_{i=1}^np_i$, one has 
\begin{equation}\label{edff2}
 \frac{D_\pi(f)}{\ent_\pi f}\leq \frac{4}{n} + O\left(\frac1{n^{2}}\right). 
\end{equation}
\end{lemma}
\begin{proof}
See \cite[Proposition 4.7]{entprod}. 
\end{proof}
To prove Proposition \ref{prop:upbo} we take $\O=\{0,1\}^n$, and $p$ as in Lemma \ref{lem:entprod}. Note that $p$ has marginals equal to Bernoulli ${\rm Be}(w)$. Take $\r_N$ such that \eqref{hypo1} holds with $\pi$ the product of Bernoulli ${\rm Be}(w)$, %as in Lemma \ref{lem:entprod} 
and write $f_N = \g(p,\r_N)/\g(\pi,\r_N)$. Recall that
 \begin{equation}\label{erdff2}
 \a(N,n)\leq \a(\O_{\r_N})\leq \frac{\cE(f_N,\log f_N)}{\ent f_N},
 \end{equation}
 where $\ent f_N = H_N(\g(p,\r_N)\tc \g(\pi,\r_N))$. 
Clearly, $p$ is irreducible. From Proposition \ref{dff2tyhm} and Proposition \ref{th:entchaos}, %the last expression satisfies 
\begin{equation}\label{errdff2}
 \limsup_{N\to\infty}\,\a(N,n)\leq  \frac{D_\pi(f)}{\ent_\pi f}\leq
 \frac{4}{n} + O\left(\frac1{n^{2}}\right).
 \end{equation}

\subsection{Proof of Theorem \ref{entprodthm}}
Let $p\in\cP(\O)$ be an arbitrary initial value for the Boltzmann equation and let $\pi=\otimes_i p_i$ denote the corresponding equilibrium.  In order to ensure that $p$ be irreducible we write
$$
p^{(\e)} = \e\,\pi + (1-\e)\,p\,
$$
with $\e\in(0,1)$ to be taken to zero eventually. Clearly, $p,p^{(\e)},\pi$ have the same marginals. Let $\r_N$ be an admissible sequence such that \eqref{hypo1} holds. Let $p^{(\e)}_{N} = \g(p^{(\e)},\r_N)$ and $\pi_N = \g(\pi,\r_N)$ and define $p^{(\e)}_{N,t}=p^{(\e)}_{N}e^{t\cL_N}$. The propagation of chaos at fixed time $t$ implies that the hypothesis of Proposition \ref{th:entUPSC} apply to $\mu^{(N)}=p^{(\e)}_{N,t}$. Therefore, 
\begin{align}
H(p^{(\e)}_t\tc\pi) \leq \liminf_{N \to \infty}\frac{H_N(p^{(\e)}_{N,t}\tc\pi_N)}{N},
\end{align}
where $p^{(\e)}_t$ is the solution to the nonlinear equation with initial datum $p^{(\e)}$.
Theorem \ref{th:ent} implies that 
\begin{align}
H_N(p^{(\e)}_{N,t}\tc\pi_N)\leq e^{-\a(\nu)\,t}H_N(p^{(\e)}_{N}\tc\pi_N).
\end{align}
Then an application of Proposition \ref{th:entchaos} shows that for all $\e>0$, $t\geq 0$, one has
\begin{align}
H(p^{(\e)}_t\tc\pi) \leq 
e^{-\a(\nu)\,t}H(p^{(\e)}\tc\pi).
\end{align}
The conclusion follows by taking $\e\to 0^+$. Indeed, $p^{(\e)}\to p$ and, by \eqref{nonlincont1} we know that $p^{(\e)}_t\to p_t$, so that   $H(p^{(\e)}\tc\pi)\to H(p\tc\pi)$, and $H(p^{(\e)}_t\tc\pi)\to H(p_t\tc\pi)$. 
We note that we can use the entropy production for symmetric functions here, so that for instance in the case of uniform crossover we may take $\a(\nu)\geq 1/2(n+2)$.

\appendix

%%%%%%%%%%%%%%%%%%% CLT THEOREM %%%%%%%%%%%%%%%%%%%%%%%

\section{Proof of Proposition \ref{prop:LCLT}}
In order to prove Proposition \ref{prop:LCLT} we adapt to our multivariate setting some classical estimates, see e.g.\ \cite{Chung,Petrov}. For the sake of clarity we give a self-contained proof. Recall the notation from Section \ref{sec:clt_chaos}. In particular, $|t|=\sqrt{\scalar{t}{t}}$ denotes the vector norm of $t\in\bbR^K$, and the random variable $\xi=(\xi_{i,x})$ with distribution $\mu\in\cP(X)$ takes values in $X=\{0,1\}^K$, and has covariance matrix $V_1$. The next lemma only requires the nondegeneracy of $\mu$, that is $\det(V_1)\neq 0$.  The proof of Proposition \ref{prop:LCLT} however requires the irreducibility of $\mu$ in order to apply Lemma \ref{charfunc}. 
\begin{lemma}
\label{BerEs}
Let $\mu\in\cP(X)$ be nondegenerate. For any $t\in\bbR^K$, define $$\phi_N(t):=\mu\left(e^{i\scalar{V_N^{-1/2}t\,}{\,\bar{\xi}}}\right)^N,$$ where $\bar{\xi}_{i,x}:=\xi_{i,x}-\mu\left[\xi_{i,x}\right],$ $V_N=N V_1$, 
%be the characteristic function of the centered random variable $V_N^{-1/2}\sum_{k=1}^N\vecx_k$
and  $L_N:=\frac1{\sqrt N}\,\mu\left[\big|V_1^{-1/2}\bar{\xi}\big|^3\right].$ Then
\begin{align}\label{beres1}
\left|\phi_N(t)-e^{-\frac{1}{2}\scalar{t}{t}}\right|\leq16L_N|t|^3e^{-\frac{1}3\scalar{t}{t}} \,,\qquad |t|\leq \frac{1}{4L_N}.
\end{align}
%if $|t|\leq \frac{1}{4L_N}.$
\end{lemma}
\begin{proof}%[Proof of lemma \ref{BerEs}]
We split the proof into two. First, let us suppose that $ \frac{1}{4L_N}\geq|t|\geq\frac{1}{2}L_N^{-\frac{1}{3}}.$ If one has
\begin{align}\label{beres2}
|\phi_N(t)|^2\leq e^{-\frac{2}{3}\scalar{t}{t}}
\end{align}
then 
\begin{align}
\left|\phi_N(t)-e^{-\frac12\scalar{t}{t}}\right|\leq|\phi_N(t)|+e^{-\frac12\scalar{t}{t}}\leq 2e^{-\frac13\scalar{t}{t}}\leq16L_N|t|^3e^{-\frac13\scalar{t}{t}}.
\end{align}
We now show \eqref{beres2}.
Let $\phi(t) := \mu\left[e^{i\scalar t{\bar{\xi}}}\right].$ Define $\Psi:=\Psi_1 - \Psi_2,$ where $\Psi_1$ and $\Psi_2$ are independent and indentically distributed as $\bar{\xi}$, so that the characteristic function of $\Psi$ is $|\phi(t)|^2,$ and its covariance matrix is $2V_1.$ Next we use %the following equation
 \begin{align}\label{exprest}
R(x)=e^{ix}-\left(1+ix-\frac{x^2}{2}\right),\qquad |R(x)|\leq \min\left\{|x|^2,\frac{|x|^3}{6}\right\},
\end{align}
%where $$|R(x)|\leq \min\left\{|x|^2,\frac{|x|^3}{6}\right\}$$ 
for all $x \in \bbR.$ By substituting $x=\scalar{t}{\Psi}$ in \eqref{exprest} and taking the expectation, 
\begin{align}
|\phi(t)|^2 &= 1-\frac{1}{2}\mu\left[\left(\scalar{t}{\Psi}\right)^2\right] + \mu\left[R\left(\scalar{t}{\Psi}\right)\right] \label{charfunc1} \\ 
&\leq 1-\scalar{t}{V_1t} + \frac{1}{6}\mu\left[\left|\scalar{t}{\Psi}\right|^3\right].
\end{align}
Then, by using the fact that $(1+x)^N\leq e^{Nx}$ for all $x,$ that $V_N^{-1/2}$ is self-adjoint, and the Cauchy-Schwarz inequality, 
\begin{align}
|\phi_N(t)|^2 &=\left|\phi\left(V_N^{-1/2}t\right)\right|^{2N}  
\leq\left(1-\frac1N \scalar{t}{t} + \frac16 \mu\left[\left|\scalar{t}{V_N^{-1/2}\Psi}\right|^3\right]\right)^N\\
&\leq \left(1-\frac1N \scalar{t}{t}  + \frac{8}{6}|t|^3\mu\left[\left|V_N^{-1/2}\bar{\xi}\right|^3\right]\right)^N\\
&\leq \exp\left\{-\scalar{t}{t}+\frac{8}{6}L_N|t|^3\right\} \leq e^{-\frac{2}{3}\scalar{t}{t}},
\end{align}
thus \eqref{beres1} is proved if $ \frac{1}{4L_N}\geq|t|\geq\frac{1}{2}L_N^{-\frac{1}{3}}.$

Next, suppose that $|t|< \frac{1}{2}L_N^{-\frac{1}{3}}.$ Define $\tau_N := \frac{L_N}{N}.$ One has
\begin{align}\label{ineq}
\frac{1}{2}>L_N^{\frac{1}{3}}|t|>\tau_N^{\frac{1}{3}}|t|=\mu\left[\left|V_1^{-1/2}\bar{\xi}\right|^3\right]^{\frac{1}{3}}\frac{|t|}{\sqrt{N}}\geq  \mu\left[\left|V_1^{-1/2}\bar{\xi}\right|^2\right]^{\frac{1}{2}}\frac{|t|}{\sqrt{N}}\geq \frac{|t|}{\sqrt{N}},
\end{align}
where the last inequality follows from 
\begin{align}
\mu\left[\left|V_1^{-1/2}\bar{\xi}\right|^2\right]^{\frac{1}{2}} &= \left(\mu\left[\scalar{V_1^{-1/2}\bar{\xi}}{V_1^{-1/2}\bar{\xi}}\right]\right)^{\frac{1}{2}} = \sqrt{K}\geq 1.
%=\left( \mu\left[\scalar{\bar{\xi}}{V_1^{-1}\bar{\xi}}\right] \right)^{\frac{1}{2}}\\ &= \left(\sum_{i,x,i',x'}\mu\left[\bar{\xi}_{i,x}\bar{\xi}_{i',x'}\right]V^{-1}_1(i,x,i',x')\right)^{\frac{1}{2}} \\
%& =  \left(\sum_{i,x,i',x'}V_1(i,x,i',x')V^{-1}_1(i,x,i',x')\right)^{\frac{1}{2}}  = \sqrt{K}.
\end{align}
From \eqref{exprest} and the inequalities \eqref{ineq}, using again the Cauchy-Schwarz inequality we can write
\begin{align}\nonumber
 \left|\phi(V_N^{-1/2}t)-1\right| &= \left|-\frac1{2N}\scalar{t}{t} + \mu\left[R\left(\scalar{t}{V_N^{-1/2}\bar{\xi}}\right)\right]\right| \\ \label{beres3}
& \leq \frac{1}{2N}\scalar{t}{t} + \tau_N\frac{|t|^3}{6}  
< \frac{1}{8}+\frac{1}{48}<\frac{1}{4}. 
\end{align}
 For all $|z|<1$ the following inequality holds:
\begin{align}\label{log}
\left|\log(z+1)-z\right| \leq \frac{|z|^2}{2(1-|z|)}
\end{align}
Then, by using, in order, \eqref{log}, \eqref{beres3},  $|a+b|^2\leq2(|a|^2+|b|^2)$ and %the inequalities
 \eqref{ineq},
 \begin{align}
& \left|\log\phi(V_N^{-1/2}t)-\left(\phi(V_N^{-1/2}t)-1\right)\right| \leq \frac{\left|\phi(V_N^{-1/2}t)-1\right|^2}{2\left(1-\left|\phi(V_N^{-1/2}t)-1\right|\right)} <\frac{2}{3}\left|\phi(V_N^{-1/2}t)-1\right|^2 \\
 &\qquad =\frac{2}{3}\left|-\frac{1}{2N}\scalar{t}{t} +\mu\left[R\left(\scalar{t}{V_N^{-1/2}\bar{\xi}}\right)\right] \right|^2 \leq \frac{4}{3}\left(\frac{|t|^4}{4N^2}+ \frac{\tau_{N}^2}{36}|t|^6\right) \\
 &\qquad \leq \frac{4}{3}\left(\frac{\tau_N}{\sqrt{N}}\frac{|t|^4}{4} + \frac{\tau_{N}^2}{36}|t|^6\right) \leq \frac{4}{3}\left(\frac{1}{2\cdot4}+\frac{1}{36\cdot8}\right)\tau_N|t|^3 
 =\frac{37}{216}\tau_N|t|^3 < \frac{1}{5}\tau_N|t|^3.
 \end{align}
  By the triangular inequality we have
  \begin{align}
 &\left|\log\phi(V_N^{-1/2}t)+\frac{1}{2N}\scalar{t}{t}\right| \leq \left|\log\phi(V_N^{-1/2}t) - \left(\phi(V_N^{-1/2}t)-1\right)\right| +
 \\ & \qquad \quad +\left|\left(\phi(V_N^{-1/2}t)-1\right)+\frac{1}{2N}\scalar{t}{t}\right| \leq \frac{1}{5}\tau_N|t|^3 + \frac{1}{6}\tau_N|t|^3 \leq \frac{1}{2}\tau_N|t|^3.
 \end{align}
Using the inequality $|e^z-1|\leq |z|e^{|z|}$, $z \in \bbC$, and \eqref{ineq},
 \begin{align}
 \left|\phi_N(t)e^{\frac{|t|^2}{2}}-1\right| %& =\left|e^{\log \phi_N(t)e^{\frac{|t|^2}{2}}}-1\right| \\
 &\leq\left|\log \phi_N(t)e^{\frac{|t|^2}{2}}\right|\exp\left(\left|\log \phi_N(t)e^{\frac{|t|^2}{2}}\right|\right) \\
 &=N\left|\log \phi(V_N^{-1/2}t)+\frac{1}{2N}\scalar{t}{t}\right|\exp\left(
 N\left|\log \phi(V_N^{-1/2}t)+\frac{1}{2N}\scalar{t}{t}\right|
 \right) \\
 &\leq \frac{1}{2}L_N|t|^3e^{\frac{1}{2}L_N|t|^3} 
 \leq \frac{1}{2}L_N|t|^3e^{\frac{1}{16}}  <|t|^3L_N,
 \end{align}
 so that \eqref{beres1} is proved.
 \end{proof}

\bigskip

\begin{proof}[{\bf Proof of Proposition \ref{prop:LCLT}}]
Set $z_N:=\frac1{\sqrt N} \;V_1^{-1/2}\left(M_N - \mu^{\otimes N}(S_N)\right)$. Using the Fourier transform as in the proof of Theorem \ref{th:eqens} and the identity %using the following well known formula
\begin{align}
\int_{\bbR^K}e^{-i\scalar{s}{z_N}-\frac{\scalar{s}{s}}{2}}ds = (2\pi)^{\frac{K}{2}}e^{-\frac{\scalar{z_N}{z_N}}{2}},
\end{align}
 one has
\begin{align}
& \left |\mu^{\otimes N}\left(S_N = M_N\right)-\frac{e^{-\frac{1}{2}\scalar{z_N}{z_N}}}{(2\pi N)^{\frac{K}{2}}\sqrt{\det V_1}}\right|  \\
& = \frac{1}{B_N(2\pi)^K}\left | \int_{Q_{N,K}}e^{-i\scalar{s}{z_N}}\phi_N(s)ds-\int_{\bbR^K}e^{-i\scalar{s}{z_N}-\frac{1}{2}\scalar{s}{s}}ds\right| \label{summation} \\
&   \leq \frac{1}{B_N(2\pi)^K}\left(\int_{A_N}\left|\phi_N(s)-e^{-\frac{1}{2}\scalar{s}{s}}\right| ds + \int_{Q_{N,K} \setminus A_N}\left|\phi_N(s)\right|ds + \int_{A_N^c}e^{-\frac{1}{2}\scalar{s}{s}}ds\right),
\end{align}
where $B_N:=\sqrt{\det V_N}=N^{K/2}\sqrt {\det V_1}$,
$Q_{N,K}:=V_N^{1/2}[-\pi,\pi]^K$, and $A_{N} := \{s\in\bbR^K:\, |s|\leq \frac{1}{4L_N}\}.$  
%Note that $A_{N} \subset Q_{N,K}$ since 
%\begin{align}
%L_N > \frac{1}{\sqrt{N}}\,\mu\left[\left|V_1^{-1/2}\bar{\xi}\right|^2\right]^{\frac{3}{2}} = \frac{K^{\frac{3}{2}}}{\sqrt{N}} \geq \frac{1}{\sqrt{N}}.
%\end{align}
We are going to show that the three terms in the parenthesis above are bounded by $C/\sqrt N$. 

Let us define $$\tau:=\mu\left[\left|V_1^{-1/2}\bar{\xi}\right|^3\right].$$ For the first term in \eqref{summation} we use Lemma \ref{BerEs}:
\begin{align}
\int_{A_{N}}\left|\phi_N(s)-e^{-\frac{<s,s>}{2}}\right| ds &\leq 16L_N \int_{A_{N}}|s|^3 e^{-\frac{<s,s>}{3}}\\ 
&\leq16L_N \int_{\bbR^K}|s|^3 e^{-\frac{<s,s>}{3}} =  \frac{C_1}{\sqrt{N}},
\end{align}
where $C_1 :=16\,\tau\int_{\bbR^K}|s|^3 e^{-\frac{<s,s>}{3}}.$
For the second term in \eqref{summation} we use Lemma \ref{charfunc}:
\begin{align}\label{eq:t2eq}
\int_{Q_{N,K} \setminus A_{N}}\left|\phi_N(s)\right|ds &= B_N\int_{[-\pi,\pi)^K \setminus \cB_{\tau}}|\phi(t)|^N dt 
 \leq B_N(2\pi)^Ke^{-\frac{CN}{16\tau^2}},
\end{align}
where $C$ is defined as in that lemma and $\cB_{\tau}:= \left\{t\in\bbR^K:\,|t|\leq\frac{1}{4\tau}\right\}$. Therefore, \eqref{eq:t2eq} is bounded by $C_2/\sqrt N$ for a suitable constant $C_2$. 
%Note that 
%\begin{align}
%\sqrt{N^K}e^{-\frac{CN}{16\tau^2}} \leq \frac{1}{\sqrt{N}}\cdot\left(\frac{(K+1)16\tau^2}{2C}\right)^{\frac{K+1}{2}},
%\end{align}
%so that we finally have
%\begin{align}
%\sqrt{N^K}\sqrt{\det(V_1)} (2\pi)^Ke^{-\frac{CN}{16\tau^2}} \leq \frac{C_2}{\sqrt{N}},
%\end{align}
%where $C_2:=(\frac{(K+1)16\tau^2}{2C})^{\frac{K+1}{2}}\sqrt{\det(V_1)} (2\pi)^K.$
Finally, to bound the last term in \eqref{summation}, using $L_N^{-1}\geq c\sqrt N$ for some constant $c>0$, a simple estimate on the gaussian integral  shows that 
%since the integrand is symmetric with respect to all the variables we can write: 
\begin{align}
\int_{|s|>\frac{1}{4L_N}}e^{-\frac{<s,s>}{2}}ds & \leq  \frac{C_3}{\sqrt N},
%\int_{|s|_{\infty}>\frac{1}{4L_N\sqrt{K}}}e^{-\frac{<s,s>}{2}}ds 
%\leq 2K \int_{\frac{1}{4L_N\sqrt{K}}}^{+\infty}\int_{\bbR^{K-1}}e^{-\frac{<s,s>}{2}}ds \\
%&= 2K (\sqrt{2\pi})^{K-1}\int_{\frac{1}{4L_N\sqrt{K}}}^{+\infty}e^{-\frac{s^2}{2}}ds 
%\leq 2K (\sqrt{2\pi})^{K-1}4 L_N \sqrt{K}  e^{-\frac{1}{32 K L_N^2}} \leq  \frac{C_3}{\sqrt{N}},
\end{align}
for some constant $C_3$. %where $C_3 := 2K (\sqrt{2\pi})^{K-1}4 \sqrt{K} \tau.$
%Then the claim follows with a constant $C:=\max\{C_1,C_2,C_3\}$. 
\end{proof}

\end{document}